%% file: main.tex
\documentclass[educ,nonblindrev]{informs-educ}
\usepackage{eqndefns-left}
\RequirePackage{bm}
\RequirePackage{endnotes}

\OneAndAHalfSpacedXI

\usepackage[numbers,sort&compress]{natbib}
 \bibpunct[, ]{[}{]}{,}{n}{}{,}%
 \def\bibfont{\small}%

\EquationsNumberedThrough
\TheoremsNumberedThrough
\ECRepeatTheorems

\input{tex/macros}

\MANUSCRIPTNO{}

\begin{document}

 \RUNAUTHOR{Niazadeh and Udwani}
%\RUNAUTHOR{}
\RUNTITLE{Primal-Dual for Prior-Free Online Resource Allocation}
\TITLE{Modern Primal-Dual Frameworks for Prior-Free Online Resource Allocation}

\ARTICLEAUTHORS{%
\AUTHOR{Rad Niazadeh}
\AFF{University of Chicago Booth School of Business, Chicago, IL,
\EMAIL{rad.niazadeh@chicagobooth.edu}}
\AUTHOR{Rajan Udwani}
\AFF{Department of Industrial Engineering \& Operations Research, University of California, Berkeley,
\EMAIL{rudwani@berkeley.edu}}
}

\ABSTRACT{%
	Linear-programming (LP)-based primal--dual methods are fundamental for designing and analyzing algorithms in adversarial (prior-free) online resource allocation. This chapter provides a tutorial on two modern primal-dual frameworks, emphasizing recent developments and contemporary models in operations research. Part~I develops an LP-based convex-programming framework where solving a regularized convex program at each arrival captures the tradeoff between greediness and hedging, yielding a dual certificate via Karush--Kuhn--Tucker (KKT) conditions. Because standard LP relaxations can be weak or intractable for stochastic outcomes, Part~II introduces a complementary LP-free framework that provides a universal certificate system for evaluating competitive ratios under such uncertainty. Covering a wide array of models---including online vertex-weighted bipartite matching, edge-weighted online matching with free disposal, online matching with stochastic rewards, reusable resources, two-sided assortment optimization, configuration allocation (whole-page optimization), AdWords, and costly cancellations---the tutorial equips readers with versatile proof templates to analyze existing algorithms and develop new solutions for emerging applications.

}
\maketitle

\input{tex/introduction}

\input{tex/models}

\section{Part I: Convex-programming-based Online Allocations and Analysis}
\label{sec:part-I}
\input{tex/recipe}

\input{tex/matching}
\input{tex/free_disposal}

\section{Part II: LP-free Framework for Stochastic Outcomes and Beyond}
\label{sec:part-II}
\input{tex/part2}

\section{Open Problems and Discussion}\label{sec:open}
\input{tex/open}

\bibliographystyle{informs2014}

\bibliography{refs}

% ============================================================
% Electronic Companion
% ============================================================

\renewcommand{\theHchapter}{A\arabic{chapter}}
\renewcommand{\theHsection}{A\arabic{section}}

\newpage
\clearpage
%\SingleSpaced
\normalsize
\pagestyle{ECheadings}%
\ECHowEquations
\ECHowSections
\setcounter{figure}{0}%
\renewcommand\thefigure{EC.\@arabic\c@figure}%
\setcounter{table}{0}%
\renewcommand\thetable{EC.\@arabic\c@table}%
\setcounter{page}{1}\def\thepage{ec\arabic{page}}%

\ECDisclaimer

\input{tex/ec_intro}
\input{tex/batch_arrival}
\input{tex/configuration}
\input{tex/extensions}

\input{tex/additional_special_cases}

\input{tex/ec}

\end{document}

%% file: tex/macros.tex
\usepackage[T1]{fontenc}
\usepackage[utf8]{inputenc}
\usepackage{microtype}
\usepackage{amsmath,amssymb,mathtools}
\usepackage{bm}
\usepackage{booktabs}
\usepackage{tabularx}
\usepackage{enumitem}
\usepackage{xcolor}

\definecolor{cornellred}{rgb}{0.7, 0.11, 0.11}
\definecolor{maroon}{rgb}{0.52, 0, 0}
\definecolor{dgreen}{rgb}{0.0, 0.5, 0.0}
\definecolor{ballblue}{rgb}{0.13, 0.67, 0.8}
\definecolor{royalblue(web)}{rgb}{0.25, 0.41, 0.88}
\definecolor{bleudefrance}{rgb}{0.19, 0.55, 0.91}
\definecolor{royalazure}{rgb}{0.0, 0.22, 0.66}
\usepackage{hyperref}
\hypersetup{
	colorlinks = true,
	linkcolor=maroon,
	urlcolor=royalazure,
	citecolor=royalazure,
	linkbordercolor = {maroon}
}
\definecolor{optblue}{RGB}{50,100,200}
\definecolor{algred}{RGB}{200,60,60}
\definecolor{alggreen}{RGB}{40,160,70}
\definecolor{medgray}{RGB}{160,160,160}

\usepackage[capitalise,noabbrev]{cleveref}
\usepackage{tikz}
\usepackage{pgfplots}
\usepackage[most]{tcolorbox}
\usepackage{multirow}
\usetikzlibrary{arrows.meta,calc,decorations.pathreplacing,fit,patterns,positioning,shapes.misc}
\pgfplotsset{compat=1.18}
\allowdisplaybreaks

\AtBeginDocument{%
\abovedisplayskip=3pt plus 2pt minus 2pt
\belowdisplayskip=3pt plus 2pt minus 2pt
\abovedisplayshortskip=3pt plus 1pt
\belowdisplayshortskip=3pt plus 1pt minus 1pt
}

\definecolor{tutorialblue}{RGB}{30,84,156}
\definecolor{tutorialred}{RGB}{177,45,45}
\definecolor{tutorialgreen}{RGB}{48,118,82}
\definecolor{tutorialgray}{RGB}{90,90,90}
\definecolor{tutoriallight}{RGB}{248,248,248}

\makeatletter
\@ifundefined{argmax}{}{}
\@ifundefined{argmin}{}{}
\makeatother

\providecommand{\dd}{\,\mathrm{d}}
\providecommand{\Expo}{\mathrm{e}}
\providecommand{\ind}[1]{\mathbf{1}\!\left\{#1\right\}}

\providecommand{\RR}{\mathbb{R}}
\providecommand{\NN}{\mathbb{N}}
\providecommand{\EE}{\mathbb{E}}

\providecommand{\Graph}{G}
\providecommand{\OnlineVertices}{U}
\providecommand{\OfflineVertices}{V}
\providecommand{\EdgeSet}{E}
\providecommand{\Neighbor}[1]{N(#1)}

\providecommand{\ArrivalIdx}{t}

\providecommand{\OfflineIdx}{j}

\providecommand{\RewardVar}{r}
\providecommand{\GammaFunc}[1]{\Gamma\!\left(#1\right)}
\providecommand{\ALG}{\mathrm{ALG}}
\providecommand{\OPT}{\mathrm{OPT}}
\providecommand{\CR}{\mathrm{CR}}

\providecommand{\MatchVar}[2]{x_{#1#2}}

\providecommand{\LoadTime}[2]{y_{#1}^{(#2)}}
\providecommand{\VertexWeight}[1]{r_{#1}}

\providecommand{\RegFun}{F}
\providecommand{\RegDer}{f}
\providecommand{\RegVal}[1]{\RegFun\!\left(#1\right)}
\providecommand{\RegSlope}[1]{\RegDer\!\left(#1\right)}

\providecommand{\CompRatio}{\Gamma}

\providecommand{\DualOn}[1]{\alpha_{#1}}
\providecommand{\DualOff}[1]{\beta_{#1}}
\providecommand{\DualOnHat}[1]{\widehat{\alpha}_{#1}}
\providecommand{\DualOffHat}[1]{\widehat{\beta}_{#1}}
\providecommand{\LambdaVar}[1]{\lambda_{#1}}
\providecommand{\ThetaVar}[2]{\theta_{#1#2}}
\providecommand{\EtaVar}[2]{\eta_{#1}\!\left(#2\right)}
\providecommand{\EtaVarOpt}[2]{\eta^{\ast}_{#1}\!\left(#2\right)}
\providecommand{\DeltaVar}[2]{\delta_{#1}\!\left(#2\right)}
\providecommand{\DeltaVarOpt}[2]{\delta^{\ast}_{#1}\!\left(#2\right)}
\providecommand{\BetaInc}[1]{\Delta\widehat{\beta}_{#1}}
\providecommand{\BetaIncTime}[2]{\Delta\widehat{\beta}_{#1}^{(#2)}}

\providecommand{\MatchStar}[2]{x_{#1#2}^{\ast}}
\providecommand{\NewMassStar}[2]{z_{#1#2}^{\ast}}
\providecommand{\PostDensityStar}[2]{x_{#1}^{\ast}\!\left(#2\right)}

\providecommand{\EdgeReward}[2]{r_{#1#2}}
\providecommand{\NewMass}[2]{z_{#1#2}}
\providecommand{\PreDensity}[2]{y_{#1}\!\left(#2\right)}
\providecommand{\PostDensity}[2]{x_{#1}\!\left(#2\right)}
\providecommand{\PreCum}[2]{Y_{#1}\!\left(#2\right)}
\providecommand{\PostCum}[2]{X_{#1}\!\left(#2\right)}
\providecommand{\DensityTime}[3]{y_{#1}^{(#2)}\!\left(#3\right)}
\providecommand{\CumTime}[3]{Y_{#1}^{(#2)}\!\left(#3\right)}
\providecommand{\ThresholdVar}{q}
\providecommand{\ThresholdFunc}{G}
\providecommand{\ThresholdVal}[1]{\ThresholdFunc\!\left(#1\right)}

\providecommand{\PriceUniverse}{[T]}
\providecommand{\PriceLevel}[1]{\widehat{w}_{#1}}

\providecommand{\ConfigState}[3]{\eta_{#1,#2}^{(#3)}}

\providecommand{\Eqref}[1]{\eqref{#1}}

\providecommand{\Secref}[1]{Section~\ref{#1}}
\providecommand{\Thmref}[1]{Theorem~\ref{#1}}
\providecommand{\Propref}[1]{Proposition~\ref{#1}}
\providecommand{\Lemref}[1]{Lemma~\ref{#1}}
\providecommand{\Corref}[1]{Corollary~\ref{#1}}

\setlist[itemize]{topsep=2pt,itemsep=1pt,leftmargin=1.25em}
\setlist[enumerate]{topsep=2pt,itemsep=1pt,leftmargin=1.4em}
\tikzset{
  handaxis/.style={-{Latex[length=2.2mm]}, line width=0.85pt},
  handline/.style={line width=0.9pt, line cap=round},
  dashedguide/.style={dash pattern=on 2.2pt off 1.8pt, line width=0.7pt},
  tradebox/.style={rounded corners=2.5pt, draw=tutorialgray, fill=tutoriallight, inner sep=4pt},
  callout/.style={rounded corners=2pt, draw=tutorialgray, fill=white, inner sep=4pt},
  offlinebar/.style={draw=tutorialgray, fill=white, line width=0.85pt},
  shadeblue/.style={pattern=north east lines, pattern color=tutorialblue, draw=tutorialblue, line width=0.75pt},
  shadered/.style={pattern=north west lines, pattern color=tutorialred, draw=tutorialred, line width=0.75pt},
  shadegreen/.style={pattern=dots, pattern color=tutorialgreen, draw=tutorialgreen, line width=0.75pt},
  onlinenode/.style={circle, fill=black, inner sep=1.7pt},
  currentnode/.style={circle, fill=tutorialred, draw=tutorialred, inner sep=2.1pt},
  futurenode/.style={circle, fill=tutorialgray!65, draw=tutorialgray!65, inner sep=1.7pt},
  offlinenode/.style={circle, draw=black, fill=white, inner sep=2pt, line width=0.85pt},
  graphedge/.style={line width=0.85pt},
  figlabel/.style={font=\footnotesize},
  note/.style={font=\small}
}
\newcommand{\Rew}{R}
\newcommand{\E}{\mathbb{E}}
\newcommand{\one}{\mathbf{1}}

\renewcommand{\qed}{$\hfill\square$}

\newcommand{\xhdr}[1]{\noindent\textbf{#1}}

%% file: tex/introduction.tex
\section{Introduction}
% {\color{blue} Requests for Rad:
% \begin{itemize}
%     \item Changing citation format - Tutorials requires a numeric style. I am hoping this will be easier for you with Prism but please let me know if I can help.
%     \item A few pages over limit would be okay but unfortunately we are 7 pages over 30, which may be an issue. Can you look at possible compression in Section 3 (currently longest at 14 pages)?
%     \item I did not address any reviewer comments' specific to Part I except for a minor modification to Figure 1, where I removed the term ``competitive ratio" from both boxes.
% \end{itemize}}
\label{sec:introduction}
Every day, billions of allocation decisions are made in real time with incomplete information, whenever a platform must match a stream of requests to scarce resources before the future is known. A cloud platform assigns computing resources to incoming jobs without knowing what workloads
will arrive next.  A ride-sharing service matches drivers to passengers as requests stream in,
unable to foresee demand surges an hour later.  An ad exchange assigns inventories of arriving impressions to advertisers who are willing to pay for them, without knowing the future impressions.  The defining feature of these problems is that postponing an allocation is either impossible or itself costly. Therefore, decisions are effectively irrevocable---once a resource is assigned or a request is
rejected, the choice cannot be undone---yet future demand remains unknown.

The above paradigm is known as the \emph{online resource allocation} problem, and it has emerged as one of the
central topics at the interface of operations research and computer science.
Over three decades of research---from the seminal work of \citet{karp1990optimal} on online
bipartite matching, through the foundational contributions of
\citet{mehta2007adwords} and \citet{buchbinder2007online} on online budgeted allocation in search advertising, to
the recent wave of models involving stochastic outcomes~\citep{mehta2012stochastic,chan2009stochastic}, assortment planning~\citep{golrezaei2014real,aouad2023assortment}, reusable
resources~\citep{gong2022reusable,goyal2025reusable,feng2021robustness}, batching~\citep{feng2024twostage,feng2024batching}, multichannel traffic~\citep{manshadi2024multichannel}, and more---the community has developed a rich and powerful theory that mostly shares a common analytical backbone for designing and analyzing online algorithms: the \emph{primal--dual framework}.

This tutorial presents two modern versions of this framework for online resource allocation under adversarial arrivals\footnote{A prior-free setting in which no assumptions are made about how demand requests arrive over time.}. Although the classic primal-dual framework is conceptually simple and has proved remarkably effective, its application is often problem-specific: the design principles that lead to optimal competitive algorithms, and the corresponding choices that make the analysis go through, are not yet fully systematized. In particular, many proofs proceed by constructing a tailored LP relaxation and then fitting dual variables, a process that typically must be redesigned for each new model variant.

Our goal is to distill this practice into a \emph{reusable methodology} by articulating a small set of proof templates and design principles that can be applied across settings. The tutorial develops two complementary templates. Before introducing them, we briefly describe the common structure underlying classic primal--dual analyses.
% This tutorial presents a unified, modern treatment of the primal--dual framework for online
% resource allocation under adversarial arrival\footnote{This setting is essentially the prior-free setting, where we make no assumption on how demand requests arrive over time}. This analytical framework, while elegant and extremely successful, is still somewhat mysterious. Despite being conceptually simple and intuitive, the design principles behind optimal competitive primal-dual-based online algorithms, and the right setup to analyze them, are not fully understood. Instances of the approach often involve designing a problem-specific LP relaxation and fitting dual variables, which must be re-tailored for each new setting. Here, we take a stab at resolving this mystery by distilling a
% \emph{reusable methodology}---a set of recipes and systematic design principles that equip the reader to
% both analyze known algorithms and design new ones.\rajan{revisit this statement - do we give recipes to design new algorithms?} In fact, this tutorial develops two complementary sets of such recipes. Before introducing them, we need to explain the basic idea behind almost all primal-dual analyses, including ours.
\subsection{Offline primal/dual LPs and dual fitting}
\label{subsec:intro-dual-fitting}

The standard benchmark for evaluating an online algorithm is the \emph{offline optimum}: the maximum total value achievable if the full arrival sequence were known in advance. In many classical resource-allocation models, this benchmark admits a natural linear-programming (LP) relaxation. The central analytical question is then the following: how can one certify that,  for every instance, an online algorithm achieves at least a certain fraction of the value of this offline LP? This fraction is the algorithm's \emph{competitive ratio}.

The standard route is to develop an online dual-update rule that constructs a feasible or approximately feasible \emph{dual certificate} in tandem with the online primal decisions, and then invoke weak duality to compare the online algorithm with the offline benchmark. To formalize this paradigm, we write the offline benchmark as the packing LP
\begin{equation}
\label{eq:intro-generic-primal}
\max\{c^\top x: Ax\le b,\ x\ge 0\},
\end{equation}
and let its dual be
\begin{equation}
\label{eq:intro-generic-dual}
\min\{b^\top y: A^\top y\ge c,\ y\ge 0\}.
\end{equation}
Suppose the online algorithm produces a feasible primal solution with value \(\ALG\).
To prove a competitive ratio \(\CompRatio\in(0,1)\), it suffices to construct a nonnegative vector
\(\widehat y\) such that
\begin{equation*}
\label{eq:intro-dual-fitting-target}
    b^\top \widehat y \le \ALG,
    \qquad
    A^\top \widehat y \ge \CompRatio c.
\end{equation*}
The second inequality implies that \(\widehat y/\CompRatio\) is feasible for the dual LP
\eqref{eq:intro-generic-dual}.  Hence, by weak duality,
\[
    \OPT
    \le b^\top(\widehat y/\CompRatio)
    \le \ALG/\CompRatio,
\]
and so \(\ALG\ge \CompRatio\, \OPT\). This trick is the standard \emph{dual-fitting} paradigm in approximation algorithms.

The weak-duality step itself is straightforward; the main challenge is to construct \(\widehat y\) in a systematic way. In many classical analyses, the dual-update rules used to build certificates are tailored to the specific model and can appear ad hoc. These difficulties become more pronounced in settings with stochastic outcomes: the choice of an offline benchmark is more delicate, LP relaxations can be loose (yielding weak competitive-ratio bounds), and in models with richer stochastic dynamics it may be unclear how to formulate a useful LP relaxation in the first place.

A closely related question is how to design online algorithms that admit strong dual certificates. In adversarial models, algorithm design is often guided by a tension between two objectives. A \emph{greedy} rule attempts to extract as much value as possible from the current arrival, whereas a \emph{hedging} rule preserves future flexibility by avoiding premature overuse of scarce resources. Classical algorithms for fractional allocation---such as BALANCE for (vertex-weighted) online bipartite matching~\citep{kalyanasundaram2000optimal} and its extension to online budgeted allocation (AdWords)~\citep{mehta2007adwords}---implement this tradeoff through scoring rules that are typically model-specific.

These considerations motivate the development of proof templates and design principles that systematize algorithm design and dual fitting across models, and that remain useful even when LP relaxations are weak or difficult to formulate.
\subsection{A two-part tutorial}
\label{subsec:intro-two-part}
This tutorial is organized into two parts around two complementary primal--dual
reasonings (\Cref{fig:two-parts}):
\begin{itemize}
    \item \textbf{Part~I} (\Cref{sec:part-I}) studies a \emph{convex-programming} viewpoint for the design and analysis of algorithms based on techniques introduced in \citet{feng2024twostage, feng2024batching}. In short, we introduce a new interpretation of several existing algorithms in the literature, where algorithms now solve the arrival-wise linear programs of the greedy rule, regularized by a convex function (hence convex programs), at each arrival instead of following an ad hoc allocation rule. The role of convex regularization is to encode the trade-off between greediness and hedging directly in the arrival-wise optimization problem. Given such convex-programming-based online allocations, we show that the dual certificate should be constructed (online) using the Karush--Kuhn--Tucker (KKT) conditions of the same arrival-wise convex program that defines the algorithm. Thus the dual-fitting certificate is not guessed after the algorithm is written down. Rather, the regularized convex program simultaneously specifies the primal decision rule, the state variable that the algorithm tracks, and the shadow prices that enter the dual construction.
\item \textbf{Part~II} (\Cref{sec:part-II}) focuses on models in which allocation decisions trigger stochastic outcomes. In these settings, the natural offline benchmark is the optimal solution to a high-dimensional stochastic dynamic program, making direct comparisons between an online policy and the offline optimum analytically challenging. A standard approach is to use an LP relaxation of the offline benchmark and fit a dual certificate, but such relaxations can be loose, and for richer stochastic dynamics it may be unclear how to formulate a useful LP in the first place. To bypass these limitations, we present the \emph{LP-free framework} developed by \citet{goyal2023stochastic, goyal2025reusable}, which replaces traditional edge-level dual constraints with a resource-level covering condition. This approach yields a modular proof structure that translates seamlessly across problem variants and naturally incorporates sample-path couplings. We illustrate the framework's versatility by analyzing the Greedy algorithm across three canonical settings---stochastic rewards, patience (where arrivals probe multiple resources sequentially), and stochastic reusability (where consumed resources become unavailable for a random duration before returning to the system)---alongside their extensions.
    
\end{itemize}

\begin{figure}[t]
    \centering
    \begin{tikzpicture}[x=1cm,y=1cm]
        \node[tradebox, minimum width=5.8cm, minimum height=1.6cm, align=left,
              font=\small] (left) at (-5.2,0) {%
            \textbf{\color{tutorialblue}Part~I: Convex-Programming Approach}\\[2pt]
            Offline LP $\to$ regularized convex program $\to$\\
            KKT dual certificate };
        \node[tradebox, minimum width=5.8cm, minimum height=1.6cm, align=left,
              font=\small] (right) at (5.2,0) {%
            \textbf{\color{tutorialred}Part~II: LP-Free Approach}\\[2pt]
            Stochastic outcomes (Weak LP) $\to$ \\ LP-free certificate $\to$
            generalized weak duality};
        \draw[handline, {Latex[length=2.4mm]}-{Latex[length=2.4mm]}, tutorialgray]
            (-1.65,0) -- (1.65,0);
        \node[note, text=tutorialgray, above] at (0,0.15) {\textbf\emph{Primal--Dual Reasoning}};
    \end{tikzpicture}
    \caption{The two complementary viewpoints developed in this tutorial.  Part~I builds dual
    certificates from KKT conditions of regularized convex programs.  Part~II constructs LP-free
    certificates for settings with stochastic outcomes and reusable resources.}
    \label{fig:two-parts}
    \vspace{-5mm}
\end{figure}
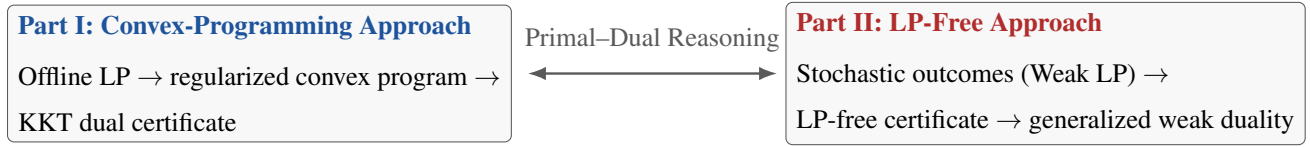

\begin{table}[t]
\centering
\small
\begin{tabularx}{0.98\linewidth}{>{\raggedright\arraybackslash}p{0.11\linewidth} >{\raggedright\arraybackslash}p{0.35\linewidth} X}
\toprule
Part & Methodological viewpoint & Representative model classes \\
\midrule
Part~I (\Cref{sec:part-I})& Offline LP benchmark, arrival-wise regularized convex program, KKT-based dual construction, and competitive analysis by dual fitting. &
Online fractional matching and vertex-weighted matching
\citep{karp1990optimal,kalyanasundaram2000optimal,aggarwal2011vertex,devanur2013randomized};
edge-weighted matching with free disposal \citep{feldman2009free};
batch arrival \citep{feng2024batching};
configuration allocation / whole-page optimization \citep{devanur2016wholepage};
AdWords \citep{mehta2007adwords,buchbinder2007online};
and costly cancellations \citep{ekbatani2022cancellations}. \\
\addlinespace[3pt]
Part~II (\Cref{sec:part-II})& LP-free certificates that compare the online policy directly against the benchmark bypassing the need for an offline LP relaxation. &
Stochastic rewards \citep{mehta2012stochastic,mehta2014stochastic,goyal2023stochastic, huang2024online2};
reusable resources \citep{gong2022reusable,goyal2025reusable,feng2021robustness,delong2024online}; stochastic rewards with patience \citep{brubach2025online, borodin2022prophet};
multichannel traffic \citep{manshadi2024multichannel};
AdWords with unknown budgets \citep{udwani2025adwords}; stochastic allocation and pricing models such as online one-sided~\citep{golrezaei2014real} and two-sided assortment \citep{aouad2023assortment} and pricing with show-all constraint~\citep{goyal2022pricing}. \\
\bottomrule
\end{tabularx}
\smallskip
\caption{The two methodological lenses developed in the tutorial.}
\vspace{-2mm}
\label{tab:two-part-roadmap}
\end{table}

There are already excellent references on online algorithms, online matching, and primal-dual techniques; see, for example,
\citet{borodin2005online} and \citet{buchbinder2009design} for general treatments of online algorithms and primal-dual method, the survey of \citet{mehta2013survey} for classical results, and the more recent survey of \citet{huang2024online} for recent developments in the literature. The emphasis of the present tutorial is different.
The objective is methodological: to organize a broad collection of online-allocation analyses around
the above two design principles. Taken together, we hope the two parts provide a unified language for various models in online resource allocation under adversarial arrival. In what comes next---\Cref{sec:part-I} and \Cref{sec:part-II}---we demonstrate our two recipes by applying them on canonical models. Section~\ref{sec:open} concludes with open problems and additional discussions. We postpone more complicated extension models to the electronic companion (EC). We have summarized the literature related to both the main body and EC materials in \Cref{tab:two-part-roadmap}.

%% file: tex/models.tex
\section{Base model and notation}
\label{sec:models}
This section formalizes a basic adversarial online-allocation setting and the common notation used in the main body. Formal models for extensions appear later in \Cref{sec:free-disposal,sec:patience,sec:reusable} and in the electronic companion. \Cref{tab:notation-main} lists notation used repeatedly across the tutorial; model-specific notation is introduced locally, with extension-model symbols recalled again in the electronic companion. 

% See \Cref{tab:ec-roadmap} for a list of these extensions. 
\subsection{Adversarial arrival and competitive analysis}
\label{subsec:model-adversarial}

At a high level, an adversarial online-allocation \emph{instance} consists of a set of offline resources, a sequence of online requests, and model-specific data (e.g., feasibility constraints, rewards, outcome distributions, and capacity-related parameters). The resource set and its associated data (e.g., capacities) are known to the algorithm at the outset. In the adversarial model, an adversary selects both the underlying instance and the order in which requests are revealed, subject only to the rules of the model. Requests then arrive sequentially at times \(\ArrivalIdx=1,2,\ldots\); upon the arrival of request \(\ArrivalIdx\), the algorithm observes the data revealed with that request and must irrevocably choose its allocation for that request without seeing any future requests. The model may also include stochastic outcomes---for example, stochastic rewards or stochastic usage durations in the case of reusable resources.

Performance is benchmarked against an offline optimum \(\OPT\). In deterministic models, \(\OPT\) observes the entire request sequence in advance and solves the corresponding offline optimization problem to optimality. In models with stochastic outcomes, we compare against a non-anticipative offline benchmark that knows the full instance but does not observe realizations of post-allocation random variables in advance. This benchmark can be formulated as a (generally exponential-sized) stochastic dynamic program that maximizes the expected objective value. We return to the precise definition of \(\OPT\) in models with stochastic outcomes later.

% \begin{definition}[Non-anticipative offline benchmark]\label{def:opt}
% The offline benchmark $\OPT$ knows the entire graph $G$ and all probabilities $\{p_{jt}\}$ in advance, but it is non-anticipative: it does not observe the outcome of an attempted match before deciding whether to attempt it. It may process arrivals in any order and may condition future actions on outcomes already observed. Equivalently, $\OPT$ can be formulated as a stochastic dynamic program whose decisions are adapted to realized outcomes.
% \end{definition}

To compare a (possibly randomized) online algorithm \(\ALG\) with \(\OPT\), we evaluate the \emph{worst-case competitive ratio} (or simply the competitive ratio), defined as
\begin{equation}
\label{eq:competitive-ratio-definition}
    \CR(\ALG)
    :=
    \inf_{\mathcal I}
    \frac{\EE[\Rew^{\ALG(\mathcal I)}]}{\EE[\Rew^{\OPT(\mathcal I)}]},
\end{equation}
where $\E[\Rew^{\ALG(\mathcal I)}]$ (resp.\ $\E[\Rew^{\OPT(\mathcal I)}]$) denotes the expected objective value (i.e., \emph{reward}) of the online algorithm (resp.\ the offline benchmark) on instance \(\mathcal I\), the infimum is taken over all instances \(\mathcal I\), and the expectation is over the internal randomness of the algorithm (and any additional randomness in the model, when present). A lower bound of the form \(\CR(\ALG)\ge \Gamma\) means that the algorithm achieves at least a \(\Gamma\)-fraction of the offline optimum on every adversarial instance.

In Part~I, we focus on \emph{fractional} algorithms, where allocations may take any value in $[0,1]$.\footnote{All results can be extended to randomized integral allocations under ``large-budget'' assumptions by using standard randomized rounding algorithms and concentration inequalities; see \citep{feng2024batching} for an example.} In Part~II, we focus on the integral allocations (and we only study ``Greedy'' algorithm).

\begin{table}[t]
\centering
\small
\begin{tabularx}{0.9\linewidth}{>{\raggedright\arraybackslash}p{0.27\linewidth} X}
\toprule
Symbol & Meaning in the main body \\
\midrule
$\OnlineVertices,\OfflineVertices,\EdgeSet$ & online requests, offline resources, feasible incidences \\
$\ArrivalIdx$ & index of an online arrival \\
$\OfflineIdx$ & index of an offline resource \\
$\Neighbor{\ArrivalIdx}$ & feasible offline neighborhood of arrival $\ArrivalIdx$ \\
$\MatchVar{\ArrivalIdx}{\OfflineIdx}$ & fractional mass assigned from arrival $\ArrivalIdx$ to resource $\OfflineIdx$ \\
$\VertexWeight{\OfflineIdx}$ & reward/weight of offline resource $\OfflineIdx$ in vertex-weighted models \\
$\EdgeReward{\ArrivalIdx}{\OfflineIdx}$ & reward of edge $(\ArrivalIdx,\OfflineIdx)$ in edge-weighted models \\
$\LoadTime{\OfflineIdx}{\ArrivalIdx}$ (equivalently $y_j$) & cumulative load of resource $\OfflineIdx$ after arrival $\ArrivalIdx$ \\
$\DualOn{\ArrivalIdx},\DualOff{\OfflineIdx}$ & online-side and offline-side dual variables in LP-based analyses \\
$\alpha_{tj},\DualOff{\OfflineIdx}$ & LP-free certificate variables in Part~II \\
$\RegFun,\RegDer$ & regularizer and its derivative \\
$\ALG,\OPT,\CR$ & online algorithm, offline benchmark, competitive ratio \\
$\Rew^{\mathcal A},\Rew_j^{\mathcal A}$ & total sample path reward under algorithm $\mathcal{A}$ and sample path reward attributable to resource $j$ under algorithm $\mathcal{A}$ (see Remark \ref{obs:det})\\
$U_j^{\mathcal A},V_t^{\mathcal A}$ & set of arrivals that algorithm $\mathcal{A}$ attempts to match to resource $j$ and set of resources available to algorithm $\mathcal{A}$ when $t$ arrives\\
$V^{\mathcal{A}}_{T+1}$ & set of resources available to algorithm $\mathcal{A}$ at the end of the horizon \\
\bottomrule
\end{tabularx}
\smallskip
\caption{Common notation used repeatedly in the main body. Additional notation for particular models is introduced locally in the relevant sections.}
\label{tab:notation-main}
\vspace{-3mm}
\end{table}

\subsection{Base model: vertex-weighted online bipartite matching with stochastic rewards}
\label{subsec:model-base}
\label{subsec:model-vertexweighted}

As a warm-up model, consider (unweighted, deterministic) online bipartite matching. In this model,
\(\Graph=(\OnlineVertices\cup\OfflineVertices,\EdgeSet)\) is a bipartite graph with offline
set \(\OfflineVertices\) and online set \(\OnlineVertices\), where the offline side represents resources and the online side represents requests arriving over time. By convention, the online set is indexed in arrival order,
\(\OnlineVertices=\{1,2,\ldots,T\}\).
When arrival \(\ArrivalIdx\) occurs, its neighborhood
\(\Neighbor{\ArrivalIdx}:=\{\OfflineIdx\in\OfflineVertices:(\ArrivalIdx,\OfflineIdx)\in\EdgeSet\}\)
is revealed. The algorithm chooses nonnegative binary variables
\(\MatchVar{\ArrivalIdx}{\OfflineIdx}\in\{0,1\}\) for
\(\OfflineIdx\in\Neighbor{\ArrivalIdx}\), and we set
\(\MatchVar{\ArrivalIdx}{\OfflineIdx}=0\) for
\(\OfflineIdx\notin\Neighbor{\ArrivalIdx}\). In the fractional setting, \(\MatchVar{\ArrivalIdx}{\OfflineIdx}\in[0,1]\). The decision must satisfy
\begin{equation*}
\label{eq:model-unweighted-online-constraint}
    \sum_{\OfflineIdx\in\Neighbor{\ArrivalIdx}} \MatchVar{\ArrivalIdx}{\OfflineIdx} \le 1,
\end{equation*}
so that each arrival \(\ArrivalIdx\) is assigned to at most one offline vertex. We normalize offline resource capacities to be one,\footnote{This normalization is without loss because one can replace each offline vertex with integer capacity by that many identical unit-capacity copies for our analysis of integral algorithms. In the case of fractional allocations, this is again without loss of generality, as one can reduce the problem with arbitrary capacities to unit capacity by normalization and change of variables.} and therefore we have
\begin{equation*}
\label{eq:model-unweighted-offline-constraint}
    \sum_{\ArrivalIdx:(\ArrivalIdx,\OfflineIdx)\in\EdgeSet}
    \MatchVar{\ArrivalIdx}{\OfflineIdx}
    \le 1
    \qquad
    \forall \OfflineIdx\in\OfflineVertices.
\end{equation*}
For the online analysis it is convenient to define the cumulative load on resource
\(\OfflineIdx\) after arrival \(\ArrivalIdx\) by
\begin{equation}
\label{eq:model-load-definition}
    \LoadTime{\OfflineIdx}{\ArrivalIdx}
    :=
    \sum_{s\le \ArrivalIdx:(s,\OfflineIdx)\in\EdgeSet}
    \MatchVar{s}{\OfflineIdx},
    \qquad
    \LoadTime{\OfflineIdx}{0}:=0.
\end{equation}
In the vertex-weighted model, each offline node \(\OfflineIdx\in\OfflineVertices\) has a weight (or reward) \(\VertexWeight{\OfflineIdx}\geq 0\), and the objective is $\sum_{(\ArrivalIdx,\OfflineIdx)\in\EdgeSet}
    \VertexWeight{\OfflineIdx}\,\MatchVar{\ArrivalIdx}{\OfflineIdx}$. In the unweighted special case, the objective is the size of (fractional or integral) matching, that is, $\sum_{(\ArrivalIdx,\OfflineIdx)\in\EdgeSet}
    \MatchVar{\ArrivalIdx}{\OfflineIdx}$. See \Cref{fig:base-model} for an illustration.

\vspace{2mm}

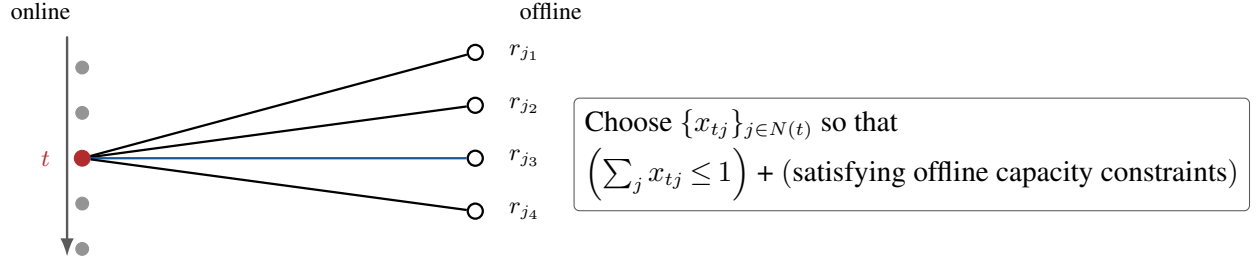
\begin{figure}[t]
    \centering
    \begin{tikzpicture}[x=1cm,y=1cm]
        \node[figlabel] at (-0.55,3.45) {online};
        \draw[handline, tutorialgray, -{Latex[length=2.4mm]}] (-0.2,3.1) -- (-0.2,0.2);
        \foreach \y in {2.7,2.1,1.5,0.9,0.3} {
            \node[futurenode] at (0,\y) {};
        }
        \node[currentnode] (tcur) at (0,1.5) {};
        \node[figlabel, text=tutorialred, left=0.18cm of tcur] {$\ArrivalIdx$};
        \node[figlabel] at (6.2,3.45) {offline};
        \foreach \idx/\y/\lab in {1/2.9/{r_{j_1}},2/2.2/{r_{j_2}},3/1.5/{r_{j_3}},4/0.8/{r_{j_4}}} {
            \node[offlinenode] (v\idx) at (5.2,\y) {};
            \node[figlabel, right=0.2cm of v\idx] {$\lab$};
        }
        \draw[graphedge] (tcur) -- (v1);
        \draw[graphedge] (tcur) -- (v2);
        \draw[graphedge, tutorialblue] (tcur) -- (v3);
        \draw[graphedge] (tcur) -- (v4);
        \node[callout, align=left, anchor=west] at (6.5,1.55)
        {Choose \(\{\MatchVar{\ArrivalIdx}{\OfflineIdx}\}_{\OfflineIdx\in\Neighbor{\ArrivalIdx}}\) so that\\
         \(\left(\sum_{\OfflineIdx}\MatchVar{\ArrivalIdx}{\OfflineIdx}\le 1\right)\) + $\left(\textrm{satisfying offline capacity constraints}\right)$};
    \end{tikzpicture}
    \caption{Vertex-weighted online (fractional or integral) matching.  The current arrival $t$ reveals its
    neighborhood and the algorithm immediately chooses a feasible allocation for this arrival.}
    \label{fig:base-model}
    \vspace{-6mm}
\end{figure}

We next describe the stochastic-rewards generalization (relevant to \Cref{sec:part-II}). Each edge \((\ArrivalIdx,\OfflineIdx)\in\OnlineVertices\times\OfflineVertices\) is associated with a success probability \(p_{\ArrivalIdx\OfflineIdx}\in[0,1]\).\footnote{In the stochastic-rewards model, we may assume a complete graph without loss of generality by setting \(p_{\ArrivalIdx\OfflineIdx} = 0\) for absent edges.} When arrival \(\ArrivalIdx\) occurs, the algorithm observes its neighborhood \(\Neighbor{\ArrivalIdx}\) and the corresponding probabilities \(\{p_{\ArrivalIdx\OfflineIdx}\}_{\OfflineIdx\in\Neighbor{\ArrivalIdx}}\), and then selects at most one available resource \(\OfflineIdx\in\Neighbor{\ArrivalIdx}\) to match with \(\ArrivalIdx\). This match succeeds with probability \(p_{\ArrivalIdx\OfflineIdx}\) and fails otherwise, independently across all arrivals. 

The rewards are stochastic in the following sense: a successful match  yields a reward \(r_{\OfflineIdx}\) and consumes the resource, whereas an unsuccessful match yields no reward and leaves the resource available for subsequent arrivals.\footnote{Equivalently, each offline vertex can be successfully matched multiple times but yields its reward at most once, upon its first successful match. The algorithms considered in this tutorial never attempt to match an unavailable resource, though this need not hold for policies that are oblivious to stochastic realizations \citep[e.g.,][]{udwani2025optimality}.} The goal is to maximize the expected total reward. The deterministic unweighted warm-up corresponds to the special case of \(r_{\OfflineIdx} = 1\) and \(p_{\ArrivalIdx\OfflineIdx}\in\{0,1\}\) for all edges.

\begin{definition}[Non-anticipative offline benchmark]\label{def:opt}
	The offline benchmark $\OPT$ knows all probabilities $\{p_{tj}\}$ in advance, but it is non-anticipative: it does not observe the outcome of an attempted match before deciding whether to attempt it. It may process arrivals in any order and may condition future actions on outcomes already observed. Equivalently, $\OPT$ can be formulated as a stochastic dynamic program whose decisions are adapted to realized outcomes.
\end{definition}

Let $\mathcal{A}\in\{\ALG,\OPT\}$ denote a generic non-anticipative policy. Along a realized sample path, let $U^{\mathcal{A}}_j$ be the set of arrivals for which $\mathcal{A}$ attempts a match with resource $j$ while $j$ is available, and let $V^{\mathcal{A}}_t$ and $V^{\mathcal{A}}_{T+1}$ denote the set of available resources at arrival $t$ and at the horizon's end, respectively. The following algebraic equivalence will prove central to our subsequent analyses.

\begin{remark}[Deterministic Reward \& Revenue Decomposition]\label{obs:det}
Fix an online or offline non-anticipative algorithm $\mathcal{A}\in\{\ALG,\OPT\}$. If the algorithm selects edge $(t,j)$ while $j$ is available, then the random success reward can be replaced by a guaranteed (deterministic) reward $p_{tj}r_j$ without changing the expected total reward. With this substitution the only remaining randomness is in whether the resource is consumed (i.e., becomes unavailable due to a successful match). Let $\Rew^{\mathcal{A}}$ denote the resulting sample-path deterministic reward, and let $\Rew^{\mathcal{A}}_j$ denote the deterministic reward attributable to resource $j$, so that $\Rew^{\mathcal{A}}=\sum_j \Rew^{\mathcal{A}}_j$. Then
\begin{equation}\label{eq:rev-decomp}
	\E[\Rew^{\mathcal{A}}_j]
	= \Pr\!\bigl[j \notin V^{\mathcal{A}}_{T+1}\bigr]\,r_j\,
	= \E\!\big[\sum_{t\in U^{\mathcal{A}}_j} p_{tj}\big]r_j.
\end{equation}
\end{remark}
%Given this simple model, the optimum offline can be formulated as a simple LP. This LP benchmark and its dual are
\subsubsection{Classical LP relaxation}
A standard relaxation is the following fluid LP, in which assigning $t$ to $j$ deterministically consumes $p_{tj}$ units of capacity from resource $j$:

\begin{minipage}[t]{0.49\linewidth}
\textbf{Primal LP}
\begin{equation}
\begin{aligned}[t]
    \underset{\mathbf{x}\geq0}{\max} \quad
    & \sum_{(\ArrivalIdx,\OfflineIdx)\in\EdgeSet}
      \VertexWeight{\OfflineIdx}\,p_{tj}\,\MatchVar{\ArrivalIdx}{\OfflineIdx} \\
    \text{s.t.}\quad
    & \sum_{\OfflineIdx:(\ArrivalIdx,\OfflineIdx)\in\EdgeSet}
      \MatchVar{\ArrivalIdx}{\OfflineIdx} \le 1
      && \forall \ArrivalIdx\in \OnlineVertices, \\
    & \sum_{\ArrivalIdx:(\ArrivalIdx,\OfflineIdx)\in\EdgeSet}
      p_{tj}\,\MatchVar{\ArrivalIdx}{\OfflineIdx} \le 1
      && \forall \OfflineIdx\in \OfflineVertices, \\
    & \MatchVar{\ArrivalIdx}{\OfflineIdx}\ge 0
      && \forall (\ArrivalIdx,\OfflineIdx)\in\EdgeSet.
\end{aligned}
\label{eq:offline-primal-basic}\label{eq:offline-primal-basic-unweighted}\label{eq:primal-lp}
\end{equation}
\end{minipage}\hfill
\begin{minipage}[t]{0.49\linewidth}
\textbf{Dual LP}
\begin{equation}
\begin{aligned}[t]
    \min \quad
    & \sum_{\ArrivalIdx\in\OnlineVertices} \DualOn{\ArrivalIdx}
      + \sum_{\OfflineIdx\in\OfflineVertices}\DualOff{\OfflineIdx} \\
    \text{s.t.}\quad
    & \DualOn{\ArrivalIdx}+p_{tj}\,\DualOff{\OfflineIdx} \ge p_{tj}\, \VertexWeight{\OfflineIdx}
      && \forall (\ArrivalIdx,\OfflineIdx)\in\EdgeSet, \\
    & \DualOn{\ArrivalIdx},\DualOff{\OfflineIdx} \ge 0
      && \forall \ArrivalIdx\in \OnlineVertices,\ \forall \OfflineIdx\in\OfflineVertices.
\end{aligned}
\label{eq:offline-dual-basic}\label{eq:offline-dual-basic-unweighted}\label{eq:dual-lp}
\end{equation}
\end{minipage}
\smallskip

One can show that $\OPT$ induces a feasible primal solution with objective value $\E[\Rew^{\OPT}]$, and therefore $\text{Primal-LP}\ge \E[\Rew^{\OPT}]$. Hence, any dual-feasible solution can be used to certify a competitive ratio via weak duality. In the deterministic setting (i.e., $p_{tj}\in\{0,1\}$), the relaxation is tight and captures the offline optimum. In contrast, in the stochastic-rewards model the relaxation can substantially overestimate the true offline value, as illustrated next.

\begin{remark}[{Looseness of the LP}]\label{rem:loose}
	The relaxation in \eqref{eq:primal-lp} can be loose when rewards are stochastic. Consider a single resource $j$ and $1/p$ arrivals, each adjacent only to $j$ with success probability $p$. The LP can attain value $1$ by fractionally matching all arrivals to $j$ so that $\sum_t p x_{tj}=1$. In contrast, any feasible policy can only attempt matches sequentially and the probability of at least one success equals $1-(1-p)^{1/p}\overset{p\to 0}{\longrightarrow } 1-1/e$. 
\end{remark}

In models with richer stochastic dynamics, identifying an effective LP relaxation can be challenging, and even a natural relaxation may exhibit nontrivial gaps. This motivates the division pursued in the remainder of the tutorial. In Part~I we focus on deterministic settings (and fractional allocations), where LP-based primal--dual analyses, combined with a convex-programming viewpoint, can be used to design algorithms with optimal competitive ratios across a wide range of models. In Part~II we focus on stochastic settings and integral allocations, where LP relaxations may be weak, and we show how LP-free certificates yield a unified analysis across different forms of stochastic outcomes.

%% file: tex/recipe.tex
 We now turn to the first methodological ingredient of this tutorial: a convex-programming-based primal-dual recipe for \textit{both} designing and analyzing competitive online fractional algorithms when  rewards are deterministic (i.e., setting $p_{\ArrivalIdx\OfflineIdx}=1$ throughout this section). We return to settings with stochastic rewards and other forms of post-allocation stochasticity in \Cref{sec:part-II}. 

 \subsection{The universal convex-regularization recipe}
\label{sec:recipe}
The starting point is to revisit the arrival-wise greedy rule: when arrival~$\ArrivalIdx$ appears, solve the local LP that maximizes the immediate reward from this arrival. This rule is usually too aggressive because it ignores the scarcity of offline capacities and fails to hedge against future arrivals. Yet, in deterministic settings, it can easily be shown that greedy is $1/2$-competitive. Our goal is to modify this rule to trade off greediness and hedging, thereby obtaining improved, and ideally optimal, competitive ratios.

Where do optimal competitive fractional algorithms in the literature, such as BALANCE~\citep{kalyanasundaram2000optimal} for unweighted fractional online matching, come from, and is there a systematic way to analyze them? In a nutshell, our recipe provides the answer: to repair the greedy algorithm, we add a convex penalty term---that is, a \emph{regularizer}---on the \emph{post-decision state} to its objective function. As we show shortly, this term not only injects the right kind of balance into the greedy rule, but also enables a dual-fitting-style analysis based solely on the Karush--Kuhn--Tucker (KKT) conditions of our arrival-wise convex programs. Rather than guessing a dual certificate, we extract it directly from the KKT conditions and the corresponding Lagrangian multipliers.

\begin{figure}[t]
\begin{tcolorbox}[
  colback=tutorialgray!8,
  colframe=tutorialgray!60,
  coltitle=black,
  title={\textbf{Five-step convex-regularization recipe}},
  fonttitle=\bfseries,
  arc=2.5pt,
  boxrule=0.7pt,
  top=4pt, bottom=4pt, left=6pt, right=6pt
]
\begin{enumerate}[leftmargin=1.4em, itemsep=3pt]
    \item \textbf{Choose the state.}
    Identify the minimal sufficient statistic that summarizes how past allocations constrain future
    feasibility and reward (e.g., the vector of remaining capacities).

    \item \textbf{Regularize the arrival-wise greedy objective.}
    At each arrival, solve
    \[
        \max\;\bigl\{\text{immediate reward}
        \;-\;
        \text{convex regularizer of the post-decision state}\bigr\}
    \]
    over the feasible allocation for the current arrival.

    \item \textbf{Write the Lagrangian and record the KKT system.}
    Stationarity and complementary slackness encode the marginal tradeoff between immediate reward and the shadow cost of consuming offline capacity.  This system drives the entire dual-fitting analysis.

    \item \textbf{Extract an online dual certificate update from the KKT multipliers.}
    The fitted dual certificate is constructed sequentially as online nodes arrive: once arrival~$\ArrivalIdx$ is processed, its online dual variable is fixed and each offline dual variable is incremented by a quantity read off from the primal optimizer and the KKT Lagrangian multipliers of that arrival.

    \item \textbf{Choose the regularizer so that approximate dual feasibility reduces to a scalar identity.}
    The regularizer is pinned down by requiring that the dual-feasibility gap close at every point in the state space, which usually leads to a suitable functional identity---for example, a differential equation---that yields the target competitive ratio.
\end{enumerate}
\end{tcolorbox}
\end{figure}

At first glance, the recipe in \Cref{sec:part-I} may seem to provide merely an alternative interpretation of classical algorithms in the literature, including BALANCE, through the lens of convex programming. Its purpose, however, is not simply reinterpretation, but rather \emph{systematization}: the same five ingredients---a state variable, a regularized arrival-wise concave program, its KKT conditions, an online dual update, and a scalar identity that closes the dual-fitting argument\footnote{In the matching warm-up this yields the functional equation $1-f(y)+F(y)=\CompRatio$; in the edge-weighted model with free disposal the same identity is applied pointwise in the reward coordinate; and in the batch-arrival model it becomes a stage-wise functional recursion.}---recur across every model we treat, and the recipe uses them to recover optimal competitive ratios in a unified way. We hope this perspective will provide a useful construction that can be applied systematically in new settings, both to design fractional online algorithms and to analyze how they improve upon greedy.

\smallskip
\xhdr{Overview of the primal--dual analysis.}
Before diving into specific models, let us outline the dual-fitting analysis template that recurs throughout Part~I.  Given an online allocation problem whose offline relaxation is a linear program, the goal is to prove that the online algorithm achieves at least a $\CompRatio$-fraction of the offline optimum.  To this end, we use the following generic lemma in \Cref{sec:part-I}.

\begin{lemma}[Generic dual-fitting certificate]
\label{lem:generic-dual-fitting}
Consider an offline primal LP of the form \eqref{eq:intro-generic-primal} and its dual
\eqref{eq:intro-generic-dual}.  Suppose an online algorithm $\ALG$  constructs a nonnegative vector
\(\widehat y\) such that
\begin{equation}
\label{eq:generic-dual-fitting-conditions}
    b^\top \widehat y \le \textrm{obj}(\ALG),
    \qquad
    A^\top \widehat y \ge \CompRatio c
    \quad\text{for some }\CompRatio\in(0,1].
\end{equation}
where $\textrm{obj}(\ALG)$ denotes the primal objective of $\ALG$. Then the algorithm is \(\CompRatio\)-competitive.
\end{lemma}
\begin{proof}{Proof.}
The second inequality in \eqref{eq:generic-dual-fitting-conditions} implies that
\(\widehat y/\CompRatio\) is feasible for the dual LP \eqref{eq:intro-generic-dual}.  Hence, by
weak duality, we have $\OPT
    \le b^\top (\widehat y/\CompRatio)
    = \frac{1}{\CompRatio} b^\top \widehat y
    \le \frac{1}{\CompRatio}\ALG.$ Rearranging gives \(\ALG\ge \CompRatio\OPT\). \qed
\end{proof}

To apply \Cref{lem:generic-dual-fitting}, we consider the offline primal--dual LP pair \eqref{eq:offline-primal-basic}--\eqref{eq:offline-dual-basic} and analyze our online algorithm through the following two-stage approach:

\emph{Stage~(i): exact primal-dual accounting.} We construct a fitted dual vector $(\widehat\alpha,\widehat\beta)$ online, in parallel with the primal allocation, so that at each arrival the increase in the dual objective equals the increase in the primal reward. Summing over all arrivals then gives $\sum_{\ArrivalIdx}\DualOnHat{\ArrivalIdx}+\sum_{\OfflineIdx}\DualOffHat{\OfflineIdx}=R^\ALG$, where $R^\ALG$ denotes the total reward of $\ALG$\footnote{In this section of the tutorial, we study only fractional algorithms, which are without loss of generality deterministic.}. Thus, the first condition of \Lemref{lem:generic-dual-fitting} holds with equality.

\emph{Stage~(ii): approximate dual feasibility.}  For every edge $(\ArrivalIdx,\OfflineIdx)$ in the constraint system of the offline dual LP, we show that $\DualOnHat{\ArrivalIdx}+\DualOffHat{\OfflineIdx}\ge \CompRatio\cdot(\text{edge weight})$.  This is the main technical step.  It typically requires two ingredients: a lower bound on the arrival-side dual variable $\DualOnHat{\ArrivalIdx}$ (obtained directly from the stationarity condition of the KKT system), and a lower bound on the resource-side dual variable $\DualOffHat{\OfflineIdx}$ (obtained by a telescoping argument that uses convexity of the regularizer).

Once these two ingredients are in place, \Lemref{lem:generic-dual-fitting} yields the competitive ratio. The key point is that the dual vector is not guessed after the fact; rather, it is \emph{read off} from the same arrival-wise convex program that defines the online algorithm. As we will see shortly, this reduces the analysis to choosing a regularizer $F$ that satisfies the appropriate scalar identity or functional equation.

\smallskip
\xhdr{How to read the rest of Part~I.}
In \Secref{sec:matching}, we study online fractional matching, where the state is a scalar load vector and the regularizer identity can be solved in closed form. In \Secref{sec:free-disposal}, we turn to edge-weighted matching with free disposal, where the same logic survives but the state becomes a reward-indexed cumulative profile. Several further extensions are treated in the electronic companion; see Table~\ref{tab:ec-roadmap}. Throughout, we encourage the reader to keep track of the main ingredients of the recipe.

\begin{table}[htb]
\centering
\small
\begin{tabularx}{0.98\linewidth}{>{\raggedright\arraybackslash}p{0.29\linewidth} >{\raggedright\arraybackslash}p{0.36\linewidth} >{\raggedright\arraybackslash}X}
\toprule
Extension model & Main additional idea & Location in the electronic companion \\
\midrule
Batch arrival / multi-stage matching & stage-dependent polynomial regularizers and a support-graph decomposition defined batch by batch & \Cref{sec:batch} and the deferred proofs in \Cref{sec:ec-batch} \\
Configuration allocation / whole-page optimization & higher dimensional price-level state variables and direct use of convex duality when no convex-programmin- based graph decomposition is available & \Cref{sec:configuration} and the deferred proofs in \Cref{sec:ec-config} \\
AdWords / online budgeted allocation & normalized budget load as the relevant state variable in the state-dependent convex programming-based allocation & \Cref{subsec:adwords} \\
Costly cancellations & costly variant of free disposal in the buyback model with linear costs& \Cref{subsec:costly-cancel} \\
Additional special cases & direct specializations such as vertex-weighted $b$-matching and  assortment & \Cref{sec:ec-special-cases} \\
\bottomrule
\end{tabularx}
\smallskip
\caption{Additional Part~I models whose full treatment is deferred to the electronic companion.}
\label{tab:ec-roadmap}
\end{table}
\vspace{-3mm}

%% file: tex/matching.tex
\subsection{Application I: Online fractional matching and vertex-weighted extension}
\label{sec:matching}
This section develops the basic convex-programming template in the simplest online matching models.   First, we analyze the unweighted problem in order to isolate the core reduction of the analysis to a functional identity (in fact, a differential equation).  Second, we extend the same calculation to the vertex-weighted model (the unweighted warm-up is simply
the special with \(\VertexWeight{\OfflineIdx}\equiv 1\)).

\subsubsection{Warm-up: online fractional bipartite matching}
\label{subsec:matching-warmup}

Assume first that \(\VertexWeight{\OfflineIdx}\equiv 1\).  Before arrival \(\ArrivalIdx\), let
\(\LoadTime{\OfflineIdx}{\ArrivalIdx-1}\in[0,1]\) denote the current load on offline node
\(\OfflineIdx\).  Let \(F:[0,1]\to\RR\) be differentiable and convex with \(F(0)=0\), and write
\(f:=F'\).
When arrival \(\ArrivalIdx\) occurs, we consider an algorithm that solves
\begin{equation}
\begin{aligned}
    \max_{\{\MatchVar{\ArrivalIdx}{\OfflineIdx}\}_{\OfflineIdx\in\Neighbor{\ArrivalIdx}}}
    \quad
    & \sum_{\OfflineIdx\in\Neighbor{\ArrivalIdx}}
      \MatchVar{\ArrivalIdx}{\OfflineIdx}
      - \sum_{\OfflineIdx\in\Neighbor{\ArrivalIdx}}
      F\bigl(\LoadTime{\OfflineIdx}{\ArrivalIdx-1}+\MatchVar{\ArrivalIdx}{\OfflineIdx}\bigr)
      \\
    \text{s.t.}\quad
    & \sum_{\OfflineIdx\in\Neighbor{\ArrivalIdx}}
      \MatchVar{\ArrivalIdx}{\OfflineIdx} \le 1, \\
    & 0\le \MatchVar{\ArrivalIdx}{\OfflineIdx}
      \le 1-\LoadTime{\OfflineIdx}{\ArrivalIdx-1}
      \qquad \forall \OfflineIdx\in\Neighbor{\ArrivalIdx}.
\end{aligned}
\label{eq:warmup-program}
\end{equation}
The objective is the current matching gain minus a convex penalty for the post-decision load. This program is the concrete realization of Steps~1--2 of the recipe.  The state is the load vector (i.e., current allocation level) $(\LoadTime{\OfflineIdx}{\ArrivalIdx-1})_{\OfflineIdx}$, and the regularizer $F$ penalizes the post-decision load---the formal version of the ``greediness versus hedging'' tradeoff. We now carry out Steps~3--5 in detail.

\smallskip
\xhdr{Lagrangian, KKT system, and the connection to water-filling.}
Introduce a Lagrange multiplier \(\DualOn{\ArrivalIdx}\ge 0\) for the arrival constraint,
multipliers \(\LambdaVar{\OfflineIdx}\ge 0\) for the offline capacity constraints, and
multipliers \(\ThetaVar{\ArrivalIdx}{\OfflineIdx}\ge 0\) for nonnegativity.
The Lagrangian of \eqref{eq:warmup-program} is
\begin{align}
L_{\ArrivalIdx}(x,\alpha,\lambda,\theta)
={}&
\sum_{\OfflineIdx\in\Neighbor{\ArrivalIdx}} \MatchVar{\ArrivalIdx}{\OfflineIdx}
-
\sum_{\OfflineIdx\in\Neighbor{\ArrivalIdx}}
      F\bigl(\LoadTime{\OfflineIdx}{\ArrivalIdx-1}+\MatchVar{\ArrivalIdx}{\OfflineIdx}\bigr)
\nonumber\\
&
+ \DualOn{\ArrivalIdx}\Bigl(1-
\sum_{\OfflineIdx\in\Neighbor{\ArrivalIdx}} \MatchVar{\ArrivalIdx}{\OfflineIdx}\Bigr)
+ \sum_{\OfflineIdx\in\Neighbor{\ArrivalIdx}}
\LambdaVar{\OfflineIdx}\Bigl(1-\LoadTime{\OfflineIdx}{\ArrivalIdx-1}-\MatchVar{\ArrivalIdx}{\OfflineIdx}\Bigr)
+ \sum_{\OfflineIdx\in\Neighbor{\ArrivalIdx}}
\ThetaVar{\ArrivalIdx}{\OfflineIdx}\,\MatchVar{\ArrivalIdx}{\OfflineIdx}~.
\label{eq:warmup-lagrangian}
\end{align}
Let \(\MatchVar{\ArrivalIdx}{\OfflineIdx}^{\ast}\) denote the optimizer.  The KKT stationarity
conditions are
\begin{equation}
1
-
\RegSlope{\LoadTime{\OfflineIdx}{\ArrivalIdx-1}+\MatchVar{\ArrivalIdx}{\OfflineIdx}^{\ast}}
-
\DualOn{\ArrivalIdx}^{\ast}
-
\LambdaVar{\OfflineIdx}^{\ast}
+
\ThetaVar{\ArrivalIdx}{\OfflineIdx}^{\ast}
=
0
\qquad
\forall \OfflineIdx\in\Neighbor{\ArrivalIdx},
\label{eq:warmup-kkt-1}
\end{equation}
with complementary-slackness conditions
\begin{align}
&(\textit{arrival unsaturated})&\sum_{\OfflineIdx\in\Neighbor{\ArrivalIdx}}
\MatchVar{\ArrivalIdx}{\OfflineIdx}^{\ast}<1
&\implies \DualOn{\ArrivalIdx}^{\ast}=0,
\label{eq:warmup-kkt-3}\\
&(\textit{slack capacity})&\LoadTime{\OfflineIdx}{\ArrivalIdx-1}+\MatchVar{\ArrivalIdx}{\OfflineIdx}^{\ast}<1
&\implies \LambdaVar{\OfflineIdx}^{\ast}=0,
\label{eq:warmup-kkt-2}\\
&(\textit{used edge})&\MatchVar{\ArrivalIdx}{\OfflineIdx}^{\ast}>0
&\implies \ThetaVar{\ArrivalIdx}{\OfflineIdx}^{\ast}=0.
\label{eq:warmup-kkt-4}
\end{align}
Equation \eqref{eq:warmup-kkt-1} already reveals that, for \emph{any} strictly convex regularizer~$F$, the optimizer is a water-filling rule---equivalently, unweighted special case of the BALANCE algorithm~\citep{kalyanasundaram2000optimal}. Indeed, any used and unsaturated offline vertex~$\OfflineIdx$ has $\LambdaVar{\OfflineIdx}^{\ast}=\ThetaVar{\ArrivalIdx}{\OfflineIdx}^{\ast}=0$, so \eqref{eq:warmup-kkt-1} yields $f(\LoadTime{\OfflineIdx}{\ArrivalIdx})=1-\DualOn{\ArrivalIdx}^{\ast}$. Since $f$ is strictly increasing, all such vertices must end at the same post-decision load. If instead $\MatchStar{\ArrivalIdx}{\OfflineIdx}=0$ and $\LoadTime{\OfflineIdx}{\ArrivalIdx-1}<1$, then $\LambdaVar{\OfflineIdx}^{\ast}=0$ and \eqref{eq:warmup-kkt-1} implies $\DualOn{\ArrivalIdx}^{\ast}\ge 1-f(\LoadTime{\OfflineIdx}{\ArrivalIdx-1})$, so any unused feasible neighbor already lies above the common water level. Thus the optimizer raises the smallest feasible loads first, stopping when one unit of mass has been assigned or some neighbors saturate.\footnote{\label{footnote:first-order}In fact, the continuous water-filling process is not a combinatorial algorithms; in fact, it is equivalent to running a particular first-order method (coordinate ascent~\citep{wright2015coordinate}) to solve the arrival-wise concave program.} Operationally, one does not need to solve \eqref{eq:warmup-program} from scratch, as its optimizer can be implemented directly by the simple water-filling procedure.

In the unweighted case, the particular choice of regularizer does not affect the structural form of the algorithm: it is always water-filling (as suggested bt the stationarity KKT condition). The choice of $F$ does, however, matter for the analysis. Later we show that the exponential regularizer is the canonical choice that closes the dual-fitting argument in the unweighted case. The choice of $F$ will also play an important role in the vertex-weighted extension developed later; indeed, in that setting the exponential regularizer is the unique choice of $F$ that yields the optimal competitive algorithm.

\smallskip
\xhdr{Online dual construction \& exact primal-dual accounting.}
The offline dual LP is \eqref{eq:offline-dual-basic-unweighted}. Guided by the KKT multipliers, we define the fitted dual certificate online rather than guessing it after the fact. Initialize \(\DualOnHat{\ArrivalIdx}=0\) and \(\DualOffHat{\OfflineIdx}=0\). When arrival~\(\ArrivalIdx\) is processed, set
\begin{equation}
\label{eq:warmup-dual-update}
    \DualOnHat{\ArrivalIdx}:=\DualOn{\ArrivalIdx}^{\ast},
    \qquad
    \BetaIncTime{\OfflineIdx}{\ArrivalIdx}
    :=
    \MatchVar{\ArrivalIdx}{\OfflineIdx}^{\ast}
    \bigl(1-\DualOn{\ArrivalIdx}^{\ast}\bigr),
    \qquad
    \DualOffHat{\OfflineIdx}
    \leftarrow
    \DualOffHat{\OfflineIdx}+\BetaIncTime{\OfflineIdx}{\ArrivalIdx}.
\end{equation}
This construction is online: once arrival~\(\ArrivalIdx\) has been processed, the value of \(\DualOnHat{\ArrivalIdx}\) is fixed permanently, while each \(\DualOffHat{\OfflineIdx}\) continues to accumulate increments according to \eqref{eq:warmup-dual-update}.

\begin{lemma}[Primal increment equals dual increment]
\label{lem:warmup-primal-equals-dual}
For every arrival \(\ArrivalIdx\),
\begin{equation}
\label{eq:warmup-primal-equals-dual}
    \DualOnHat{\ArrivalIdx}
    + \sum_{\OfflineIdx\in\Neighbor{\ArrivalIdx}}
      \BetaIncTime{\OfflineIdx}{\ArrivalIdx}
    =
    \sum_{\OfflineIdx\in\Neighbor{\ArrivalIdx}}
      \MatchVar{\ArrivalIdx}{\OfflineIdx}^{\ast}.
\end{equation}
Consequently, after summing over all arrivals $t$, the fitted dual objective equals the primal value (i.e., fractional matching size) of the online algorithm.
\end{lemma}

\begin{proof}{Proof.}
By definition,
\begin{align*}
    \DualOnHat{\ArrivalIdx}
    + \sum_{\OfflineIdx\in\Neighbor{\ArrivalIdx}} \BetaIncTime{\OfflineIdx}{\ArrivalIdx}
    =
    \DualOn{\ArrivalIdx}^{\ast}
    + \sum_{\OfflineIdx\in\Neighbor{\ArrivalIdx}}
      \MatchVar{\ArrivalIdx}{\OfflineIdx}^{\ast}
      \bigl(1-\DualOn{\ArrivalIdx}^{\ast}\bigr) 
    =
    \sum_{\OfflineIdx\in\Neighbor{\ArrivalIdx}}
      \MatchVar{\ArrivalIdx}{\OfflineIdx}^{\ast}
    + \DualOn{\ArrivalIdx}^{\ast}
      \Bigl(1-
      \sum_{\OfflineIdx\in\Neighbor{\ArrivalIdx}}
      \MatchVar{\ArrivalIdx}{\OfflineIdx}^{\ast}\Bigr).
\end{align*}
The second term vanishes by \eqref{eq:warmup-kkt-3}.  This proves
\eqref{eq:warmup-primal-equals-dual}.\qed
\end{proof}

\smallskip
\xhdr{Choosing the regularizer \& approximate dual feasibility.}
The remaining task is to verify approximate dual feasibility. Our goal is to show that
$\DualOnHat{\ArrivalIdx}+\DualOffHat{\OfflineIdx}\ge \Gamma$ for all edges \((\ArrivalIdx,\OfflineIdx)\in\EdgeSet\),
where $\Gamma$ is the target competitive ratio, to be determined shortly. The key step is the following lemma, which lower-bounds the fitted offline dual variable in terms of the current load.

\begin{lemma}[Lower bound on the offline dual]
\label{lem:warmup-beta-lower}
For every offline node \(\OfflineIdx\) and every time \(\ArrivalIdx\),
\begin{equation}
\label{eq:warmup-beta-lower}
    \DualOffHat{\OfflineIdx}
    \ge
    F\bigl(\LoadTime{\OfflineIdx}{\ArrivalIdx}\bigr).
\end{equation}
\end{lemma}

\begin{proof}{Proof.}
For every earlier arrival \(s\le \ArrivalIdx\) with
\(\MatchVar{s}{\OfflineIdx}^{\ast}>0\), condition
\eqref{eq:warmup-kkt-1} together with \eqref{eq:warmup-kkt-4} gives

\[
    1-\DualOn{s}^{\ast}=\RegSlope{\LoadTime{\OfflineIdx}{s-1}+\MatchVar{s}{\OfflineIdx}^{\ast}}+\LambdaVar{\OfflineIdx}^{\ast}
    \ge
    f(\LoadTime{\OfflineIdx}{s})
\]
Multiplying by
\(\MatchVar{s}{\OfflineIdx}^{\ast}=\LoadTime{\OfflineIdx}{s}-\LoadTime{\OfflineIdx}{s-1}\)
and using convexity of \(F\) (\cref{fig:ec-convexity-argument}),
\begin{equation}
\label{eq:convexity}
    \BetaIncTime{\OfflineIdx}{s}
    \ge
    \bigl(\LoadTime{\OfflineIdx}{s}-\LoadTime{\OfflineIdx}{s-1}\bigr)
    \RegSlope{\LoadTime{\OfflineIdx}{s}}
    \ge
    F\bigl(\LoadTime{\OfflineIdx}{s}\bigr)
    -F\bigl(\LoadTime{\OfflineIdx}{s-1}\bigr).
\end{equation}
Summing over $s=1,\ldots,\ArrivalIdx$ telescopes to \eqref{eq:warmup-beta-lower}.\qed
\end{proof}

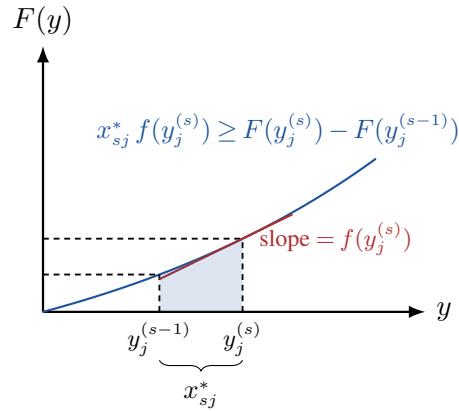
\begin{figure}[htb]
  \centering
  \begin{tikzpicture}[x=4.4cm,y=3.2cm]
    \draw[handaxis] (0,-0.0) -- (1.15,0) node[right] {$y$};
    \draw[handaxis] (0,-0.0) -- (0,1.1) node[above] {$F(y)$};

    \draw[handline, tutorialblue, thick]
      plot[smooth, domain=0:1, samples=80] (\x, {exp(\x - 1) - exp(-1)});

    \def\ya{0.35}
    \def\yb{0.6}
    \pgfmathsetmacro{\Fa}{exp((\ya - 1)) - exp(-1)}
    \pgfmathsetmacro{\Fb}{exp((\yb - 1)) - exp(-1)}
    \pgfmathsetmacro{\fb}{exp(\yb - 1)}

    \draw[dashedguide] (\ya, 0) -- (\ya, \Fa);
    \draw[dashedguide] (\yb, 0) -- (\yb, \Fb);
    \draw[dashedguide] (0, \Fa) -- (\ya, \Fa);
    \draw[dashedguide] (0, \Fb) -- (\yb, \Fb);

    \draw[tutorialred, thick]
      ({\ya}, {\Fb + \fb*(\ya - \yb)}) -- ({\yb + 0.15}, {\Fb + \fb*0.15});
    \node[figlabel, tutorialred, right] at ({\yb + 0.02}, {\Fb + \fb*0.15-0.1})
      {slope $= f(y_j^{(s)})$};

    \fill[tutorialblue, opacity=0.15]
      plot[smooth, domain=\ya:\yb, samples=40] (\x, {exp(\x - 1) - exp(-1)})
      -- (\yb, 0) -- (\ya, 0) -- cycle;

    \node[figlabel, below] at (\ya, 0.0) {$y_j^{(s-1)}$};
    \node[figlabel, below] at (\yb, 0.0) {$y_j^{(s)}$};

    \draw[decorate, decoration={brace, mirror, amplitude=4pt}]
      (\ya, -0.2) -- (\yb, -0.2)
      node[midway, below, yshift=-4pt, font=\small] {$x_{sj}^*$};

    \node[font=\small, anchor=west, tutorialblue] at (0.13, 0.75)
      {$x_{sj}^*\,f(y_j^{(s)})\ge F(y_j^{(s)})-F(y_j^{(s-1)})$};
  \end{tikzpicture}
  \vspace{-3mm}
  \caption{The convexity inequality behind \eqref{eq:convexity}.  The increase of
 linear function with slope $f(y_j^{(s)})$ dominates the increase of the convex function $F$ between consecutive load
  levels. }
  \label{fig:ec-convexity-argument}
\end{figure}

The preceding lemma reduces approximate dual feasibility to a one-dimensional functional identity, which also identifies the correct regularizer.

\begin{theorem}[Warm-up guarantee]
\label{thm:warmup-main}
With the exponential regularizer, the algorithm defined by \eqref{eq:warmup-program} is \((1-1/\Expo)\)-competitive for online fractional bipartite matching.
\end{theorem}

\begin{proof}{Proof.}
Fix an edge \((\ArrivalIdx,\OfflineIdx)\in\EdgeSet\). If \(\LoadTime{\OfflineIdx}{\ArrivalIdx}=1\), then \Lemref{lem:warmup-beta-lower} gives \(\DualOffHat{\OfflineIdx}\ge F(1)\). If instead \(\LoadTime{\OfflineIdx}{\ArrivalIdx}<1\), then \eqref{eq:warmup-kkt-2} implies \(\LambdaVar{\OfflineIdx}^{\ast}=0\), and \eqref{eq:warmup-kkt-1}, together with \(\ThetaVar{\ArrivalIdx}{\OfflineIdx}^{\ast}\ge 0\), yields
\[
    \DualOnHat{\ArrivalIdx}
    = \DualOn{\ArrivalIdx}^{\ast}
    \ge 1-f\bigl(\LoadTime{\OfflineIdx}{\ArrivalIdx}\bigr).
\]
Combining this with \eqref{eq:warmup-beta-lower}, it is natural to choose \(F\) so that the two cases coincide, namely
\begin{equation}
\label{eq:warmup-functional-identity}
    F(1)
    =
    1-f(y)+F(y)
    \qquad \forall y\in[0,1].
\end{equation}
Differentiating \eqref{eq:warmup-functional-identity} and using \(F'=f\) gives \(f'(y)=f(y)\) on \([0,1]\), so \(f(y)=C\Expo^{y}\) for some constant \(C\). Evaluating \eqref{eq:warmup-functional-identity} at \(y=1\) yields \(f(1)=1\), hence \(C=\Expo^{-1}\). Integrating and using \(F(0)=0\), we obtain
\begin{equation}
\label{eq:canonical-regularizer}
    f(y)=\Expo^{y-1},
    \qquad
    F(y)=\Expo^{y-1}-\Expo^{-1},
    \qquad
    F(1)=1-\tfrac{1}{\Expo}.
\end{equation}
With this choice, if \(\LoadTime{\OfflineIdx}{\ArrivalIdx}=1\), then \(\DualOnHat{\ArrivalIdx}+\DualOffHat{\OfflineIdx}\ge F(1)\); if \(\LoadTime{\OfflineIdx}{\ArrivalIdx}<1\), then
\[
    \DualOnHat{\ArrivalIdx}+\DualOffHat{\OfflineIdx}
    \ge 1-f\bigl(\LoadTime{\OfflineIdx}{\ArrivalIdx}\bigr)
      +F\bigl(\LoadTime{\OfflineIdx}{\ArrivalIdx}\bigr)
    = F(1)
    = 1-1/\Expo,
\]
where the equality uses \eqref{eq:warmup-functional-identity}. Therefore, for every edge \((\ArrivalIdx,\OfflineIdx)\in\EdgeSet\),
\[
    \DualOnHat{\ArrivalIdx}+\DualOffHat{\OfflineIdx}
    \ge 1-1/\Expo.
\]
By \Lemref{lem:warmup-primal-equals-dual}, the fitted dual objective equals the value of the online algorithm. Hence \((\widehat\alpha,\widehat\beta)/(1-1/\Expo)\) is feasible for the offline dual LP \eqref{eq:offline-dual-basic-unweighted}, and \Lemref{lem:generic-dual-fitting} completes the proof.\qed
\end{proof}

\subsubsection{Vertex-weighted extension}
\label{subsec:matching-weighted}

We now return to the vertex-weighted model from \Secref{subsec:model-vertexweighted}. The point of the warm-up is that almost all of the analysis is already in place. The state remains the load vector $(\LoadTime{\OfflineIdx}{\ArrivalIdx-1})_{\OfflineIdx}$, and the only new feature is that one unit of mass assigned to $\OfflineIdx$ is now worth $\VertexWeight{\OfflineIdx}$ rather than~$1$. Accordingly, our modified algorithm solves this concave program at each arrival~$\ArrivalIdx$:
\begin{equation}
\begin{aligned}
    \max_{\{\MatchVar{\ArrivalIdx}{\OfflineIdx}\}_{\OfflineIdx\in\Neighbor{\ArrivalIdx}}}
    \quad
    & \sum_{\OfflineIdx\in\Neighbor{\ArrivalIdx}}
      \VertexWeight{\OfflineIdx}\,\MatchVar{\ArrivalIdx}{\OfflineIdx}
      - \sum_{\OfflineIdx\in\Neighbor{\ArrivalIdx}}
      \VertexWeight{\OfflineIdx}
      F\bigl(\LoadTime{\OfflineIdx}{\ArrivalIdx-1}+\MatchVar{\ArrivalIdx}{\OfflineIdx}\bigr)
      \\
    \text{s.t.}\quad
    & \sum_{\OfflineIdx\in\Neighbor{\ArrivalIdx}}
      \MatchVar{\ArrivalIdx}{\OfflineIdx} \le 1, \qquad
     0\le \MatchVar{\ArrivalIdx}{\OfflineIdx}
      \le 1-\LoadTime{\OfflineIdx}{\ArrivalIdx-1}
      \quad \forall \OfflineIdx\in\Neighbor{\ArrivalIdx}.
\end{aligned}
\label{eq:weighted-program}
\end{equation}
Notably, the optimizer of this concave program coincides exactly with the classical BALANCE algorithm~\citep{aggarwal2011vertex,kalyanasundaram2000optimal}, as can be seen from the KKT conditions. We leave this verification as an exercise for the reader; see also \Cref{footnote:first-order}. Compared with the unweighted setting in \Cref{subsec:matching-warmup}, the only KKT condition that changes is stationarity:
\begin{equation}
\label{eq:weighted-kkt}
    \VertexWeight{\OfflineIdx}
    \Bigl(1-f\bigl(\LoadTime{\OfflineIdx}{\ArrivalIdx-1}+\MatchStar{\ArrivalIdx}{\OfflineIdx}\bigr)\Bigr)
    - \DualOn{\ArrivalIdx}^{\ast}
    - \LambdaVar{\OfflineIdx}^{\ast}
    + \ThetaVar{\ArrivalIdx}{\OfflineIdx}^{\ast}
    =0
    \qquad
    \forall \OfflineIdx\in\Neighbor{\ArrivalIdx},
\end{equation}
while \eqref{eq:warmup-kkt-2}--\eqref{eq:warmup-kkt-4} remain unchanged.  The fitted dual update is
\begin{equation}
\label{eq:weighted-dual-update}
    \DualOnHat{\ArrivalIdx}:=\DualOn{\ArrivalIdx}^{\ast},
    \qquad
    \BetaIncTime{\OfflineIdx}{\ArrivalIdx}
    := \MatchStar{\ArrivalIdx}{\OfflineIdx}
       \bigl(\VertexWeight{\OfflineIdx}-\DualOn{\ArrivalIdx}^{\ast}\bigr),
    \qquad
    \DualOffHat{\OfflineIdx}
    \leftarrow
    \DualOffHat{\OfflineIdx}+\BetaIncTime{\OfflineIdx}{\ArrivalIdx}.
\end{equation}

\begin{proposition}[Vertex-weighted guarantee]
\label{prop:weighted-main}
With the exponential regularizer \eqref{eq:canonical-regularizer}, the algorithm \eqref{eq:weighted-program} is $(1-1/\Expo)$-competitive for fractional vertex-weighted online matching.
\end{proposition}

\begin{proof}{Proof.}
The exact primal-dual accounting identity carries over: expanding \eqref{eq:weighted-dual-update} and applying \eqref{eq:warmup-kkt-3} gives $\DualOnHat{\ArrivalIdx}+\sum_{\OfflineIdx}\BetaIncTime{\OfflineIdx}{\ArrivalIdx}=\sum_{\OfflineIdx}\VertexWeight{\OfflineIdx}\MatchStar{\ArrivalIdx}{\OfflineIdx}$, so the fitted dual objective equals the primal value.

For the offline-side lower bound, fix $\OfflineIdx$ and an earlier arrival $s\le\ArrivalIdx$ with $\MatchStar{s}{\OfflineIdx}>0$.  From \eqref{eq:weighted-kkt} and \eqref{eq:warmup-kkt-4}, $\VertexWeight{\OfflineIdx}-\DualOn{s}^{\ast}\ge \VertexWeight{\OfflineIdx}f(\LoadTime{\OfflineIdx}{s})$.  Multiplying by $\MatchStar{s}{\OfflineIdx}$ and applying convexity of~$F$ yields $\BetaIncTime{\OfflineIdx}{s}\ge \VertexWeight{\OfflineIdx}(F(\LoadTime{\OfflineIdx}{s})-F(\LoadTime{\OfflineIdx}{s-1}))$.  Summing over~$s$ therefore gives
\begin{equation}
\label{eq:weighted-beta-lower}
    \DualOffHat{\OfflineIdx}
    \ge \VertexWeight{\OfflineIdx}F\bigl(\LoadTime{\OfflineIdx}{\ArrivalIdx}\bigr).
\end{equation}

Now fix an edge $(\ArrivalIdx,\OfflineIdx)\in\EdgeSet$.  If $\LoadTime{\OfflineIdx}{\ArrivalIdx}=1$, then \eqref{eq:weighted-beta-lower} gives $\DualOnHat{\ArrivalIdx}+\DualOffHat{\OfflineIdx}\ge \VertexWeight{\OfflineIdx}F(1)=(1-1/\Expo)\VertexWeight{\OfflineIdx}$.  If $\LoadTime{\OfflineIdx}{\ArrivalIdx}<1$, then $\LambdaVar{\OfflineIdx}^{\ast}=0$ by \eqref{eq:warmup-kkt-2} and \eqref{eq:weighted-kkt} gives $\DualOnHat{\ArrivalIdx}\ge \VertexWeight{\OfflineIdx}(1-f(\LoadTime{\OfflineIdx}{\ArrivalIdx}))$.  Combining with \eqref{eq:weighted-beta-lower} and applying \eqref{eq:warmup-functional-identity},
\[
    \DualOnHat{\ArrivalIdx}+\DualOffHat{\OfflineIdx}
    \ge \VertexWeight{\OfflineIdx}
       \bigl(1-f(\LoadTime{\OfflineIdx}{\ArrivalIdx})
             +F(\LoadTime{\OfflineIdx}{\ArrivalIdx})\bigr)
    = (1-1/\Expo)\VertexWeight{\OfflineIdx}.
\]
Hence $(\widehat\alpha,\widehat\beta)/(1-1/\Expo)$ is feasible for the offline dual LP, and \Lemref{lem:generic-dual-fitting} completes the proof. \qed
\end{proof}

\begin{remark}\label{unweigh-to-weigh}
The vertex-weighted extension illustrates a useful modeling lesson: changing the reward of an
offline vertex by a multiplicative factor does not force a new state description. In this case, the
load-based convex-regularization framework absorbs the weights without altering the underlying
geometry of the argument, which is why the extension remains conceptually mild. This should be
contrasted with the edge-weighted model treated next, where heterogeneity appears at the edge level
and therefore requires a genuinely richer state variable. It turns out that the bookkeeping
changes in these richer models, but the recipe does not.
\end{remark}

% \begin{figure}[t]
%     \centering
%     \begin{tikzpicture}[x=4.4cm,y=3.2cm]
%         \draw[handaxis] (0,0) -- (1.15,0) node[right] {$y$};
%         \draw[handaxis] (0,0) -- (0,0.92);
%         \draw[handline, tutorialblue]
%             plot[smooth] coordinates {(0,0)(0.15,0.05)(0.35,0.13)(0.55,0.28)(0.75,0.52)(1,0.632)};
%         \draw[dashedguide] (1,0) -- (1,0.632);
%         \draw[dashedguide] (0,0.632) -- (1,0.632);
%         \node[figlabel, below] at (1,0) {$1$};
%         \node[figlabel, left] at (0,0.632) {$1{-}1/\Expo$};
%         \node[figlabel, text=tutorialblue, anchor=south west] at (0.43,0.33)
%         {$F(y)=\Expo^{y-1}-\Expo^{-1}$};
%     \end{tikzpicture}
%     \caption{The canonical exponential regularizer.  The quantity \(F(1)=1-1/\Expo\) is exactly
%     the competitive ratio certified by the dual-fitting argument.}
%     \label{fig:regularizer-curve}
% \end{figure}

%% file: tex/free_disposal.tex
\subsection{Application II: Edge-weighted matching with free disposal}
\label{sec:free-disposal}
We now turn to the second core model of Part~I: edge-weighted matching with free disposal. The transition from vertex weights to edge weights is substantive, not cosmetic. In \Secref{sec:matching}, the state of offline vertex~$\OfflineIdx$ is summarized by a single scalar load because every retained allocation unit at~$\OfflineIdx$ has the \emph{same} value $\VertexWeight{\OfflineIdx}$. In the edge-weighted model, by contrast, retained mass might have been created at different reward levels and they are not interchangeable, so the state must record a reward profile (e.g., a density function) rather than a single scalar. This distinction is consequential: without additional structure, even a single offline vertex of capacity one yields no constant competitive online algorithm, since an adversary can reveal a small reward first and an arbitrarily larger reward later.

Free disposal is the structural assumption that removes exactly this obstruction. It allows the algorithm to discard previously accepted low-value mass when a later arrival creates more valuable mass at the same offline vertex, so the online policy may retain only the highest-value mass up to capacity. That is the appropriate benchmark in applications such as display advertising~\citep{feldman2009free}, and it is the assumption that restores tractability in the edge-weighted model.

\subsubsection{Problem setting, notation, and benchmark}
\label{subsec:model-freedisposal}

Formally, the reward of assigning arrival \(\ArrivalIdx\) to vertex \(\OfflineIdx\) is the edge
value \(\EdgeReward{\ArrivalIdx}{\OfflineIdx}\ge 0\).  The offline LP benchmark remains the maximum edge-weighted fractional matching problem.  The new ingredient is the richer state of this allocation problem.

For each offline node \(\OfflineIdx\), the pre-arrival state at time \(\ArrivalIdx\) is described by a
nonnegative reward density
\(\DensityTime{\OfflineIdx}{\ArrivalIdx-1}{\RewardVar}\) on \([0,\infty)\).  Intuitively,
\(\DensityTime{\OfflineIdx}{\ArrivalIdx-1}{\RewardVar}\,\dd\RewardVar\) is the amount of mass
currently retained at reward levels in \([\RewardVar,\RewardVar+\dd\RewardVar]\).  It is often more
convenient to work with the cumulative profile
\begin{equation}
\label{eq:model-free-cumulative}
    \CumTime{\OfflineIdx}{\ArrivalIdx-1}{\RewardVar}
    :=
    \int_{\RewardVar}^{\infty}
    \DensityTime{\OfflineIdx}{\ArrivalIdx-1}{\omega}\,\dd\omega.
\end{equation}
Thus \(\CumTime{\OfflineIdx}{\ArrivalIdx-1}{\RewardVar}\) is the retained mass of vertex
\(\OfflineIdx\) whose reward is at least \(\RewardVar\), and the total allocated mass at node \(\OfflineIdx\) is encoded by
\(\CumTime{\OfflineIdx}{\ArrivalIdx-1}{0}\), which should be no more than the capacity $1$.

When arrival \(\ArrivalIdx\) appears, the algorithm chooses
\begin{enumerate}[leftmargin=1.4em,itemsep=1pt]
    \item new allocation masses \(\NewMass{\ArrivalIdx}{\OfflineIdx}\ge 0\) for
    \(\OfflineIdx\in\Neighbor{\ArrivalIdx}\), satisfying
    \(\sum_{\OfflineIdx\in\Neighbor{\ArrivalIdx}} \NewMass{\ArrivalIdx}{\OfflineIdx}\le 1\), and
    \item a post-cancellation density
    \(\PostDensity{\OfflineIdx}{\RewardVar}\) with
    \(0\le \PostDensity{\OfflineIdx}{\RewardVar}
      \le \DensityTime{\OfflineIdx}{\ArrivalIdx-1}{\RewardVar}\).
\end{enumerate}
The total post-decision mass kept at offline node \(\OfflineIdx\) must satisfy
\begin{equation}
\label{eq:model-free-offline-constraint}
    \NewMass{\ArrivalIdx}{\OfflineIdx}
    + \int_0^{\infty} \PostDensity{\OfflineIdx}{\RewardVar}\,\dd\RewardVar
    \le 1.
\end{equation}
The resulting cumulative state is
\begin{equation}
\label{eq:model-free-state-update}
    \CumTime{\OfflineIdx}{\ArrivalIdx}{\RewardVar}
    :=
    \int_{\RewardVar}^{\infty}
    \PostDensity{\OfflineIdx}{\omega}\,\dd\omega
    + \NewMass{\ArrivalIdx}{\OfflineIdx}
      \ind{\EdgeReward{\ArrivalIdx}{\OfflineIdx}\ge \RewardVar}.
\end{equation}
Equation \eqref{eq:model-free-state-update} says that newly created mass contributes to every
threshold below its reward value, while previously retained mass contributes through the surviving
density \(\PostDensity{j}{\cdot}\).

The offline LP benchmark is the maximum edge-weighted fractional matching problem:

\begin{minipage}[t]{0.49\linewidth}
\textbf{Primal LP}
\begin{equation}
\begin{aligned}[t]
    \max \quad
    & \sum_{(\ArrivalIdx,\OfflineIdx)\in\EdgeSet}
      \EdgeReward{\ArrivalIdx}{\OfflineIdx}\,\MatchVar{\ArrivalIdx}{\OfflineIdx} \\
    \text{s.t.}\quad
    & \sum_{\OfflineIdx:(\ArrivalIdx,\OfflineIdx)\in\EdgeSet}
      \MatchVar{\ArrivalIdx}{\OfflineIdx} \le 1
      && \forall \ArrivalIdx, \\
    & \sum_{\ArrivalIdx:(\ArrivalIdx,\OfflineIdx)\in\EdgeSet}
      \MatchVar{\ArrivalIdx}{\OfflineIdx} \le 1
      && \forall \OfflineIdx, \\
    & \MatchVar{\ArrivalIdx}{\OfflineIdx}\ge 0
      && \forall (\ArrivalIdx,\OfflineIdx)\in\EdgeSet.
\end{aligned}
\label{eq:offline-primal-free}
\end{equation}
\end{minipage}\hfill
\begin{minipage}[t]{0.49\linewidth}
\textbf{Dual LP}
\begin{equation}
\begin{aligned}[t]
    \min \quad
    & \sum_{\ArrivalIdx\in\OnlineVertices}\DualOn{\ArrivalIdx}
      + \sum_{\OfflineIdx\in\OfflineVertices}\DualOff{\OfflineIdx} \\
    \text{s.t.}\quad
    & \DualOn{\ArrivalIdx}+\DualOff{\OfflineIdx}
      \ge \EdgeReward{\ArrivalIdx}{\OfflineIdx}
      && \forall (\ArrivalIdx,\OfflineIdx)\in\EdgeSet, \\
    & \DualOn{\ArrivalIdx},\DualOff{\OfflineIdx}\ge 0
      && \forall \ArrivalIdx,\ \forall \OfflineIdx.
\end{aligned}
\label{eq:offline-dual-free}
\end{equation}
\end{minipage}
\smallskip

\subsubsection{The regularized arrival-wise convex program \& analysis}
\label{subsec:free-program}
Before processing arrival~\(\ArrivalIdx\), the state of offline node~\(\OfflineIdx\) is the
reward density \(\PreDensity{\OfflineIdx}{\RewardVar}\) and its cumulative profile
\(\PreCum{\OfflineIdx}{\RewardVar}=\int_{\RewardVar}^{\infty}
\PreDensity{\OfflineIdx}{\omega}\,\dd\omega\).\footnote{When the current time is clear from the context, we write $y_j(\cdot)$ and $Y_j(\cdot)$ instead of their time-dependent forms $y_j^t(\cdot)$ and $Y_j^t(\cdot)$.} At arrival~\(\ArrivalIdx\), the algorithm
chooses new mass \(\NewMass{\ArrivalIdx}{\OfflineIdx}\ge 0\) and a retained
post-cancellation density \(\PostDensity{\OfflineIdx}{\RewardVar}\le
\PreDensity{\OfflineIdx}{\RewardVar}\). The regularized arrival-wise program is
\begin{equation}
\begin{aligned}
    \max \quad
    & \sum_{\OfflineIdx\in\Neighbor{\ArrivalIdx}}
      \NewMass{\ArrivalIdx}{\OfflineIdx}\,\EdgeReward{\ArrivalIdx}{\OfflineIdx}
      - \sum_{\OfflineIdx\in\OfflineVertices}
        \int_0^{\infty}
        \bigl(\PreDensity{\OfflineIdx}{\RewardVar}-\PostDensity{\OfflineIdx}{\RewardVar}\bigr)
        \RewardVar\,\dd\RewardVar 
      - \sum_{\OfflineIdx\in\OfflineVertices}
        \int_0^{\infty}
        \RegVal{\PostCum{\OfflineIdx}{\RewardVar}
        + \NewMass{\ArrivalIdx}{\OfflineIdx}
          \ind{\EdgeReward{\ArrivalIdx}{\OfflineIdx}\ge \RewardVar}}
        \dd\RewardVar \\
    \text{s.t.}\quad
    & \sum_{\OfflineIdx\in\Neighbor{\ArrivalIdx}}
      \NewMass{\ArrivalIdx}{\OfflineIdx} \le 1, \\
    & \NewMass{\ArrivalIdx}{\OfflineIdx}
      + \int_0^{\infty}\PostDensity{\OfflineIdx}{\RewardVar}\,\dd\RewardVar
      \le 1
      \qquad \forall \OfflineIdx\in\OfflineVertices, \\
    & 0\le \PostDensity{\OfflineIdx}{\RewardVar}
      \le \PreDensity{\OfflineIdx}{\RewardVar}
      \qquad \forall \OfflineIdx\in\OfflineVertices,\ \forall \RewardVar\ge 0, \\
    & \NewMass{\ArrivalIdx}{\OfflineIdx}\ge 0
      \qquad \forall \OfflineIdx\in\OfflineVertices.
\end{aligned}
\label{eq:free-program}
\end{equation}
The first term is the reward of the new assignment, the second subtracts the reward lost through disposal, and the third regularizes the post-decision reward histogram.  This is the exact analogue of the matching program from \Secref{sec:matching}; only the state of each vertex $j$ has changed from a scalar load to a reward-indexed density/cumulative profile.

\smallskip
\xhdr{Lagrangian, KKT conditions, \& dual construction.}
Introduce a multiplier \(\DualOn{\ArrivalIdx}\ge 0\) for the online arrival constraint,
\(\LambdaVar{\OfflineIdx}\ge 0\) for the capacity constraint of each offline vertex,
\(\EtaVar{\OfflineIdx}{\RewardVar}\ge 0\) for the upper bound
\( \PreDensity{\OfflineIdx}{\RewardVar}\) upper bound on retained density $\PostDensity{\OfflineIdx}{\RewardVar}$,
\(\ThetaVar{\ArrivalIdx}{\OfflineIdx}\ge 0\) for nonnegativity of the new mass,
and \(\DeltaVar{\OfflineIdx}{\RewardVar}\ge 0\) for nonnegativity of the retained density.
The Lagrangian is
\begin{align}
L_{\ArrivalIdx}(z,x,\alpha,\lambda,\eta,\theta,\delta)
={}&
\sum_{\OfflineIdx}\NewMass{\ArrivalIdx}{\OfflineIdx}\,\EdgeReward{\ArrivalIdx}{\OfflineIdx}
-
\sum_{\OfflineIdx}\int_0^{\infty}
\bigl(\PreDensity{\OfflineIdx}{\RewardVar}-\PostDensity{\OfflineIdx}{\RewardVar}\bigr)
\RewardVar\,\dd\RewardVar
\nonumber\\
&
-
\sum_{\OfflineIdx}\int_0^{\infty}
\RegVal{\PostCum{\OfflineIdx}{\RewardVar}
+\NewMass{\ArrivalIdx}{\OfflineIdx}
\ind{\EdgeReward{\ArrivalIdx}{\OfflineIdx}\ge \RewardVar}}
\dd\RewardVar
\nonumber\\
&
+ \DualOn{\ArrivalIdx}\Bigl(1-\sum_{\OfflineIdx}\NewMass{\ArrivalIdx}{\OfflineIdx}\Bigr)
+ \sum_{\OfflineIdx}\LambdaVar{\OfflineIdx}
\Bigl(1-\NewMass{\ArrivalIdx}{\OfflineIdx}-\int_0^{\infty}\PostDensity{\OfflineIdx}{\RewardVar}\,\dd\RewardVar\Bigr)
\nonumber\\
&
+ \sum_{\OfflineIdx}\int_0^{\infty}
\EtaVar{\OfflineIdx}{\RewardVar}
\bigl(\PreDensity{\OfflineIdx}{\RewardVar}-\PostDensity{\OfflineIdx}{\RewardVar}\bigr)
\dd\RewardVar
+ \sum_{\OfflineIdx}\ThetaVar{\ArrivalIdx}{\OfflineIdx}\NewMass{\ArrivalIdx}{\OfflineIdx}
\nonumber\\
&
+ \sum_{\OfflineIdx}\int_0^{\infty}
\DeltaVar{\OfflineIdx}{\RewardVar}\PostDensity{\OfflineIdx}{\RewardVar}\,\dd\RewardVar.
\label{eq:free-lagrangian}
\end{align}
Let \((\NewMass{\ArrivalIdx}{\OfflineIdx}^{\ast},\PostDensity{\OfflineIdx}{\RewardVar}^{\ast})\)
denote an optimal solution of the Lagrangian problem.  The stationarity conditions  are
\begin{equation}
\EdgeReward{\ArrivalIdx}{\OfflineIdx}
- \int_0^{\EdgeReward{\ArrivalIdx}{\OfflineIdx}}
  \RegSlope{\CumTime{\OfflineIdx}{\ArrivalIdx}{\RewardVar}}\,\dd\RewardVar
- \DualOn{\ArrivalIdx}^{\ast}
- \LambdaVar{\OfflineIdx}^{\ast}
+ \ThetaVar{\ArrivalIdx}{\OfflineIdx}^{\ast}=0,
\label{eq:free-kkt-z}
\end{equation}
for the new mass (i.e., first-order condition of the gradient with respect to variable $\NewMass{\ArrivalIdx}{\OfflineIdx}$), and
\begin{equation}
\RewardVar
- \int_0^{\RewardVar}
  \RegSlope{\CumTime{\OfflineIdx}{\ArrivalIdx}{\omega}}\,\dd\omega
- \LambdaVar{\OfflineIdx}^{\ast}
- \EtaVar{\OfflineIdx}{\RewardVar}^{\ast}
+ \DeltaVar{\OfflineIdx}{\RewardVar}^{\ast}=0,
\label{eq:free-kkt-x}
\end{equation}
for the retained density (i.e., first-order condition of the gradient with respect to variable $\PostDensity{\OfflineIdx}{\RewardVar}$).
The complementary-slackness relations are
\begin{align}
&(\textit{arrival unsaturated})&
\sum_{\OfflineIdx\in\Neighbor{\ArrivalIdx}}\NewMass{\ArrivalIdx}{\OfflineIdx}^{\ast}<1
&\implies \DualOn{\ArrivalIdx}^{\ast}=0,
\label{eq:free-cs-1}\\
&(\textit{slack capacity})&
\NewMass{\ArrivalIdx}{\OfflineIdx}^{\ast}
+ \int_0^{\infty}\PostDensity{\OfflineIdx}{\RewardVar}^{\ast}\,\dd\RewardVar<1
&\implies \LambdaVar{\OfflineIdx}^{\ast}=0,
\label{eq:free-cs-2}\\
&(\textit{used edge})&
\NewMass{\ArrivalIdx}{\OfflineIdx}^{\ast}>0
&\implies \ThetaVar{\ArrivalIdx}{\OfflineIdx}^{\ast}=0,
\label{eq:free-cs-3}\\
&(\textit{disposed mass})&
\PostDensity{\OfflineIdx}{\RewardVar}^{\ast}<\PreDensity{\OfflineIdx}{\RewardVar}
&\implies \EtaVar{\OfflineIdx}{\RewardVar}^{\ast}=0,
\label{eq:free-cs-4}\\
&(\textit{retained mass})&
\PostDensity{\OfflineIdx}{\RewardVar}^{\ast}>0
&\implies \DeltaVar{\OfflineIdx}{\RewardVar}^{\ast}=0.
\label{eq:free-cs-5}
\end{align}
Relative to the vertex-weighted matching model in \Cref{sec:matching}, the main substantive change---beyond the new primal and dual variables for retention and disposal---is that the scalar marginal term \(f(y_j)\) is replaced by the cumulative integral \(\int_0^{r}f(\CumTime{\OfflineIdx}{\ArrivalIdx}{\omega})\,\dd\omega\).

The offline benchmark LP and its dual were given in \eqref{eq:offline-primal-free}--\eqref{eq:offline-dual-free}. Guided by the KKT system, we initialize $\DualOnHat{\ArrivalIdx}=0$ and $\DualOffHat{\OfflineIdx}=0$, and after processing arrival~$\ArrivalIdx$ we set
\begin{equation}
\DualOnHat{\ArrivalIdx} := \DualOn{\ArrivalIdx}^{\ast},
\quad
\BetaInc{\OfflineIdx}
:= \NewMass{\ArrivalIdx}{\OfflineIdx}^{\ast}
   \bigl(\EdgeReward{\ArrivalIdx}{\OfflineIdx}-\DualOn{\ArrivalIdx}^{\ast}\bigr)
   - \int_0^{\infty}
     \bigl(\PreDensity{\OfflineIdx}{\RewardVar}-\PostDensity{\OfflineIdx}{\RewardVar}^{\ast}\bigr)
     \RewardVar\,\dd \RewardVar,
\label{eq:free-dual-update}
\end{equation}
followed by the update $\DualOffHat{\OfflineIdx}\leftarrow\DualOffHat{\OfflineIdx}+\BetaInc{\OfflineIdx}$. Thus $\DualOnHat{\ArrivalIdx}$ is fixed immediately when arrival~$\ArrivalIdx$ is processed, while $\DualOffHat{\OfflineIdx}$ accumulates offline-side increments over time. The quantity $\BetaInc{\OfflineIdx}$ records the reward contribution of the new assignment net of the arrival dual and subtracts the value of the discarded mass, so the fitted dual is built online in lockstep with the primal state.

\begin{lemma}[Primal increment equals dual increment]
\label{lem:free-primal-equals-dual}
For every arrival~$\ArrivalIdx$, the increase in the fitted dual objective equals the change in the primal reward.
\end{lemma}

\begin{proof}{Proof.}
Summing \eqref{eq:free-dual-update} over~$\OfflineIdx$ yields
\[
\DualOnHat{\ArrivalIdx} + \sum_{\OfflineIdx} \BetaInc{\OfflineIdx}
= \sum_{\OfflineIdx}
   \NewMassStar{\ArrivalIdx}{\OfflineIdx}\EdgeReward{\ArrivalIdx}{\OfflineIdx}
   - \sum_{\OfflineIdx} \int_0^{\infty}
     \bigl(\PreDensity{\OfflineIdx}{\RewardVar}-\PostDensityStar{\OfflineIdx}{\RewardVar}\bigr)
     \RewardVar\,\dd \RewardVar
   + \DualOn{\ArrivalIdx}^{\ast}
     \Bigl(1-\sum_{\OfflineIdx} \NewMassStar{\ArrivalIdx}{\OfflineIdx}\Bigr).
\]
The last term vanishes by \eqref{eq:free-cs-1}.  The remaining expression is the new-assignment reward minus the disposal loss, i.e., the primal reward increment at arrival~$\ArrivalIdx$. \qed
\end{proof}

\smallskip
\xhdr{Resource-side dual lower bound.}
The next lemma is the continuous-state analogue of the load-based lower bound from \Secref{sec:matching}.  It gives both the one-arrival increment inequality and the telescoping lower bound used in the competitive-ratio proof. We assume $F$ is strictly convex and $F(0)=0$.

\begin{lemma}[Resource-side dual lower bound]
\label{lem:free-beta-increment}
For every offline vertex~$\OfflineIdx$ and every arrival~$\ArrivalIdx$,
\begin{equation}
\BetaInc{\OfflineIdx}
\ge \int_0^{\infty}
  \RegSlope{\CumTime{\OfflineIdx}{\ArrivalIdx}{\RewardVar}}
  \bigl(\CumTime{\OfflineIdx}{\ArrivalIdx}{\RewardVar}
        - \CumTime{\OfflineIdx}{\ArrivalIdx-1}{\RewardVar}\bigr)
  \dd \RewardVar.
\label{eq:free-beta-lower}
\end{equation}
Consequently, for every offline vertex~$\OfflineIdx$ and every time~$\ArrivalIdx$,
\begin{equation}
\label{eq:free-beta-telescope}
    \DualOffHat{\OfflineIdx}
    \ge \int_0^{\infty}
       \RegVal{\CumTime{\OfflineIdx}{\ArrivalIdx}{\RewardVar}}\,\dd \RewardVar.
\end{equation}
\end{lemma}
\begin{proof}{Proof.}
Fix~$\OfflineIdx$ and write $z^*:=\NewMassStar{\ArrivalIdx}{\OfflineIdx}$, $x^*(\RewardVar):=\PostDensityStar{\OfflineIdx}{\RewardVar}$, $y(\RewardVar):=\PreDensity{\OfflineIdx}{\RewardVar}$, and $Y^*(\RewardVar):=\CumTime{\OfflineIdx}{\ArrivalIdx}{\RewardVar}$.
By \eqref{eq:free-dual-update},
\begin{equation}
\label{eq:free-beta-proof-start}
    \BetaInc{\OfflineIdx}
    = z^*\bigl(\EdgeReward{\ArrivalIdx}{\OfflineIdx}-\DualOn{\ArrivalIdx}^{\ast}\bigr)
      - \int_0^{\infty} \bigl(y(\RewardVar)-x^*(\RewardVar)\bigr)\RewardVar\,\dd\RewardVar.
\end{equation}
If $z^*>0$, then \eqref{eq:free-cs-3} and \eqref{eq:free-kkt-z} give
\begin{equation}
\label{eq:free-z-kkt-active}
    \EdgeReward{\ArrivalIdx}{\OfflineIdx}-\DualOn{\ArrivalIdx}^{\ast}
    = \LambdaVar{\OfflineIdx}^{\ast}
      + \int_0^{\EdgeReward{\ArrivalIdx}{\OfflineIdx}}
        \RegSlope{Y^*(\RewardVar)}\,\dd\RewardVar.
\end{equation}
If $z^*=0$, the first term of \eqref{eq:free-beta-proof-start} vanishes and the same bound holds trivially.  Likewise, if $y(\RewardVar)-x^*(\RewardVar)>0$, then \eqref{eq:free-cs-4} and \eqref{eq:free-kkt-x} imply
\begin{equation}
\label{eq:free-x-kkt-disposed}
    \RewardVar
    \le \LambdaVar{\OfflineIdx}^{\ast}
      + \int_0^{\RewardVar}\RegSlope{Y^*(\omega)}\,\dd\omega.
\end{equation}
Substituting these bounds into \eqref{eq:free-beta-proof-start} gives
\begin{align*}
    \BetaInc{\OfflineIdx}
    &\ge z^* \int_0^{\EdgeReward{\ArrivalIdx}{\OfflineIdx}}
           \RegSlope{Y^*(\RewardVar)}\,\dd\RewardVar
      - \int_0^{\infty}
        \bigl(y(\RewardVar)-x^*(\RewardVar)\bigr)
        \int_0^{\RewardVar}\RegSlope{Y^*(\omega)}\,\dd\omega\,\dd\RewardVar \\
    &\qquad
      + \LambdaVar{\OfflineIdx}^{\ast}
        \Bigl(z^* + \int_0^{\infty}x^*(\RewardVar)\,\dd\RewardVar
              - \int_0^{\infty}y(\RewardVar)\,\dd\RewardVar\Bigr).
\end{align*}
The last term is nonnegative: either $\LambdaVar{\OfflineIdx}^{\ast}=0$ by \eqref{eq:free-cs-2}, or the capacity constraint is tight so that $z^*+\int x^*=1\ge \int y$.  Dropping this term and applying Fubini's theorem,
\begin{align*}
    \BetaInc{\OfflineIdx}
    &\ge \int_0^{\infty}
      \RegSlope{Y^*(\omega)}
      \Bigl(
        z^*\ind{\EdgeReward{\ArrivalIdx}{\OfflineIdx}\ge \omega}
        + \int_{\omega}^{\infty}x^*(\RewardVar)\,\dd\RewardVar
        - \int_{\omega}^{\infty}y(\RewardVar)\,\dd\RewardVar
      \Bigr)\dd\omega.
\end{align*}
By the state-update identity \eqref{eq:model-free-state-update}, the quantity in parentheses is $\CumTime{\OfflineIdx}{\ArrivalIdx}{\omega}-\CumTime{\OfflineIdx}{\ArrivalIdx-1}{\omega}$, proving \eqref{eq:free-beta-lower}.

For the second claim, sum \eqref{eq:free-beta-lower} over $s=1,\ldots,\ArrivalIdx$ to get
\[
    \DualOffHat{\OfflineIdx}
    \ge \int_0^{\infty}
       \sum_{s=1}^{\ArrivalIdx}
       \RegSlope{\CumTime{\OfflineIdx}{s}{\RewardVar}}
       \bigl(\CumTime{\OfflineIdx}{s}{\RewardVar}
             - \CumTime{\OfflineIdx}{s-1}{\RewardVar}\bigr)
       \dd\RewardVar.
\]
For each fixed threshold~$\RewardVar$, convexity of $F$ implies $F(a)-F(b)\le f(a)(a-b)$.  Applying this with $a=\CumTime{\OfflineIdx}{s}{\RewardVar}$ and $b=\CumTime{\OfflineIdx}{s-1}{\RewardVar}$, then telescoping in~$s$, gives
\[
    \sum_{s=1}^{\ArrivalIdx}
      \RegSlope{\CumTime{\OfflineIdx}{s}{\RewardVar}}
      \bigl(\CumTime{\OfflineIdx}{s}{\RewardVar}
            - \CumTime{\OfflineIdx}{s-1}{\RewardVar}\bigr)
    \ge \RegVal{\CumTime{\OfflineIdx}{\ArrivalIdx}{\RewardVar}}
\]
since the initial state is empty and $F(0)=0$.  Integrating over~$\RewardVar$ proves \eqref{eq:free-beta-telescope}. \qed
\end{proof}

% \begin{proof}{Proof.}
% Write
% $\CumTime{\OfflineIdx}{\ArrivalIdx}{\RewardVar}
% = \PostCum{\OfflineIdx}{\RewardVar}^{\ast}
%   + \NewMass{\ArrivalIdx}{\OfflineIdx}^{\ast}\ind{\EdgeReward{\ArrivalIdx}{\OfflineIdx}\ge \RewardVar}$.
% If $\NewMass{\ArrivalIdx}{\OfflineIdx}^{\ast}>0$, then~\eqref{eq:free-kkt-z}
% and~\eqref{eq:free-cs-3} imply
% $\EdgeReward{\ArrivalIdx}{\OfflineIdx} - \DualOn{\ArrivalIdx}^{\ast}
% = \int_0^{\EdgeReward{\ArrivalIdx}{\OfflineIdx}}
%   \RegSlope{\CumTime{\OfflineIdx}{\ArrivalIdx}{\RewardVar}}\dd \RewardVar
%   + \LambdaVar{\OfflineIdx}^{\ast}$.
% For reward levels where mass is disposed
% ($\PreDensity{\OfflineIdx}{\RewardVar}-\PostDensity{\OfflineIdx}{\RewardVar}^{\ast}>0$),
% condition~\eqref{eq:free-cs-4} gives $\EtaVar{\OfflineIdx}{\RewardVar}^{\ast}=0$,
% and~\eqref{eq:free-kkt-x} yields
% $\RewardVar - \int_0^{\RewardVar}
% \RegSlope{\CumTime{\OfflineIdx}{\ArrivalIdx}{\omega}}\dd \omega
% \le \LambdaVar{\OfflineIdx}^{\ast}$.
% Substituting these relations into~\eqref{eq:free-dual-update} and applying Fubini's theorem
% to exchange the order of integration gives the bound~\eqref{eq:free-beta-lower}.\qed
% \end{proof}

\smallskip
\xhdr{Threshold structure in the saturated case.} The next lemma is the analogue of the water-level description from ordinary water-filling: when an offline node is active, the KKT system implies that the optimal solution discards the lowest-reward mass first and therefore retains only the highest-value mass, yielding a threshold description of the retained histogram.

\begin{lemma}[Threshold structure of the retained histogram]
\label{lem:free-threshold}
Suppose the capacity constraint of offline vertex~$\OfflineIdx$ is tight after processing arrival~$\ArrivalIdx$, and define the threshold function
\begin{equation}
\ThresholdVal{\RewardVar}
:= \RewardVar - \int_0^{\RewardVar}
   \RegSlope{\CumTime{\OfflineIdx}{\ArrivalIdx}{\omega}}\dd \omega.
\label{eq:free-threshold-function}
\end{equation}
With $f(y)=\Expo^{y-1}$, the function $G$ is nondecreasing, and there exists a threshold $\ThresholdVar\ge 0$ such that all retained mass has reward at least~$\ThresholdVar$.  Equivalently, $\CumTime{\OfflineIdx}{\ArrivalIdx}{\RewardVar}=1$ for every $\RewardVar\in[0,\ThresholdVar]$.
\end{lemma}

\begin{proof}{Proof.}
Because $f(y)=\Expo^{y-1}\le 1$ for $y\in[0,1]$, the derivative $G'(\RewardVar)=1-f(\CumTime{\OfflineIdx}{\ArrivalIdx}{\RewardVar})\ge 0$, so $G$ is nondecreasing.  Now inspect the retained-density KKT condition \eqref{eq:free-kkt-x}.  If $\PostDensityStar{\OfflineIdx}{\RewardVar}>0$, then $\DeltaVarOpt{\OfflineIdx}{\RewardVar}=0$ by \eqref{eq:free-cs-5}, so $G(\RewardVar)=\LambdaVar{\OfflineIdx}^{\ast}+\EtaVarOpt{\OfflineIdx}{\RewardVar}\ge \LambdaVar{\OfflineIdx}^{\ast}$.  Conversely, if disposal occurs at reward level~$\RewardVar$ (i.e., $\PreDensity{\OfflineIdx}{\RewardVar}-\PostDensityStar{\OfflineIdx}{\RewardVar}>0$), then $\EtaVarOpt{\OfflineIdx}{\RewardVar}=0$ by \eqref{eq:free-cs-4}, so $G(\RewardVar)=\LambdaVar{\OfflineIdx}^{\ast}-\DeltaVarOpt{\OfflineIdx}{\RewardVar}\le \LambdaVar{\OfflineIdx}^{\ast}$.  Since $G$ is nondecreasing, these two inequalities force a threshold: disposed mass lies below the crossing point of $G$ with $\LambdaVar{\OfflineIdx}^{\ast}$, and retained mass lies above it.  Because the resource is saturated, $\CumTime{\OfflineIdx}{\ArrivalIdx}{\RewardVar}=1$ for every threshold below that point. \qed
\end{proof}

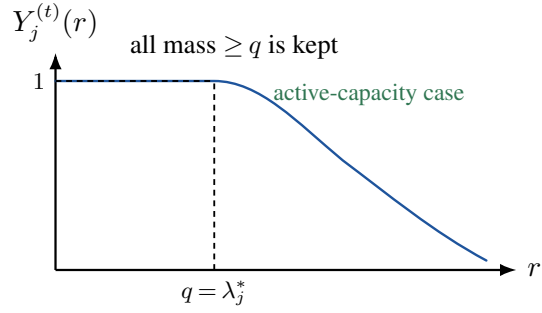
\begin{figure}[t]
    \centering
    \begin{tikzpicture}[x=1cm,y=2.5cm]
        \draw[handaxis] (0,0) -- (6.1,0) node[right] {$\RewardVar$};
        \draw[handaxis] (0,0) -- (0,1.15)
             node[above] {$\CumTime{\OfflineIdx}{\ArrivalIdx}{\RewardVar}$};
        \draw[handline, tutorialblue]
            (0,1) -- (2.1,1)
            .. controls (2.7,1) and (3.2,0.78) .. (3.8,0.58)
            .. controls (4.4,0.40) and (5.0,0.20) .. (5.7,0.05);
        \draw[dashedguide] (2.1,0) -- (2.1,1);
        \draw[dashedguide] (0,1) -- (2.1,1);
        \node[figlabel, below] at (2.1,0) {$\ThresholdVar=\lambda_j^*$};
        \node[figlabel, left] at (0,1) {$1$};
        \node[figlabel, text=tutorialgreen, anchor=south west] at (2.75,0.82)
        {active-capacity case};
        \node[figlabel, anchor=north west] at (0.85,1.3)
        {\small all mass $\ge\ThresholdVar$ is kept};
    \end{tikzpicture}
   \caption{Threshold structure in the saturated case.  The algorithm disposes the lowest-value mass first; the retained histogram has a sharp cutoff at~$\boldsymbol{\ThresholdVar=\lambda^*_j}$.}
    \label{fig:free-threshold}
\end{figure}

\smallskip
\xhdr{Main guarantee.} Putting these ingredients together, we obtain the following theorem.
\begin{theorem}[Edge-weighted with free-disposal guarantee]
\label{thm:free-main}
With $f(y)=\Expo^{y-1}$ and $F(y)=\Expo^{y-1}-\Expo^{-1}$, the online algorithm solving the regularized concave program \eqref{eq:free-program} for each arrival $t$ is $(1-1/\Expo)$-competitive for edge-weighted online matching with free disposal.
\end{theorem}

\begin{proof}{Proof.}
By \Lemref{lem:free-primal-equals-dual}, the fitted dual objective equals the online algorithm's value.  It remains to verify approximate feasibility of the dual LP \eqref{eq:offline-dual-free}.  Fix an edge $(\ArrivalIdx,\OfflineIdx)\in\EdgeSet$ and abbreviate $r:=\EdgeReward{\ArrivalIdx}{\OfflineIdx}$.

\emph{Case 1: $\OfflineIdx$ is not saturated after processing arrival~$\ArrivalIdx$.}
Then $\LambdaVar{\OfflineIdx}^{\ast}=0$ by \eqref{eq:free-cs-2}, so \eqref{eq:free-kkt-z} and $\ThetaVar{\ArrivalIdx}{\OfflineIdx}^{\ast}\ge 0$ give $\DualOnHat{\ArrivalIdx}\ge r-\int_0^{r}f(\CumTime{\OfflineIdx}{\ArrivalIdx}{\RewardVar})\,\dd\RewardVar$.  Combining with \eqref{eq:free-beta-telescope} (restricting the integral to $[0,r]$) and applying the scalar identity $1-f(y)+F(y)=1-1/\Expo$ pointwise,
\[
    \DualOnHat{\ArrivalIdx}+\DualOffHat{\OfflineIdx}
    \ge \int_0^{r}
       \Bigl(1-f\bigl(\CumTime{\OfflineIdx}{\ArrivalIdx}{\RewardVar}\bigr)
             +F\bigl(\CumTime{\OfflineIdx}{\ArrivalIdx}{\RewardVar}\bigr)\Bigr)
       \,\dd\RewardVar
    = (1-1/\Expo)\,r.
\]

\emph{Case 2: $\OfflineIdx$ is saturated at arrival~$\ArrivalIdx$.}
Let $\ThresholdVar$ be the threshold from \Lemref{lem:free-threshold}, so that $q$ is the largest number such that $\CumTime{\OfflineIdx}{\ArrivalIdx}{\RewardVar}=1$ for $\RewardVar\in[0,\ThresholdVar]$. Recall from the proof of \Cref{lem:free-threshold} that  $G$ from \eqref{eq:free-threshold-function} is nondecreasing, and the $\ThresholdVar$ is such that $G(\ThresholdVar)=\LambdaVar{\OfflineIdx}^{\ast}$. We now have two cases. If $r\le \ThresholdVar$, then \eqref{eq:free-beta-telescope} gives 
$$\DualOffHat{\OfflineIdx}\ge \int_0^{r}F(1)\,\dd\RewardVar=(1-1/\Expo)r~,$$
and therefore $ \DualOnHat{\ArrivalIdx}+\DualOffHat{\OfflineIdx}\geq (1-1/\Expo)r$.

Otherwise, if $r>\ThresholdVar$, we first derive an arrival-side lower bound.  By \eqref{eq:free-kkt-z} and $\ThetaVar{\ArrivalIdx}{\OfflineIdx}^{\ast}\ge 0$, we have:
\[
    \DualOnHat{\ArrivalIdx}
    \ge r
      - \int_0^{r}
        f\bigl(\CumTime{\OfflineIdx}{\ArrivalIdx}{\RewardVar}\bigr)\,\dd\RewardVar
      - \LambdaVar{\OfflineIdx}^{\ast}.
\]
Substituting $\LambdaVar{\OfflineIdx}^{\ast}=G(\ThresholdVar)=\ThresholdVar-\int_0^{\ThresholdVar}f(\CumTime{\OfflineIdx}{\ArrivalIdx}{\omega})\,\dd\omega$ and simplifying yields
\begin{equation}
\label{eq:free-alpha-threshold}
    \DualOnHat{\ArrivalIdx}
    \ge \int_{\ThresholdVar}^{r}
       \bigl(1-f\bigl(\CumTime{\OfflineIdx}{\ArrivalIdx}{\RewardVar}\bigr)\bigr)
       \,\dd\RewardVar.
\end{equation}
Moreover, \eqref{eq:free-beta-telescope} and $\CumTime{\OfflineIdx}{\ArrivalIdx}{\RewardVar}=1$ on $[0,\ThresholdVar]$ give
\begin{equation}
\label{eq:free-disp-lower}
    \DualOffHat{\OfflineIdx}
    \ge (1-1/\Expo)\,\ThresholdVar
       + \int_{\ThresholdVar}^{r}
         F\bigl(\CumTime{\OfflineIdx}{\ArrivalIdx}{\RewardVar}\bigr)\,\dd\RewardVar.
\end{equation}
Adding \eqref{eq:free-alpha-threshold} and \eqref{eq:free-disp-lower}, then applying $1-f(y)+F(y)=1-1/\Expo$ pointwise on $[\ThresholdVar,r]$, we obtain
\begin{align*}
    \DualOnHat{\ArrivalIdx}+\DualOffHat{\OfflineIdx}
    &\ge (1-1/\Expo)\,\ThresholdVar
       + \int_{\ThresholdVar}^{r}
         \Bigl(1-f\bigl(\CumTime{\OfflineIdx}{\ArrivalIdx}{\RewardVar}\bigr)
               +F\bigl(\CumTime{\OfflineIdx}{\ArrivalIdx}{\RewardVar}\bigr)\Bigr)
         \,\dd\RewardVar\\
    &= (1-1/\Expo)\,\ThresholdVar + (1-1/\Expo)(r-\ThresholdVar)
    = (1-1/\Expo)r.
\end{align*}
In both cases, $(\widehat\alpha/(1-1/\Expo),\widehat\beta/(1-1/\Expo))$ is feasible for the offline dual LP \eqref{eq:offline-dual-free}, and \Lemref{lem:generic-dual-fitting} completes the proof. \qed
\end{proof}

\begin{remark}
This section shows the main message of Part~I: the convex-programming / dual-fitting template still works even when the state is no longer a single number. Here the state is a reward-indexed cumulative profile. Even so, the same plan goes through: the arrival-wise program chooses the primal action, the KKT conditions define the dual update, and the proof again uses the identity $1-f(y)+F(y)=1-1/\Expo$, now one threshold at a time. The new idea is the threshold description of the retained histogram. For more extensions, see the electronic companion: batching in \Cref{sec:batch,sec:ec-batch} and configuration allocation / whole-page optimization in \Cref{sec:configuration,sec:ec-config}.
\end{remark}

%% file: tex/part2.tex
Recall from Section~\ref{subsec:model-base} that fluid LP relaxations can be loose for stochastic outcomes. This part presents a certificate framework that bypasses these relaxations by replacing the edge-level dual constraints used in Part~I with a \emph{resource-level covering condition}. This universal certificate applies across various stochastic models and enables sample-path coupling when needed. To demonstrate the seamless adaptability of the framework, we analyze the Greedy algorithm in three distinct settings, establishing a $0.5$ lower bound on its competitive ratio in each. Because these settings generalize the base model, this implies that Greedy is \emph{exactly} $0.5$-competitive.

Section~\ref{sec:lp-free} first introduces this LP-free certificate, proving its validity, tightness, and relationship to classical dual constraints in the stochastic-rewards setting. Sections~\ref{sec:greedy-sr}--\ref{sec:reusability} then apply the framework to Greedy under stochastic rewards, matching with patience, and stochastic reusability, unifying results previously derived via disparate proof templates \citep{golrezaei2014real, chan2009stochastic, borodin2022prophet, brubach2025online, gong2022reusable} into a single framework.

  Before proceeding, we summarize the notation from Section~\ref{sec:models} that is used throughout this section. Let $\mathcal{A}\in\{\ALG,\OPT\}$ denote a generic non-anticipative policy. Along a sample path, $U^{\mathcal{A}}_j$ is the set of arrivals $\mathcal{A}$ attempts to match to resource $j$ while available, while $V^{\mathcal{A}}_t$ and $V^{\mathcal{A}}_{T+1}$ denote the available resources at arrival $t$ and at the horizon's end, respectively. Let $\Rew^{\mathcal{A}}$ denote the sample-path deterministic reward, and let $\Rew^{\mathcal{A}}_j$ denote the deterministic reward attributable to resource $j$, so that $\Rew^{\mathcal{A}}=\sum_j \Rew^{\mathcal{A}}_j$. Unless stated otherwise, expectations are taken over the independent random outcomes realized by both policies.

\subsection{The LP-Free Framework}\label{sec:lp-free}
We use the stochastic-rewards setting to motivate the framework. In this setting, the classical dual imposes the following constraint for each edge:
$$\alpha_t+p_{tj}\beta_j\geq p_{tj} r_j.$$
As demonstrated in Example~\ref{rem:loose}, the classical LP upper bound is loose even for trivial problem instances. To obtain a tighter upper bound on the offline benchmark in the primal (maximization) space, one may (for instance) impose additional constraints, which lowers the optimal value. In the dual (minimization) space, this corresponds to a relaxed program with an expanded feasible region. In this spirit, the LP-free certificate deliberately relaxes the standard edge-level dual constraints, opting instead to ``cover'' each resource. Replacing numerous edge-specific constraints with a unified resource-level condition expands the dual feasible region. Minimizing over this larger space yields a smaller dual optimum, ultimately enabling a tight comparison with the offline benchmark.

A first attempt would be to declare resource $j$ covered whenever
\begin{equation}\label{eq:naive-cover}
	\E[\Rew^{\ALG}_j]\;\ge\;\tau\,\E[\Rew^{\OPT}_j],
\end{equation}
  However, \eqref{eq:naive-cover} is generally too strong. For instance, when resources are abundant, $\ALG$ may never use a particular resource $j$ even though $\OPT$ does.

We therefore certify coverage using a \emph{pseudo-reward} defined in terms of auxiliary (sample path) variables $\alpha_{tj}$ and $\beta_j$. We declare resource $j$ covered if its pseudo-reward is at least $\tau\,\E[\Rew^{\OPT}_j]$, where
\begin{equation}\label{eq:pseudoreward}
	\text{Pseudo-reward of $j$:}\qquad
	\E\!\left[\sum_{t\in U^{\OPT}_j}\alpha_{tj}\right]
	\;+\;
	\E[\beta_j].
\end{equation}
The interpretation is as follows. The term $\E[\beta_j]$ is intended to track (a portion of) the expected reward that $\ALG$ obtains from using resource $j$. Same as \eqref{eq:naive-cover}, this is not sufficient on its own, because $\ALG$ may obtain value by assigning arrivals that $\OPT$ sends to $j$ to \emph{other} resources. The variables $\alpha_{tj}$ are designed to account for this effect: if $\OPT$ assigns arrival $t$ to $j$ but $\ALG$ assigns $t$ elsewhere, then $\alpha_{tj}$ can capture (a portion of) the reward that $\ALG$ extracts from this alternative assignment. Summing $\alpha_{tj}$ over $t\in U^{\OPT}_j$ then yields a resource-level proxy for the value that $\ALG$ obtains from arrivals that $\OPT$ uses to generate $\Rew^{\OPT}_j$. 

In the remainder of this section, we utilize a mathematically more flexible notation for the pseudo-reward by summing $\alpha_{tj}$ over all arrivals $t \in T$. Setting $\alpha_{tj} = 0$ for all $t \notin U^{\OPT}_j$ recovers the base definition in \eqref{eq:pseudoreward}.

\begin{theorem}[LP-Free Certificate]\label{thm:form1}
	$\ALG$ is $\tau$-competitive if there exist random variables
	$\{\alpha_{tj}\}_{j\in V,\,t\in U}$ and $\{\beta_j\}_{j\in V}$
	such that
\begin{equation} \label{eq:f1-framework}
	\begin{aligned}
		&\text{\textbf{Budget:}} \quad \E\left[\sum_{t,j} \alpha_{tj} + \sum_{j} \beta_j\right] \le \E[\Rew^{\ALG}] \\[5pt]
		&\text{\textbf{Cover:}}  \quad \E\left[\sum_{t} \alpha_{tj} + \beta_j\right] \ge \tau \cdot \E[\Rew^{\OPT}_j] \quad \forall j \in V
	\end{aligned}
\end{equation}
\end{theorem}

\begin{proof}{Proof.}
	Summing the covering constraints in \eqref{eq:f1-framework} over all $j\in V$ yields
	\begin{equation}\label{eq:f1-combined}
		\begin{aligned}
			\E\!\left[\sum_{j\in V}\sum_{t\in U} \alpha_{tj}
			\;+\;\sum_{j\in V} \beta_j\right]
			&\ge \tau \sum_{j\in V} \E[\Rew^{\OPT}_j] \\[2pt]
			&= \tau \cdot \E[\Rew^{\OPT}].
		\end{aligned}
	\end{equation}
	The budget constraint in \eqref{eq:f1-framework} upper bounds the left-hand side of
	\eqref{eq:f1-combined} by $\E[\Rew^{\ALG}]$, which implies
	$\E[\Rew^{\ALG}] \ge \tau\,\E[\Rew^{\OPT}]$.
\end{proof}

\subsubsection{Tightness}\label{sec:tightness}

The LP-free certificate is tight in the following sense.

\begin{theorem}\label{thm:tight}
	If $\E[\Rew^{\ALG}] \ge \tau \cdot \E[\Rew^{\OPT}]$, then the LP-free certificate admits a feasible solution.
\end{theorem}

\begin{proof}{Proof.}
	Set $\alpha_{tj}= 0$ for all $(t,j)$ and set $\beta_j = \tau\,\E[\Rew^{\OPT}_j]$
	for all $j\in V$ (deterministic constants). Then the covering constraints hold with equality.
	 Moreover, $\E\![\sum_{j\in V}\beta_j]=\sum_{j\in V} \tau\,\E[\Rew^{\OPT}_j]\leq \E[\Rew^{\ALG}]$, so the budget constraint holds. 
%	\begin{equation}\label{eq:tight-budget-check}
%		\begin{aligned}
%			\E\!\left[\sum_{j\in\mathcal{J}}\sum_{t\in\mathcal{T}} \alpha_{jt}
%			+\sum_{j\in\mathcal{J}} \beta_j\right]
%			&= \sum_{j\in\mathcal{J}} \tau\,\E[\Rew^{\OPT}_j] \\[2pt]
%			&= \tau\,\E[\Rew^{\OPT}] \\[2pt]
%			&\le \E[\Rew^{\ALG}],
%		\end{aligned}
%	\end{equation}
\end{proof}

\subsubsection{Convex Combination Interpretation}\label{sec:convex}

In the stochastic rewards model, the LP-free covering constraint can be obtained by aggregating the classical dual constraints \eqref{eq:dual-lp} using weights induced by $\OPT$.

\begin{proposition}\label{prop:convex}
	Fix resource $j$. Multiplying each dual constraint
	$\alpha'_t + p_{tj}\beta'_j \ge p_{tj} r_j$
	by $Q^{\OPT}_{tj} := \Pr[t\in U^{\OPT}_j]$ and summing over $t\in U$ yields:
	\begin{equation}
		\begin{aligned}
		 \sum_t Q^{\OPT}_{tj}\,\alpha'_t + \frac{\beta'_j}{r_j} \underbrace{\sum_t Q^{\OPT}_{tj}\,p_{tj}r_j}_{=\,\E[\Rew^{\OPT}_j]} 
			\ge \underbrace{\sum_{t} Q^{\OPT}_{tj}\, p_{tj}r_j}_{= \E[\Rew^{\OPT}_j]},
		\end{aligned}
	\end{equation}
	Letting $\E[\alpha_{tj}]= Q^{\OPT}_{tj}\,\alpha'_t$ and $\E[\beta_j]=\frac{\beta'_j}{r_j} \,\E[\Rew^{\OPT}_j] $ gives us precisely the LP-free covering constraint for $j$ with $\tau=1$. %Here, we used the fact that $\E[\Rew^{\OPT}_j]\leq r_j$.
\end{proposition}

We now apply this framework to analyze the Greedy algorithm under stochastic rewards, establishing the baseline analysis before extending it to the remaining settings.

\subsection{Application~I: Greedy for Stochastic Rewards}\label{sec:greedy-sr}

\subsubsection{The Greedy Algorithm \& Dual Construction}

Upon arrival $t$, Greedy matches $t$ to an available resource
\[
j^* \in \arg\max_{j\in V^{\ALG}_t} p_{tj}r_j,
\]
breaking ties arbitrarily (and rejecting $t$ if $V^{\ALG}_t=\emptyset$).

To construct certificate variables for the LP-free framework, we simulate $\ALG$ and $\OPT$ on independent sample paths and invoke the deterministic-reward representation (Remark~\ref{obs:det}). We initialize all variables to $0$ and update them online. Whenever $\ALG$ matches $t$ to $j^*$, regardless of outcome we split the deterministic reward $p_{t j^*}r_{j^*}$ into two equal parts: one part contributes to the resource variable $\beta_{j^*}$, and the other is assigned to the unique $\alpha_{tj}$ corresponding to the resource $j$ (if any) that $\OPT$ matches to arrival $t$. Concretely:
\vspace{4mm}
\begin{center}
	\begin{minipage}{0.48\textwidth}
		\begin{equation*}%\label{eq:duals-sr-alpha}
			\text{ $
				\alpha_{t j} := \begin{cases}
					\frac{1}{2}p_{t j^*}r_{j^*} & \text{if } t \in U^{\OPT}_{j} \\ 
					0 & \text{otherwise}
				\end{cases}
				$}
		\end{equation*}
	\end{minipage}
	\hfill
	\begin{minipage}{0.48\textwidth}
		\begin{equation}\label{eq:duals-sr}%\label{eq:duals-sr-beta}
			\text{ $
			\Delta\beta^t_{j^*} = \frac{1}{2}\,p_{t j^*}r_{j^*}.
				$}
		\end{equation}
	\end{minipage}
	\vspace{4mm}
\end{center}
Equivalently, the final value of $\beta_{j^*} =  \sum_{t \in U^{\ALG}_{j^*}} \tfrac{1}{2}p_{t j^*}r_{j^*} = \tfrac{1}{2}\Rew^{\ALG}_{j^*}$ on every sample path.

Note that $\alpha_{tj}$ is indexed by the resource $j$ matched to arrival $t$ by $\OPT$, which need not equal $j^*$. This choice implements the pseudo-reward interpretation: when $\OPT$ uses $j$ at $t$ but $\ALG$ assigns $t$ elsewhere, $\alpha_{tj}$ credits $\ALG$ for the value it obtains from its alternative assignment at $t$.

\subsubsection{Analysis}

\begin{lemma}[Budget]\label{lem:budget-sr}
	$\sum_{t,j}\alpha_{tj} + \sum_j\beta_j \leq \sum_{j}\Rew^{\ALG}_j = \Rew^{\ALG}$.
\end{lemma}
This follows because each match performed by $\ALG$ contributes $\tfrac{1}{2}p_{t j^*}r_{j^*}$ to $\beta_{j^*}$ and contributes at most $\tfrac{1}{2}p_{t j^*}r_{j^*}$ to $\sum_{j}\alpha_{tj}$.

\begin{lemma}[Greedy lower bound]\label{lem:greedy-lb}
	If $\OPT$ matches $t$ to $j$, then $\alpha_{tj} \ge \tfrac{1}{2}\,p_{tj}r_j\cdot\one[j\in V^{\ALG}_{T+1}]$.
\end{lemma}

\begin{proof}{Proof}
	If $j\notin V^{\ALG}_{T+1}$, the bound is immediate. If $j\in V^{\ALG}_{T+1}$, then $j$ is available to $\ALG$ at time $t$. Since Greedy selects $j^*$ with $p_{t j^*}r_{j^*}\ge p_{tj}r_j$, we have
	$\alpha_{tj} = \tfrac{1}{2}p_{t j^*}r_{j^*}\ge \tfrac{1}{2}p_{tj}r_j$.
\end{proof}

\begin{theorem}\label{thm:greedy-sr}
	Greedy is $\frac{1}{2}$-competitive for online matching with stochastic rewards.
\end{theorem}

\begin{proof}{Proof}
	Fix a resource $j$ and independent sample paths for $\ALG$ and $\OPT$. By Lemma~\ref{lem:greedy-lb},
	\[
	\sum_{t\in U^{\OPT}_j}\!\alpha_{tj} + \beta_j
		\ge \tfrac{1}{2}\,\one[j\in V^{\ALG}_{T+1}]\cdot\!\!\sum_{t\in U^{\OPT}_j}\!p_{tj}r_j + \tfrac{1}{2}\,\Rew^{\ALG}_j. 
	\]
	Taking expectations over the independent sample paths yields:
	\begin{equation} \label{eq:sr-step2}
		\begin{aligned}
			\E\left[\sum_{t\in U^{\OPT}_j} \alpha_{tj} + \beta_j\right] 
			&\ge \tfrac{1}{2} \Pr[j \in V^{\ALG}_{T+1}] \cdot \E[\Rew^{\OPT}_j]  + \tfrac{1}{2} \E[\Rew^{\ALG}_j].
		\end{aligned}
	\end{equation}
	By Remark~\ref{obs:det}, $\E[\Rew^{\ALG}_j]= r_j\Pr[j\notin V^{\ALG}_{T+1}]$. Since $\E[\Rew^{\OPT}_j]\le r_j$:
	\begin{equation} \label{eq:sr-final}
		\begin{aligned}
			\text{RHS of~\eqref{eq:sr-step2}} 
			&\ge \tfrac{1}{2}\,\E[\Rew^{\OPT}_j] \cdot \bigl(\Pr[j\in V^{\ALG}_{T+1}]+\Pr[j\notin V^{\ALG}_{T+1}]\bigr) = \tfrac{1}{2}\,\E[\Rew^{\OPT}_j]. 
		\end{aligned}
	\end{equation}
	Thus the covering constraint holds with $\tau=\tfrac{1}{2}$ for every resource $j$. Together with Lemma~\ref{lem:budget-sr}, Theorem~\ref{thm:form1} implies $\E[\Rew^{\ALG}] \ge \tfrac{1}{2}\E[\Rew^{\OPT}]$.
\end{proof}

\begin{remark}[The ``Copy $\OPT$'' Perspective]\label{obs:copy}
	The proof uses only the inequality
	$\alpha_{tj}\ge\frac{1}{2}p_{tj}r_j\cdot\one[j\in V^{\ALG}_{T+1}]$,
	which corresponds to the deterministic reward $\ALG$ would obtain if it replicates $\OPT$'s action and assigns $t$ to $j$ if $j$ is available to $\ALG$ at every arrival. 
    This perspective guides the choice of certificate variables in more general settings.
\end{remark}

\subsection{Application~II: Matching with Patience}\label{sec:patience}
\subsubsection{Setting}

Each arrival $t$ has a patience parameter $\ell_t \ge 1$. Upon arrival, the algorithm selects an ordered sequence $\sigma_t = (j_1, \ldots, j_k)$ of at most $k \le \ell_t$ distinct available resources and attempts to match $t$ to them sequentially. The process begins with $j_1$; if the attempt fails (with probability $1-p_{t j_1}$), the algorithm proceeds to $j_2$, and continues in this manner. This probing sequence terminates at the first successful match or when the patience limit is reached.

The non-anticipative offline benchmark $\OPT$ knows all edge probabilities and patience parameters. It can probe edges in an arbitrary order---including interleaving matching attempts across different arrivals---subject only to the patience constraints and the condition that an arrival is successfully matched at most once. This represents the strongest conceivable offline benchmark against which meaningful competitive ratios can be established.

Because both $\ALG$ and $\OPT$ are non-anticipative with respect to edge outcomes, we continue to use the deterministic-reward representation from Remark~\ref{obs:det}: on each sample path, probing edge $(t,j)$ contributes deterministic reward $p_{tj}r_j$ when $j$ is available and the probe is actually executed (i.e., if it is reached after preceding failures), while the randomness remains in which probes are reached and which resources are ultimately consumed. Note that we continue to use $U^{\mathcal{A}}_j$ to denote the set of arrivals that $\mathcal{A}\in\{\ALG,\OPT\}$ attempted to match to $j$, i.e., by actually probing edge $(t,j)$.

\subsubsection{Greedy, Dual Construction, and Analysis}

Greedy chooses, at each arrival $t$, a feasible probing sequence (of length at most $\ell_t$ using available resources) that maximizes the conditional expected reward from $t$, where probing resource $j$ yields deterministic reward $p_{tj}r_j$ if it is reached and $j$ is available.

Guided by the ``copy $\OPT$'' perspective, we set:
\vspace{4mm}
\begin{center}
	\begin{minipage}{0.55\textwidth}
		\begin{equation*}%\label{eq:alpha-def-mini}
			\text{ $
				\alpha_{tj} := \begin{cases}
					\frac{1}{2}p_{tj}r_j\mathbf{1}[j\in V^{\ALG}_{T+1}] & \text{if } t\in U^{\OPT}_j \\ 
					0 & \text{otherwise}
				\end{cases}
				$}
		\end{equation*}
	\end{minipage}
	\hfill
	\begin{minipage}{0.40\textwidth}
		\begin{equation*}%\label{eq:beta-def-mini}
			\text{ $
				\beta_j := \tfrac{1}{2}\Rew^{\ALG}_j
				$}
		\end{equation*}
	\end{minipage}
\end{center}
\vspace{4mm}
Recall that the indicator $\mathbf{1}[j\in V^{\ALG}_{T+1}]$ captures the event that resource $j$ remains unused by $\ALG$ throughout the horizon. %; in this case, $j$ is available whenever $\OPT$ probes it, and $\alpha_{tj}$ can be interpreted as credit for ``copying'' $\OPT$ on $j$.

%The indicator $\mathbf{1}[j\in V^{\ALG}_{T+1}]$ captures the event that resource $j$ remains unused by $\ALG$ throughout the horizon; in this case, $j$ is available whenever $\OPT$ probes it, and $\alpha_{tj}$ can be interpreted as credit for ``copying'' $\OPT$ on $j$.

\begin{theorem}\label{thm:greedy-pat}
	Greedy is $\frac{1}{2}$-competitive for matching with patience.
\end{theorem}

\begin{proof}{Proof.}
	\emph{Covering constraint.} The covering argument is identical to the stochastic-rewards case: in expectation the indicator $\mathbf{1}[j\in V^{\ALG}_{T+1}]$ and the term $\Rew^{\ALG}_j$ are complementary, and together yield a lower bound of $\tfrac{1}{2}\E[\Rew^{\OPT}_j]$ for every resource $j$.
	
	\medskip\noindent\emph{Budget constraint.} It suffices to show that for every arrival $t \in U$:
	\begin{equation}\label{eq:pat-budget-goal}
		\begin{aligned}
			\E\left[\sum_{j\in V}\alpha_{tj}\right]&= \tfrac{1}{2}\,\E\left[\sum_{j\in V} p_{tj}r_j\, \one[t\in U^{\OPT}_j] \one[j\in V^{\ALG}_{T+1}]\right] \\[5pt]
			&\le \tfrac{1}{2}\,\E\left[\sum_{j\in V} p_{tj}r_j\, \one[t\in U^{\ALG}_j]\right].
		\end{aligned}
	\end{equation}
	This inequality follows from the defining property of Greedy: at each arrival $t$, it selects a feasible probing policy that maximizes the conditional expected reward from $t$, where probing $j$ yields deterministic reward $p_{tj}r_j$ if it is reached and $j$ is available. To see this, condition on arbitrary edge realizations for all arrivals other than $t$, and let $\E_t$ denote expectation with respect to the (independent) realizations of edges incident to $t$. Under this conditioning, the set of resources that remain available in $\ALG$ when $t$ arrives, namely $V^{\ALG}_{t}$, is fixed. 
	The (possibly adaptive) sequence of probes that $\OPT$ executes for arrival $t$ using resources in $V^{\ALG}_{t}$ is feasible for $\ALG$ at $t$, and therefore its expected reward is at most that of Greedy, yielding:
	\begin{equation}\label{eq:greedy-patience}
		\begin{aligned}
			\E_t\left[\sum_{j\in V} p_{tj}r_j\, \one[t\in U^{\OPT}_j]\, \one[j\in V^{\ALG}_{T+1}]\right] 
				&\le \E_t\left[\sum_{j\in V} p_{tj}r_j\, \one[t\in U^{\OPT}_j]\, \one[j\in V^{\ALG}_{t}]\right] \\
			&\le \E_t\left[\sum_{j\in V} p_{tj}r_j\, \one[t\in U^{\ALG}_j]\right].
		\end{aligned}
	\end{equation}
	Applying the tower property establishes \eqref{eq:pat-budget-goal}.
\end{proof}

\begin{remark}[Approximate Greedy]\label{rem:approx-greedy}
	If optimization over probe sequences is computationally prohibitive, employing an $\eta$-approximation algorithm at each arrival degrades the competitive ratio to $\frac{\eta}{1+\eta}$. By scaling $\alpha_{tj}$ and $\beta_j$ by $\frac{\eta}{1+\eta}$, the budget constraint absorbs the approximation loss while the covering argument remains unchanged.
\end{remark}

\subsubsection{Extensions Beyond Matching with Patience}\label{sec:beyond-patience}
The framework extends to a broad class of models that enrich the within-arrival interaction between an arrival and the resource set. This includes: (i) \emph{stochastic patience}, where the number of probes $\ell_t$ is random; (ii) \emph{sequential assortments}, where the algorithm shows a sequence of resource subsets (assortments) to each arrival; (iii) \emph{multi-purchase} or other multi-outcome choice processes, where an arrival may select multiple resources before terminating; and (iv) \emph{two-sided matching and assortment optimization} \citep{aouad2023assortment}, where each resource $j$ derives a non-linear value from its final set of \emph{successfully} matched arrivals. More generally, the interaction between the algorithm and an arrival can be represented by a decision tree: the algorithm selects an action, the arrival produces a random response, and the algorithm adapts its subsequent decisions to the history of responses for that arrival. Throughout, we maintain the standard assumptions that responses are independent across arrivals and that both $\ALG$ and $\OPT$ are non-anticipative with respect to the outcomes of each action (only know outcomes distribution prior to selecting an action).

All of these extensions can be handled by the same proof template. The main step to verify is the per-arrival budget comparison (cf.\ \eqref{eq:pat-budget-goal}). In stochastic patience, sequential assortments, and multi-purchase models, the critical observation is that the expected revenue contribution of arrival $t$ can still be expressed as a sum of marginal expected rewards across resources, under any ``within-arrival" policy governing the sequence of actions and responses at each arrival. Greedy is defined to select, at each arrival, a within-arrival policy that maximizes this conditional expected contribution subject to feasibility with respect to the currently available set $V^{\ALG}_t$.

Formally, consider a within-arrival policy $a \in \mathcal{A}_t$ at $t$. Let $p_{tja}$ denote the marginal probability that (an available) resource $j$ is successfully selected by $t$ under $a$. For any non-anticipative policy, the expected reward from arrival $t$ admits the decomposition $\E_t[\Rew_t(a)] = \sum_j p_{tja}r_j$. Consequently, using the dual construction with $p_{tja}r_j$ in place of $p_{tj}r_j$, the per-arrival budget step reduces to establishing that Greedy's chosen policy $a^* \in \mathcal{A}_t$ dominates $\OPT$'s within-arrival policy at $t$ after restricting attention to the resources feasible for $\ALG$, namely
\begin{equation}\label{eq:generalized-budget}
	\E_t\Big[\sum_{j\in V_t^{\ALG}} p_{tja}r_j\Big] \;\le\; \E_t\Big[\sum_{j\in V} p_{tja^*}r_j\Big],
\end{equation}
where $a$ is the within-arrival policy induced by $\OPT$ at arrival $t$ (noting that $\OPT$ may interleave actions across arrivals). Inequality \eqref{eq:generalized-budget} follows under the same two requirements as in the patience model: (1) the restriction of $a$ to the available set $V^{\ALG}_t$ constitutes a feasible within-arrival policy for $\ALG$, and (2) by non-anticipativity and independence across arrivals, $\OPT$ has no additional information about the realization of the within-arrival interaction beyond what is revealed through probing the arrival.

Two-sided objectives fit into the same scheme with one modification: the ``marginals'' are no longer success probabilities but marginal values under a (typically) submodular resource objective. Specifically, in two-sided environments each resource $j$ derives value $f_j(\mathcal{S}^{\mathcal{A}}_{j})$ from its final set of successfully matched arrivals $\mathcal{S}^{\mathcal{A}}_{j}$. Under standard assumptions, $f_j$ is monotone submodular, and unit rewards are replaced by marginal contributions
$f_j(t \mid S) = f_j(S \cup \{t\}) - f_j(S)$.
The certificate variables are defined using these marginals:
$\alpha_{tj}=\tfrac{1}{2}f_j(t\mid \mathcal{S}^{\ALG}_{j})\one[t\in \mathcal{S}^{\OPT}_j]$ and $\beta_j = \tfrac{1}{2} f_j(\mathcal{S}^{\ALG}_{j})$.
The per-arrival budget step is verified by the same dominance logic, with marginal values replacing probabilities. The covering step is also unchanged in structure, and yields, for any resource $j$,
\begin{align}
	\E\bigg[\sum_{t\in U}\!\alpha_{tj} + \beta_j\bigg]\nonumber
	&\;\ge\; \frac{1}{2} \, \E\Bigg[ \sum_{t\in\mathcal{S}^{\OPT}_j} f_j\!\left(t \mid \mathcal{S}^{\ALG}_{j}\right) + f_j\!\left(\mathcal{S}^{\ALG}_{j}\right) \Bigg] \nonumber \\
	&\;\ge\; \frac{1}{2} \, \E\Big[ f_j\!\left(\mathcal{S}^{\OPT}_{j}\right) \Big], \nonumber
\end{align}
here the final inequality follows from monotonicity and submodularity.

\subsection{Application~III: Stochastic Reusability}\label{sec:reusability}
\label{sec:reusable}
\subsubsection*{Problem setting}

%	Here, we evaluate \emph{reusable} resources. 
When arrival $t$ is matched successfully to resource $j$, the match yields reward $r_j$ and resource $j$ becomes unavailable for a random duration $D\sim \mathcal{D}_j$, drawn independently each time the resource is used. After duration $D$ elapses, the resource returns and may be used again. The algorithm observes return events but does not know realized durations in advance. For simplicity, we take match success probabilities to be $1$, so the only randomness arises through the durations. The offline benchmark $\OPT$ knows the graph and the distributions $\{\mathcal{D}_j\}$ and processes arrivals in chronological order.

\subsubsection{Greedy, Dual Construction, and the New Challenge}\label{sec:reuse-challenge}
We define Greedy as the policy that, at each arrival, matches to an arbitrary available neighboring resource if one exists (and rejects otherwise). We maintain the same ``copy $\OPT$'' template as in the previous sections. In the present model, each successful match yields deterministic reward $r_j$, and the difficulty lies in relating $\ALG$ and $\OPT$ through resource availability over time. We set:
%\[
%\alpha_{jt} := \begin{cases}
%	\frac{1}{2}\cdot\one[j\in\mathcal{J}^{\ALG}_t] 
%	& \text{if } t\in\mathcal{T}^{\OPT}_j,\\[2pt] 
%	0 & \text{otherwise,}
%\end{cases}
%\qquad
%\beta_j := \tfrac{1}{2}\,|\mathcal{T}^{\ALG}_j|
%= \tfrac{1}{2}\,\Rew^{\ALG}_j.
%\]
\vspace{4mm}
\begin{center}
	\begin{minipage}{0.55\textwidth}
		\begin{equation*}%\label{eq:alpha-def-mini}
			\text{ $
				\alpha_{tj} := \begin{cases}
					\frac{1}{2}r_j\mathbf{1}[j\in V^{\ALG}_{t}] & \text{if } t\in U^{\OPT}_j \\ 
					0 & \text{otherwise}
				\end{cases}
				$}
		\end{equation*}
	\end{minipage}
	\hfill
	\begin{minipage}{0.40\textwidth}
		\begin{equation*}%\label{eq:beta-def-mini}
			\text{ $
				\beta_j := \tfrac{1}{2}r_j\,|U^{\ALG}_j|
				= \tfrac{1}{2}\,\Rew^{\ALG}_j.
				$}
		\end{equation*}
	\end{minipage}
\end{center}
\vspace{4mm}

Using these variables, we establish the following guarantee.

\begin{theorem}\label{thm:greedy-reuse}
Greedy is $\frac{1}{2}$-competitive for online matching with stochastic reusability.
\end{theorem}

We sketch the main idea behind the proof of this theorem. Note that the budget constraint follows directly from reward splitting: each successful match to resource $j$ contributes $r_j$ in total across $\alpha$ and $\beta$. In fact, the $\sum_{j\in V}\alpha_{tj}$ may be strictly less than $\tfrac{1}{2}r_{j^*}$ even when $\ALG$ matches $t$ to some available $j^*$, due to the additional indicator $\one[j\in V^{\ALG}_t]$.

The main issue is the covering constraint. Expanding the pseudo-reward of resource $j$ yields
\begin{equation}\label{eq:reuse-cover-expand}
	\begin{aligned}
		\sum_{t\in U} \alpha_{tj} + \beta_j= \frac{1}{2}r_j \bigg( \underbrace{\sum_{t \in U} \bm{1}[t \in U_j^{\OPT}, j \in V_t^{\ALG}]}_{\text{(I): $\OPT$ uses $j$ while available in $\ALG$}} + \underbrace{|U^{\ALG}_j|}_{\text{(II): $\ALG$ uses $j$}} \bigg).
	\end{aligned}
\end{equation}
Unlike the stochastic-rewards and patience settings, the two terms in \eqref{eq:reuse-cover-expand} are not complementary under independent sampling: a single long busy period in $\ALG$ can substantially reduce Term (I) while leaving Term (II) unchanged. To obtain a resource-by-resource lower bound, we therefore analyze $\ALG$ and $\OPT$ under a carefully chosen \emph{coupling} of their duration samples.

\begin{definition}[Duration coupling]\label{def:coupling}
	For each resource $j$, fix an infinite i.i.d.\ sequence (a ``common list'')
	\[
	D_j^{(1)}, D_j^{(2)}, D_j^{(3)},\ldots \stackrel{\text{i.i.d.}}{\sim} \mathcal{D}_j.
	\]
	Under the coupling, the $k$-th use of $j$ by $\OPT$ has duration $D_j^{(k)}$. Each time $\ALG$ uses $j$, it has duration equal to the first ``unused" element of the common list---skipping any element that has already been consumed by $\OPT$ (and $\ALG$ itself). This defines a joint distribution under which each algorithm's marginal duration sequence is i.i.d.\ with law $\mathcal{D}_j$.
\end{definition}

\begin{lemma}[Coupling inequality]\label{lem:coupling}
Under the coupling in Definition~\ref{def:coupling}, for every resource $j$ and every coupled sample path:
\begin{equation}\label{eq:coupling}
	\sum_{t \in U} 
	\one\{t \in U_j^{\OPT},\; j \in V_t^{\ALG}\}
	\;+\; |U_j^{\ALG}|
	\;\ge\; |U_j^{\OPT}|.
\end{equation}
\end{lemma}

\begin{proof}{Proof.}
We partition $U^{\OPT}_j$ into two disjoint subsets: arrivals $t$ at which $j\in V^{\ALG}_t$ (counted in Term (I)) and arrivals $t$ at which $j\notin V^{\ALG}_t$ (``missed'' uses of $j$ by $\OPT$). It suffices to injectively map each missed $\OPT$ use to a distinct use of $j$ by $\ALG$, which implies
\[
\bigl|\{t \in U_j^{\OPT} : j \notin V_t^{\ALG}\}\bigr|
\;\le\; |U^{\ALG}_j|.
\]

Consider a missed $\OPT$ use at arrival $t$. Then $j$ is unavailable in $\ALG$ at time $t$ because an earlier $\ALG$ use at some time $s<t$ is still blocking it. Under the coupling, $\OPT$ consumes the common-list samples in order, while $\ALG$ consumes the first unused sample at each use. Consequently, the duration sample used by $\OPT$ at $t$ must correspond to a common-list element that was previously consumed by $\ALG$ no later than time $s$. Map the missed $\OPT$ use at $t$ to that earlier $\ALG$ consumption event. Distinct missed $\OPT$ uses correspond to distinct common-list elements and therefore map to distinct $\ALG$ uses, yielding the required injection. We provide an example in Figure~\ref{fig:coupling-proof}.
\end{proof}

% %\begin{proof}{Proof of Theorem~\ref{thm:greedy-reuse}.}
% Fix resource $j$ and consider the coupled sample paths. Combining \eqref{eq:reuse-cover-expand} with Lemma~\ref{lem:coupling} shows that
% $\sum_{t\in U} \alpha_{tj} + \beta_j \ge \tfrac{1}{2}r_j\,|U^{\OPT}_j|$
% on every coupled sample path, which verifies the covering constraint for $j$ with $\tau=\tfrac{1}{2}$. As discussed above, the budget constraint holds by construction. Applying the LP-free certificate (Theorem~\ref{thm:form1}) completes the proof.
% \end{proof}

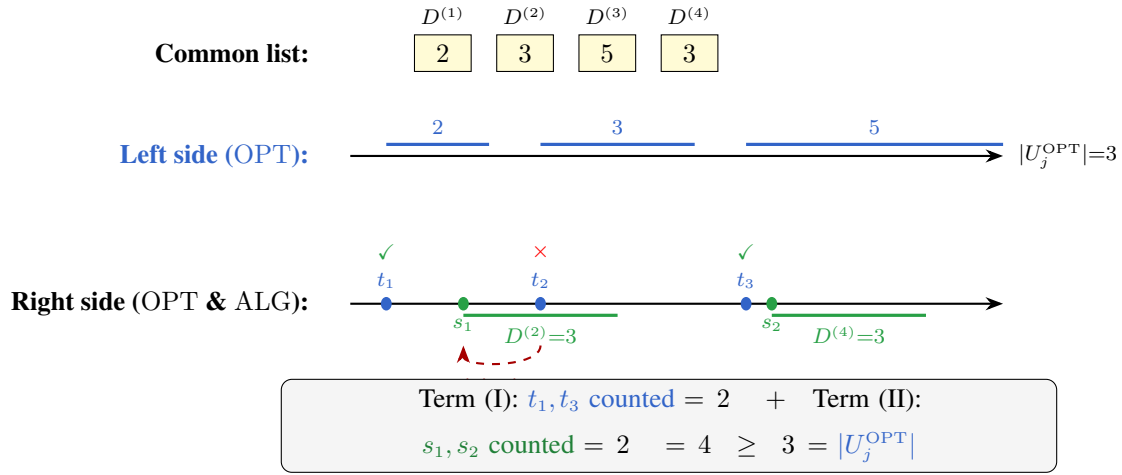
\begin{figure}
\centering
\begin{tikzpicture}[>=Stealth, xscale=0.68, yscale=0.85]
	% ---- Common list ----
	\node[font=\small, anchor=east] at (-0.3, 4.2) {\textbf{Common list:}};
	\foreach \i/\v in {1/2, 2/3, 3/5, 4/3} {
		\node[draw, fill=yellow!20, minimum width=0.75cm, 
		minimum height=0.45cm, font=\small] (D\i) at (\i*1.6+0.5, 4.2) {$\v$};
		\node[above=0pt of D\i, font=\scriptsize] {$D^{(\i)}$};
	}
	
	% ---- OPT timeline ----
	\node[font=\small, anchor=east, optblue] at (-0.3, 2.6) 
	{\textbf{Left side ($\OPT$):}};
	\draw[->, thick] (0.3,2.6) -- (13,2.6);
	
	% Intervals for OPT
	\draw[optblue, very thick] (1,2.78) -- (3,2.78);
	\node[above=0pt, font=\scriptsize, optblue] at (2,2.78) {$2$};
	\draw[optblue, very thick] (4,2.78) -- (7,2.78);
	\node[above=0pt, font=\scriptsize, optblue] at (5.5,2.78) {$3$};
	\draw[optblue, very thick] (8,2.78) -- (13,2.78);
	\node[above=0pt, font=\scriptsize, optblue] at (10.5,2.78) {$5$};
	\node[right, font=\scriptsize] at (13.1, 2.6) 
	{$|U^{\OPT}_j|{=}3$};
	
	% ---- ALG timeline ----
	\node[font=\small, anchor=east] at (-0.3, 0.3) {\textbf{Right side ($\OPT$ \& $\ALG$):}};
	\draw[->, thick] (0.3,0.3) -- (13,0.3);
	
	% ALG matches (s_i)
	\fill[alggreen] (2.5,0.3) circle (3pt);
	\node[below=2pt, font=\scriptsize, alggreen] at (2.5,0.3) {$s_1$};
	\draw[alggreen, very thick] (2.5,0.12) -- (5.5,0.12);
	\node[below=0pt, font=\scriptsize, alggreen] at (4,0.12) {$D^{(2)}{=}3$};
	
	\fill[alggreen] (8.5,0.3) circle (3pt);
	\node[below=2pt, font=\scriptsize, alggreen] at (8.5,0.3) {$s_2$};
	\draw[alggreen, very thick] (8.5,0.12) -- (11.5,0.12);
	\node[below=0pt, font=\scriptsize, alggreen] at (10,0.12) {$D^{(4)}{=}3$};
	
	% OPT arrival dots (t_i) moved to ALG line
	\fill[optblue] (1,0.3) circle (3pt);
	\node[above=2pt, font=\scriptsize, optblue] at (1,0.3) {$t_1$};
	
	\fill[optblue] (4,0.3) circle (3pt);
	\node[above=2pt, font=\scriptsize, optblue] at (4,0.3) {$t_2$};
	
	\fill[optblue] (8,0.3) circle (3pt);
	\node[above=2pt, font=\scriptsize, optblue] at (8,0.3) {$t_3$};
	
	% ---- Availability annotations on OPT matches ----
	\node[above=0.45cm, font=\scriptsize, alggreen] at (1,0.3) {\checkmark};
	\node[above=0.45cm, font=\scriptsize, red] at (4,0.3) {$\times$};
	\node[above=0.45cm, font=\scriptsize, alggreen] at (8,0.3) {\checkmark};
	
	% ---- Injection arrow ----
	\draw[->, thick, red!65!black, dashed] 
	(4, -0.35) to[out=-90,in=-90,looseness=0.6] 
	node[below=2pt, font=\scriptsize, pos=0.5]{injection} 
	(2.5, -0.35);
	
	% ---- Counting summary ----
	\node[draw, rounded corners, fill=gray!8, font=\small, 
	text width=10cm, align=center] at (6.5, -1.6) {
		Term (I): \textcolor{optblue}{$t_1,t_3$ counted} $= 2$
		\quad$+$\quad 
		Term (II): \textcolor{alggreen!80!black}{$s_1,s_2$ counted} $= 2$
		\quad$= 4 \;\ge\; 3 =$ 
		\textcolor{optblue}{$|U^{\OPT}_j|$}
	};
\end{tikzpicture}
\caption{Coupling visualization for proving \eqref{eq:coupling}. We use decomposition introduced in \eqref{eq:reuse-cover-expand}. At $t_1$ and $t_3$, resource $j$ is available in $\ALG$ (\textcolor{alggreen}{\checkmark}), explicitly contributing to Term (I) on the left hand side of \eqref{eq:coupling}. At $t_2$, $j$ is occupied (\textcolor{red}{$\times$}); the dashed injection maps this blocked match to $\ALG$'s active busy period started at $s_1$. The aggregate explicitly covers the right hand side.}
\label{fig:coupling-proof}
\end{figure}

%\subsection{Competitive Ratio}
% ...

\begin{remark}[{Further Applications}]
The LP-free framework was originally developed to analyze advanced policies that outperform Greedy in online matching with stochastic rewards \citep{goyal2023stochastic} and the online allocation of reusable resources \citep{goyal2025reusable}. Beyond the core settings explored here, the framework applies to a wide array of environments where fluid LP relaxations of the offline benchmark prove intractable or overly loose. For instance, it has been leveraged to analyze algorithms for platforms with multi-channel traffic \citep{manshadi2024multichannel}, two-sided assortment optimization \citep{aouad2023assortment}, Adwords with unknown advertiser budgets \citep{udwani2025adwords}, overbooking with delayed purchases \citep{fengsimple}, and constrained-pricing problems \citep{goyal2022pricing}. Notably, the framework's utility is not confined to stochastic environments. For example, in \citet{udwani2025adwords}, it facilitates the analysis of a randomized algorithm by offering a constraint system where a feasible solution is significantly easier to construct.
\end{remark}

%% file: tex/open.tex
The two parts of this tutorial point to a common theme: in adversarial online resource allocation, the main challenge is to identify the right certificate system for the model at hand. In classical deterministic models, an offline LP together with a fitted dual solution often suffices. In richer stochastic models, the harder question is whether one should strengthen the LP relaxation, abandon it altogether, or compare the online policy against the benchmark through a different family of certificates. Seen this way, Part~I and Part~II offer two complementary ways to organize a competitive-ratio proof: the LP-based convex-programming route works best when the benchmark has a clean packing relaxation and the right online state can be regularized effectively, whereas the LP-free route works best when post-allocation stochasticity makes the benchmark inherently dynamic or when natural LP relaxations are loose. The open problems below illustrate this broader agenda and the larger challenge of identifying, for each model, the simplest certificate system that still yields the right competitive ratio.

\smallskip
\xhdr{Integral algorithms from convex regularizers.}
All results in Part~I are stated for fractional algorithms. Can the convex-programming-based approach be extended to integral allocations? A concrete open problem comes from \citet{feng2024batching}, who use this technique to obtain a competitive ratio of $\Gamma(K)$ for the $K$-stage batch-arrival model (see \Cref{sec:batch,sec:ec-batch}) via a sequence of convex \emph{polynomial} regularizers. Can these regularizers be coupled with a rounding scheme that preserves the sharp factor $\Gamma(K)$? The fractional solution provides the right benchmark, but a matching integral guarantee would require controlling the interaction between stage-wise dependence and rounding loss.

\smallskip
\xhdr{Optimal ratio for stochastic rewards.} Is there an online algorithm with competitive ratio strictly larger than $1/2$ for adversarial online matching with stochastic rewards? Positive results above $1/2$ are known under additional structure, for example in regimes with homogeneous or vanishing probabilities \citep{mehta2012stochastic,mehta2014stochastic,huang2023online,huang2024online2,udwani2024stochastic, goyal2023stochastic}. The optimal ratio in the general adversarial model, however, remains open.

\smallskip
\xhdr{Beyond $\bm{1/2}$ for stochastic reusability.}
For models with stochastic reusability, the best-known competitive ratio in the general setting remains $1/2$. Notably, when resources are endowed with large initial capacities---a setting asymptotically equivalent to fractional allocation as capacities approach infinity---the optimal competitive ratio improves to $1-1/e$. This bound is achieved using the Balance algorithm for deterministic usage durations \citep{feng2021robustness}, while a distinct algorithmic approach is required for the general stochastic case \citep{goyal2025reusable}. It remains a prominent open question whether the general model admits a competitive ratio strictly greater than $1/2$. To date, the only established improvement beyond this $1/2$ threshold requires the restrictive assumption of deterministic and identical usage durations across all resources \citep{delong2024online}.

\smallskip
\xhdr{Sequential offerings, patience, and assortment.}
Can one beat the $1/2$ barrier in structured models that generalize stochastic rewards, such as matching with patience, ranking and sequential offerings, and online assortment optimization? Greedy-type policies achieve a $1/2$ guarantee across a broad family of such models, and the general hardness result of \citet{udwani2025optimality} shows that, unless NP=RP, no universal improvement is possible in the most general formulation, namely online submodular welfare maximization with stochastic outcomes. What remains unclear is what additional structure is sufficient to recover a larger ratio in canonical subclasses \citep{borodin2022prophet,brubach2025online,golrezaei2014real,aouad2023assortment}.

\smallskip
\xhdr{Learning-augmented competitive analysis.}
Suppose the algorithm receives predictions about future arrivals or prices. Can one adapt the regularizer online so that the algorithm achieves a better guarantee when the predictions are accurate, while degrading gracefully to the worst-case bound when they are not? In other words, can the convex-programming recipe be made robust to prediction error in the spirit of learning-augmented online algorithms? This question has been partially answered for special case of two-stage matching under batch arrival~\citep{jin2022online}, but it remains open more broadly.

\smallskip
\xhdr{Integrating convex-programming-based and LP-free templates.} %Settings arising in modern platforms may require a combination of both parts of this tutorial: rich adversarial structure on the one hand, and complex stochastic post-allocation effects on the other. 
Another compelling open question is identifying natural settings where a combination of both proof templates is fundamentally required, and finding out if such an integration can unlock novel performance guarantees for advanced algorithms operating beyond the Greedy paradigm.
%Many modern platform problems combine the ingredients of both parts of this tutorial: rich adversarial structure on the one hand, and stochastic post-allocation effects on the other. Examples include reusable resources with heterogeneous rewards, assortment problems with delayed stochastic conversion, and models in which cancellations, substitutions, or reuse interact with online prices. Can one develop a single proof template that interpolates between convex-programming-based dual fitting and LP-free certificates, while preserving the transparency and portability of both approaches?

% A central question in online allocation with stochastic outcomes is whether adaptivity---changing decisions in response to realized outcomes---improves worst-case competitive guarantees. Under a broad formulation of online submodular welfare maximization with stochastic outcomes, adaptivity does not yield a worst-case improvement: non-adaptive Greedy achieves the best possible competitive ratio among polynomial-time online algorithms (unless NP=RP) \citep{udwani2025optimality}. However, this negative result does not resolve the role of adaptivity in structured canonical subproblems. Understanding when, and by how much, one can exceed the $1/2$ barrier in such models remains a basic open direction.
	\section*{Acknowledgements}
	We are very grateful to the editorial team and anonymous reviewers for their insightful and constructive feedback.

%% file: tex/ec_intro.tex
\section{Electronic companion: additional models and deferred details}
\label{sec:ec-intro}

The main body develops the two core adversarial models in Part~I---online fractional matching and edge-weighted matching with free disposal---and keeps the technical discussion focused on the reusable design-and-analysis template introduced there. This electronic companion records the broader development behind that template. In particular, it collects the richer deterministic extensions that would otherwise interrupt the flow of the main exposition, together with the deferred proofs whose full details are useful for reference.

Throughout the electronic companion, the notation table in Table~\ref{tab:notation-main} remains in force. Any additional symbol is defined locally at first use. The guiding philosophy is also unchanged: identify the right state, regularize the post-decision state rather than the current reward, extract the dual update from the KKT system, and then choose the regularizer so that approximate dual feasibility collapses to a one-dimensional identity or recursion.

More specifically, the electronic companion contains the following pieces.
\begin{enumerate}[leftmargin=1.4em]
    \item \textbf{Batch arrival.} We study the $K$-stage arrival model, introduce the stage-dependent polynomial regularizers that yield the optimal factor $\GammaFunc{K}=1-(1-1/K)^K$, and record the recursive dual-fitting argument that replaces the scalar identity from the fully online setting (\Secref{sec:batch} and \Secref{sec:ec-batch}).
    \item \textbf{Configuration allocation / whole-page optimization.} We move from scalar or reward-indexed states to a price-level state vector and show how direct convex duality replaces the support-graph decomposition that was available in matching (\Secref{sec:configuration} and \Secref{sec:ec-config}).
    \item \textbf{Other extensions.} We explain how the same template specializes to AdWords, costly cancellations, and several nearby allocation models. The emphasis is on reusability of the main ideas rather than on developing a separate proof for every variant (\Secref{sec:extensions}).
\end{enumerate}

The companion is therefore meant to be read in two ways. A first pass can follow only the model statements and the high-level proof roadmaps. A second pass can then use the supplementary sections as a self-contained reference for the detailed primal--dual calculations.

%% file: tex/batch_arrival.tex
\section{Batch arrival: multi-stage fractional matching}
\label{sec:batch}

We next record the $K$-stage arrival model studied in \citet{feng2024batching}. Relative to the fully online setting of \Secref{sec:matching}, the main new issue is temporal heterogeneity: a fixed regularizer is no longer appropriate across the entire horizon. Early stages should hedge more aggressively because much of the future is still unknown, whereas late stages should behave more like the residual greedy problem. The resulting analysis is still primal--dual, but the scalar identity from the one-arrival model is replaced by a recursion across stages.

\subsection{How the universal recipe changes}

The five-step recipe from \Secref{sec:recipe} survives almost verbatim, with three modifications.
\begin{enumerate}
    \item The state remains the offline load vector, but it is updated stage by stage rather than arrival by arrival.
    \item The regularizer is stage dependent: the stage-$k$ program uses $F_k$ rather than a single function $F$.
    \item The KKT conditions are now written for an entire batch. As a result, the relevant structure is no longer the one-arrival water-filling picture, but a decomposition of the positive-support graph of the batch.
\end{enumerate}
The conceptual message is unchanged: the regularizer determines the online rule, and the KKT system determines the dual update.

\subsection{Stage-dependent polynomial regularizers}

The appropriate stage-wise regularizers are determined by a backward recursion.

\begin{proposition}[Optimal stage-dependent regularizers]
\label{prop:batch-polynomials}
Let $K\in\NN$. Define functions $f_1,\ldots,f_K$ on $[0,1]$ by
\[
    f_K(x)=\ind{x=1},
    \qquad
    f_k(x)=\max_{y\in[0,1-x]} (1-y)f_{k+1}(x+y)
    \qquad \forall k\in[K-1].
\]
Then, for every $k\in[K-1]$,
\[
    f_k(x)=\left(1-\frac{1-x}{K-k}\right)^{K-k}.
\]
Moreover:
\begin{enumerate}[label=(\roman*)]
    \item each $f_k$ is increasing on $[0,1]$, and therefore $F_k(x):=\int_0^x f_k(u)\,\dd u$ is differentiable and strictly convex;
    \item $f_k(x)\le \Expo^{x-1}$ for every $x\in[0,1]$;
    \item $f_1$ converges pointwise to $\Expo^{x-1}$ as $K\to\infty$;
    \item
    \[
        \GammaFunc{K}=1-\left(1-\frac{1}{K}\right)^K
        = \min_{y\in[0,1]}\bigl[1-(1-y)f_1(y)\bigr].
    \]
\end{enumerate}
\end{proposition}

The proof is given in \Secref{sec:ec-batch}. The degree of $f_k$ is $K-k$, so the regularization intensity decreases with the stage index. In particular, when $K$ is large and the batches are correspondingly small, the polynomial family approaches the exponential rule from the fully online model.

\subsection{The stage-wise algorithm}

Let $y_j^{(k)}$ denote the cumulative load on offline node $j$ after stage $k$, with $y_j^{(0)}=0$. After observing the batch $U_k$ and its edge set $E_k$, the algorithm solves
\begin{equation}
\begin{aligned}
    \max_{x_{tj}^{(k)}\ge 0}\quad
    & \sum_{(t,j)\in E_k} w_jx_{tj}^{(k)}
      - \sum_{j\in V} w_jF_k\!\left(y_j^{(k-1)}+
      \sum_{t:(t,j)\in E_k}x_{tj}^{(k)}\right) \\
    \text{s.t.}\quad
    & \sum_{j:(t,j)\in E_k} x_{tj}^{(k)}\le 1 \qquad \forall t\in U_k, \\
    & \sum_{t:(t,j)\in E_k} x_{tj}^{(k)}\le 1-y_j^{(k-1)} \qquad \forall j\in V.
\end{aligned}
\label{eq:batch-stage-program}
\end{equation}
and then updates the state by
\[
    y_j^{(k)}=y_j^{(k-1)}+\sum_{t:(t,j)\in E_k}x_{tj}^{(k)}.
\]
Here $x_{tj}^{(k)}$ is the fraction of batch vertex $t\in U_k$ assigned to offline node $j$ during stage $k$.

The unweighted case is a useful special case. There, the optimizer of \Eqref{eq:batch-stage-program} is simply batched water-filling: within a batch, mass is pushed toward the least loaded feasible offline vertices until one either exhausts the batch supply or saturates a subset of offline vertices. The polynomial family is therefore not needed to describe the primal rule in the unweighted model; it is needed to analyze that rule and to obtain the sharp factor $\GammaFunc{K}$.

\subsection{What the KKT conditions reveal}

Fix a stage $k\le K-1$, and let $x^{(k)}=\{x_{tj}^{(k)}\}$ denote the optimizer of \Eqref{eq:batch-stage-program}. Consider the positive-support graph
\[
    G_k' := \Bigl(U_k\cup V,\ \{(t,j)\in E_k: x_{tj}^{(k)}>0\}\Bigr).
\]
Let
\[
    \mathcal V_{k,0}:=\{j\in V: y_j^{(k)}=1\},
\]
and let $\mathcal U_{k,0}$ be the set of their neighbors in $G_k'$. Remove those vertices, and denote the connected components of the remaining support graph by $(\mathcal U_{k,\ell},\mathcal V_{k,\ell})$, $\ell=1,\ldots,L_k$.

\begin{proposition}[KKT decomposition of a batch]
\label{prop:batch-structure}
For every stage $k\le K-1$, there exist numbers
$0=c_{k,0}<c_{k,1}<\cdots<c_{k,L_k}$ such that:
\begin{enumerate}[label=(\roman*)]
    \item \textbf{Uniformity.} For every component $\ell\in\{0,1,\ldots,L_k\}$ and every $j,j'\in\mathcal V_{k,\ell}$,
    \[
        w_j\bigl(1-f_k(y_j^{(k)})\bigr)
        =w_{j'}\bigl(1-f_k(y_{j'}^{(k)})\bigr)
        = c_{k,\ell}.
    \]
    \item \textbf{Saturation.} Every online node in $\mathcal U_{k,\ell}$, $\ell\ge 1$, is fully matched in stage $k$.
    \item \textbf{Monotonicity.} If $c_{k,\ell}<c_{k,\ell'}$, then no edge of $E_k$ joins $\mathcal U_{k,\ell}$ to $\mathcal V_{k,\ell'}$.
\end{enumerate}
\end{proposition}

The proof is given in \Secref{sec:ec-batch}. This is the stage-wise analogue of the KKT picture from \Secref{sec:matching}. In the one-arrival model the KKT conditions describe how a single request splits its mass. In the batched model they describe how the support graph decomposes into components with a common marginal ``matchability level'' $c_{k,\ell}$.

\subsection{Dual construction and the scalar recursion}

The offline dual benchmark is still \Eqref{eq:offline-dual-basic}. Guided by \Propref{prop:batch-structure}, we define the fitted dual solution stage by stage. For stages $k\le K-1$, set
\[
    \widehat\alpha_t := c_{k,\ell}
    \qquad \text{for every } t\in\mathcal U_{k,\ell},
\]
and increment $\widehat\beta_j$ by
\begin{equation}
    \Delta\widehat\beta_j^{(k)}
    := w_j\bigl(y_j^{(k)}-y_j^{(k-1)}\bigr)f_k\bigl(y_j^{(k)}\bigr)
    \qquad \forall j\in V.
    \label{eq:batch-beta-increment-main}
\end{equation}
At the last stage, $F_K\equiv 0$, so the stage problem reduces to the residual maximum-weight matching LP. We therefore use an optimal dual solution of that residual LP for the final dual increment.

Approximate dual feasibility is governed by the scalar recursion
\begin{equation}
    1-(1-x)f_k(u_k+x)+\sum_{s=1}^{k-1}(u_{s+1}-u_s)f_s(u_{s+1})\ge \GammaFunc{K}
    \label{eq:batch-recursion-main}
\end{equation}
for every nondecreasing sequence $0=u_1\le \cdots\le u_k\le 1-x$. The proof is elementary but recursive, so we record it in \Secref{sec:ec-batch}. Equation \Eqref{eq:batch-recursion-main} is the stage-wise counterpart of the identity $1-f(y)+F(y)=\Gamma$ from the fully online model.

\begin{theorem}[Optimal competitive ratio under $K$-stage arrival]
\label{thm:batch-main}
For $K$-stage fractional vertex-weighted matching, the stage-wise algorithm defined by \Eqref{eq:batch-stage-program} is $\GammaFunc{K}$-competitive, where
\[
    \GammaFunc{K}=1-\left(1-\frac{1}{K}\right)^K.
\]
\end{theorem}

\begin{proof}{Proof.}
The theorem is proved in full in \Secref{sec:ec-batch}. The argument follows exactly the same structure as in the main body. First, the KKT decomposition in \Propref{prop:batch-structure} yields an online dual update for which the stage-wise dual increment matches the stage-wise primal reward exactly. Second, the scalar recursion \Eqref{eq:batch-recursion-main} implies approximate dual feasibility on every edge. Weak duality then gives the competitive ratio $\GammaFunc{K}$.\qed
\end{proof}

\begin{remark}
For $K=2$, the theorem gives the sharp ratio $\GammaFunc{2}=3/4$, already strictly larger than the fully online benchmark $1-1/\Expo$. As $K\to\infty$, $\GammaFunc{K}\downarrow 1-1/\Expo$, so the stage-wise model converges back to the online limit. The upper bound proved in \citet{feng2024batching} shows that this ratio is tight.
\end{remark}

%% file: tex/configuration.tex
\section{Configuration allocation}
\label{sec:configuration}

Configuration allocation is the most general deterministic model treated in Part~I. It contains vertex-weighted matching, AdWords, edge-weighted matching with free disposal, and the whole-page optimization model of \citet{devanur2016wholepage} as special cases. The same convex-programming viewpoint still applies, but the state is now multi-dimensional and the clean support-graph decomposition from matching is no longer available. The natural replacement is direct convex duality.

\subsection{Why this model is harder}

Three new issues appear at once.
\begin{enumerate}
    \item \textbf{Multiple price levels.} In matching, all mass assigned to a fixed offline node has the same value. In configuration allocation, advertiser $j$ may value different user--configuration pairs at very different prices.
    \item \textbf{No useful graph decomposition.} In batch matching, the KKT system induced a graph decomposition of the positive-support edges. Here an advertiser may be tight at some price thresholds and slack at others, so a graph-only description no longer captures the relevant structure.
    \item \textbf{Allocation and preemption are coupled.} At each stage, the algorithm must decide not only what new mass to keep, but also which previously retained mass to discard.
\end{enumerate}
These are precisely the reasons the analysis moves from combinatorial structure to direct convex duality.

\subsection{State, stage program, and interpretation of the variables}

Assume there is a finite price universe
\[
    0=\widehat w_0<\widehat w_1<\cdots<\widehat w_T.
\]
For each advertiser $j$ and price index $\tau$, let $\eta_{j,\tau}^{(k-1)}$ denote the amount of mass retained by advertiser $j$ at price level $\widehat w_\tau$ before stage $k$.

Once batch $U_k$ arrives, the stage-$k$ convex program chooses three families of variables.
\begin{center}
\small
\begin{tabularx}{0.97\linewidth}{>{\raggedright\arraybackslash}p{0.20\linewidth} X}
\toprule
Variable & Meaning \\
\midrule
$z_{t,c}$ & fraction of user $t\in U_k$ assigned to configuration $c$ \\
$x_{t,c,j}$ & amount of new mass from $(t,c)$ that advertiser $j$ still pays for \emph{after} stage $k$ \\
$y_{j,\tau}$ & amount of previously retained mass of advertiser $j$ at price level $\widehat w_\tau$ that survives stage $k$ \\
\bottomrule
\end{tabularx}
\end{center}
The input data are the prices $w_{t,c,j}$, the consumptions $\xi_{t,c,j}$, and the pre-stage state $\eta_{j,\tau}^{(k-1)}$.

For any candidate solution $(x,y,z)$, define the cumulative post-stage mass of advertiser $j$ above threshold $\widehat w_\tau$ by
\begin{equation}
    H_{j,\tau}(x,y)
    :=
    \sum_{\substack{t\in U_k,\,c\in\mathcal C:\
                     w_{t,c,j}\ge \widehat w_\tau}} x_{t,c,j}
    + \sum_{\tau'\ge \tau} y_{j,\tau'}.
    \label{eq:config-stage-cumulative}
\end{equation}
This is the price-level analogue of the cumulative state $Y_j(r)$ from \Secref{sec:free-disposal}.

The stage-wise regularized convex program is
\begin{subequations}
\begin{align}
    \max_{x,\,y,\,z\ge 0}\quad
    & \sum_{t\in U_k}\sum_{c\in\mathcal C}\sum_{j\in V} w_{t,c,j}x_{t,c,j} \nonumber \\
    &\qquad - \sum_{j\in V}\sum_{\tau\in[T]}
      \widehat w_\tau\bigl(\eta_{j,\tau}^{(k-1)}-y_{j,\tau}\bigr) \nonumber \\
    &\qquad - \sum_{j\in V}\sum_{\tau\in[T]}
      (\widehat w_\tau-\widehat w_{\tau-1})
      F_k\bigl(H_{j,\tau}(x,y)\bigr)
    \label{eq:config-stage-program}
\end{align}
subject to
\begin{align}
    \sum_{t\in U_k}\sum_{c\in\mathcal C} x_{t,c,j} + \sum_{\tau\in[T]} y_{j,\tau}
    &\le 1 && \forall j\in V,
    \label{eq:config-stage-capacity}\\
    \sum_{c\in\mathcal C} z_{t,c}
    &\le 1 && \forall t\in U_k,
    \label{eq:config-stage-choice}\\
    x_{t,c,j}
    &\le \xi_{t,c,j}z_{t,c} && \forall t\in U_k,\ \forall c\in\mathcal C,\ \forall j\in V,
    \label{eq:config-stage-link}\\
    y_{j,\tau}
    &\le \eta_{j,\tau}^{(k-1)} && \forall j\in V,\ \forall \tau\in[T].
    \label{eq:config-stage-carry}
\end{align}
\end{subequations}

The three lines of the objective have a direct interpretation: the first is the retained revenue from the new batch, the second subtracts the value of previously retained mass that is discarded in stage $k$, and the third regularizes the post-stage cumulative mass of each advertiser at every threshold. The state update is
\[
    \eta_{j,\tau}^{(k)}
    :=
    \sum_{\substack{t\in U_k,\,c\in\mathcal C:\
                     w_{t,c,j}=\widehat w_\tau}} x_{t,c,j}
    + y_{j,\tau},
\]
so $\eta^{(k)}$ is exactly the retained price distribution after stage $k$.

\subsection{Offline benchmark}

The standard offline LP benchmark is
\begin{equation}
\begin{aligned}
    \max \quad & \sum_{t\in U}\sum_{c\in\mathcal C}\sum_{j\in V} w_{t,c,j}x_{t,c,j} \\
    \text{s.t.}\quad
    & \sum_{t\in U}\sum_{c\in\mathcal C} x_{t,c,j} \le 1 && \forall j\in V, \\
    & \sum_{c\in\mathcal C} z_{t,c}\le 1 && \forall t\in U, \\
    & x_{t,c,j}\le \xi_{t,c,j}z_{t,c} && \forall t\in U,\ \forall c\in\mathcal C,\ \forall j\in V, \\
    & x_{t,c,j},z_{t,c}\ge 0.
\end{aligned}
\label{eq:offline-primal-config}
\end{equation}
Its dual is
\begin{equation}
\begin{aligned}
    \min \quad & \sum_{t\in U}\alpha_t + \sum_{j\in V}\beta_j \\
    \text{s.t.}\quad
    & \gamma_{t,c,j}+\beta_j \ge w_{t,c,j} && \forall t,c,j, \\
    & \alpha_t \ge \sum_{j\in V}\xi_{t,c,j}\gamma_{t,c,j} && \forall t,c, \\
    & \alpha_t,\beta_j,\gamma_{t,c,j}\ge 0.
\end{aligned}
\label{eq:offline-dual-config}
\end{equation}

\subsection{How the universal recipe specializes}

In this model, the five steps from \Secref{sec:recipe} become:
\begin{enumerate}
    \item \textbf{State:} the retained price-level vector $\eta^{(k-1)}$.
    \item \textbf{Regularized stage problem:} solve \Eqref{eq:config-stage-program}--\Eqref{eq:config-stage-carry}.
    \item \textbf{KKT signal:} read off the multipliers associated with the configuration, linking, and carry constraints.
    \item \textbf{Dual construction:} use the Lagrangian dual of the stage program directly rather than a graph decomposition.
    \item \textbf{Scalar inequality:} apply the same one-dimensional recursion from batch arrival, now threshold by threshold.
\end{enumerate}
The next two subsections make these points precise (with proofs and details in \Cref{sec:ec-config}).

\subsection{Structural lemmas from convex duality}

The first lemma records the KKT identities that later drive the offline-dual construction.

\begin{lemma}[Stage-wise KKT relations]
\label{lem:config-kkt-main}
Fix a stage $k\le K-1$, and let $(x^*,y^*,z^*)$ be the optimizer of \Eqref{eq:config-stage-program}--\Eqref{eq:config-stage-carry}. There exist nonnegative multipliers $\lambda_t^{(k)}$, $\mu_{t,c,j}^{(k)}$, $\psi_{t,c,j}^{(k)}$, and $\phi_{t,c}^{(k)}$ such that:
\begin{enumerate}[label=(\roman*)]
    \item complementary slackness holds for the configuration-choice and linking constraints;
    \item for every user--configuration pair $(t,c)$,
    \[
        \sum_{j\in V}\xi_{t,c,j}\mu_{t,c,j}^{(k)}-\lambda_t^{(k)}+\phi_{t,c}^{(k)}=0;
    \]
    \item for every triple $(t,c,j)$,
    \[
        w_{t,c,j}
        - \sum_{\tau:\,\widehat w_\tau\le w_{t,c,j}}
          (\widehat w_\tau-\widehat w_{\tau-1})f_k\bigl(H_{j,\tau}(x^*,y^*)\bigr)
        - \mu_{t,c,j}^{(k)} + \psi_{t,c,j}^{(k)} = 0.
    \]
\end{enumerate}
\end{lemma}

The next lemma identifies the preemption rule induced by optimality.

\begin{lemma}[Optimal preemption discards the lowest retained prices first]
\label{lem:config-dispose}
Fix $k\le K-1$, and let $(x^*,y^*,z^*)$ be the stage optimizer. If
\[
    \eta_{j,\tau}^{(k-1)}-y_{j,\tau}^*>0,
\]
then
\[
    H_{j,\tau}(x^*,y^*)=1.
\]
Equivalently, whenever some previously retained mass at price level $\widehat w_\tau$ is discarded, the retained post-stage state is already full at that threshold.
\end{lemma}

\begin{corollary}[Monotonicity of the price-level state]
\label{cor:config-monotone}
For every advertiser $j$, threshold $\tau$, and stage $k\le K-1$, we have:
\begin{enumerate}[label=(\roman*)]
    \item \textbf{Distribution monotonicity:}
    \[
        H_{j,\tau}(x^*,y^*)
        \ge \sum_{\tau'\ge \tau} \eta_{j,\tau'}^{(k-1)};
    \]
    \item \textbf{Only the smallest retained prices are discarded:}
    \[
        \sum_{\tau'\ge \tau}\bigl(\eta_{j,\tau'}^{(k-1)}-y_{j,\tau'}^*\bigr)
        =
        \sum_{\tau'\ge \tau}\bigl(\eta_{j,\tau'}^{(k-1)}-y_{j,\tau'}^*\bigr)
        f_k\bigl(H_{j,\tau}(x^*,y^*)\bigr).
    \]
\end{enumerate}
\end{corollary}

The detailed proofs are deferred to \Secref{sec:ec-config}. Conceptually, these statements play the same role as the graph decomposition in \Secref{sec:batch}: they explain how the state evolves and therefore how the advertiser-side dual mass should be charged.

\subsection{Dual construction and competitive ratio}

Recall the offline dual benchmark \Eqref{eq:offline-dual-config}. Guided by \Lemref{lem:config-kkt-main}, we construct a fitted dual solution stage by stage. For stages $k\le K-1$, set
\[
    \widehat\alpha_t := \lambda_t^{(k)},
    \qquad
    \widehat\gamma_{t,c,j}:=\mu_{t,c,j}^{(k)},
\]
and define the advertiser-side increment by
\begin{equation}
\begin{aligned}
    \Delta\widehat\beta_j^{(k)}
    :={}& \sum_{t\in U_k}\sum_{c\in\mathcal C} x_{t,c,j}^{(k)}
           \sum_{\tau:\,\widehat w_\tau\le w_{t,c,j}}
           (\widehat w_\tau-\widehat w_{\tau-1})f_k\bigl(H_{j,\tau}^{(k)}\bigr) \\
        &\qquad - \sum_{\tau\in[T]} \widehat w_\tau
           \bigl(\eta_{j,\tau}^{(k-1)}-y_{j,\tau}^{(k)}\bigr).
\end{aligned}
\label{eq:config-beta-increment-main}
\end{equation}
At the final stage, $F_K\equiv 0$, so the stage problem is linear and we use an optimal dual solution of the residual LP to define the last advertiser-side increment.

The important point is that approximate dual feasibility reduces to exactly the same scalar recursion as in \Secref{sec:batch}, but now one threshold at a time.

\begin{theorem}[Competitive ratio for $K$-stage configuration allocation]
\label{thm:config-main}
Let $f_1,\ldots,f_K$ be the polynomial family from \Propref{prop:batch-polynomials}. Then the stage-wise algorithm defined by \Eqref{eq:config-stage-program}--\Eqref{eq:config-stage-carry} is $\GammaFunc{K}$-competitive for $K$-stage fractional configuration allocation.
\end{theorem}

\begin{proof}{Proof.}
The full proof is given in \Secref{sec:ec-config}. The argument follows the same two-step pattern as before. First, the KKT multipliers in \Lemref{lem:config-kkt-main} yield a dual update for which the stage-wise increment of the fitted dual objective matches the stage-wise primal reward exactly. Second, \Lemref{lem:config-dispose} and \Corref{cor:config-monotone} reduce approximate dual feasibility to the scalar recursion from \Secref{sec:batch}, applied separately at each threshold. Weak duality for \Eqref{eq:offline-dual-config} then gives the factor $\GammaFunc{K}$.\qed
\end{proof}

\begin{remark}
This theorem subsumes the preceding adversarial models in a clean way. The reward-indexed state from \Secref{sec:free-disposal} reappears here after discretizing the reward axis into price levels, and the stage-dependent regularizers from \Secref{sec:batch} reappear unchanged. What disappears is the helpful support-graph decomposition; direct convex duality replaces it.
\end{remark}

%% file: tex/extensions.tex
\section{Short extensions and further directions}
\label{sec:extensions}
\label{sec:ec-special-cases}

This section collects nearby models that fit the same convex-programming viewpoint with only modest changes in state and scaling. The purpose is not to reproduce a full proof for each variant. Rather, the point is that once the state variable, the regularized one-step objective, and the KKT-based dual update are written down correctly, the earlier arguments can usually be reused with little additional work.

\subsection{AdWords / online budgeted allocation}
\label{subsec:adwords}
\label{subsec:ec-adwords-special}

The AdWords problem is the canonical budget-allocation analogue of online matching. Advertiser $j$ has budget $B_j$ and submits bid $b_{tj}$ for query $t$. The platform may assign at most one advertiser to each arriving query, and if $(t,j)$ is selected then advertiser $j$ pays $b_{tj}$ and the remaining budget decreases accordingly.

The right state variable is the normalized budget consumption
\[
  y_j := \frac{1}{B_j}\sum_{s \le t} b_{sj}x_{sj} \in [0,1].
\]
With this normalization, the arrival-wise regularized program is
\begin{equation}
\max_{x_{tj}\ge 0} \;
\sum_j b_{tj}x_{tj} - \sum_j B_j F\!\left(y_j + \frac{b_{tj}}{B_j}x_{tj}\right)
\quad \text{s.t.} \;
\sum_j x_{tj}\le 1, \;\;
 y_j+\frac{b_{tj}}{B_j}x_{tj}\le 1 \;\;\forall j.
\label{eq:adwords-program}
\end{equation}
Geometrically, this is the same object as vertex-weighted matching: the unit-capacity state $y_j$ is replaced by normalized budget load, and the load increment induced by query $t$ is $b_{tj}/B_j$ rather than $x_{tj}$. This is exactly the scaling that appears in the inventory-balancing view of AdWords \citep{mehta2007adwords,buchbinder2007online}. The fitted dual construction is still arrival-by-arrival: after solving the convex program, one fixes the online dual for the arriving query and increments the advertiser-side dual variables according to the normalized budget consumption.

\begin{proposition}[{\citet{mehta2007adwords,buchbinder2007online}}]
\label{prop:adwords-solution}
With the exponential regularizer, the algorithm defined by \eqref{eq:adwords-program} is $(1-1/\Expo)$-competitive for fractional online budgeted allocation in the small-bid regime.
\end{proposition}

The proof is the same dual-fitting argument as in \Secref{sec:matching}: the KKT system yields the online dual update, convexity of $F$ yields the lower bound on the advertiser-side dual variable, and the telescoping lower bound is written for normalized budget load. The only substantive change is that all loads are measured relative to budgets. In short, AdWords does not require a new primal--dual philosophy; it only requires the correct notion of load.

\begin{exercise}
\label{ex:adwords-kkt}
Derive the KKT conditions for \eqref{eq:adwords-program} and show that the fitted dual update takes the form $\widehat\alpha_t:=\alpha_t^*$ and $\widehat\beta_j\leftarrow \widehat\beta_j+x_{tj}^*(b_{tj}-\alpha_t^*)$. Verify that the analogue of the primal-gain-equals-dual-gain lemma from \Secref{sec:matching} holds with the same complementary-slackness argument.
\end{exercise}

\subsection{Costly cancellations}
\label{subsec:costly-cancel}

In the free-disposal model of \Secref{sec:free-disposal}, previously allocated mass may be discarded at no cost. In many applications, however, revoking an allocation incurs a buyback cost. Suppose that disposing of one unit of allocation at reward level $r$ costs $\kappa r$ for some buyback parameter $\kappa\ge 0$. Then the cancellation term in \Eqref{eq:free-program} is scaled by $(1+\kappa)$:
\[
\int_0^{\infty}
(\PreDensity{\OfflineIdx}{\RewardVar}-\PostDensity{\OfflineIdx}{\RewardVar})
\RewardVar\,\dd \RewardVar
\;\longrightarrow\;
(1+\kappa)\int_0^{\infty}
(\PreDensity{\OfflineIdx}{\RewardVar}-\PostDensity{\OfflineIdx}{\RewardVar})
\RewardVar\,\dd \RewardVar.
\]
When $\kappa=0$ we recover free disposal. As $\kappa$ increases, the algorithm becomes more reluctant to replace old allocations with new ones.

The convex-programming template extends directly. The stage program is unchanged except for the factor $(1+\kappa)$ in the disposal term, the KKT system yields the same style of threshold condition, and the dual-fitting proof again reduces to a one-dimensional identity. The effect of the buyback parameter is to decrease the advertiser-side dual increment and therefore the attainable competitive ratio. \citet{ekbatani2022cancellations} show that the optimal ratio varies continuously with $\kappa$, interpolating between $1-1/\Expo$ at $\kappa=0$ and $1/2$ as $\kappa\to\infty$.

\begin{exercise}
\label{ex:costly-threshold}
Show that with buyback parameter $\kappa>0$, the threshold from \Lemref{lem:free-threshold} becomes
$G(r)=(1+\kappa)r-\int_0^r f(Y_j^{(t)}(\omega))\dd\omega$.
Verify that $G$ remains nondecreasing when $f(y)=\Expo^{y-1}$, and explain why the threshold is higher than in the free-disposal case.
\end{exercise}

\begin{exercise}
\label{ex:costly-cr}
Assume $f(y)=\Expo^{y-1}$. Show that the slack-capacity case still yields the identity $\Gamma=1-f(y)+F(y)$, while the active-capacity case is weakened by the cancellation cost through the smaller $\beta$-increment. Recover the resulting competitive ratio as a function of $\kappa$; see \citet{ekbatani2022cancellations}.
\end{exercise}

%% file: tex/additional_special_cases.tex
\subsection{Vertex-weighted \texorpdfstring{$b$}{b}-matching}
\label{subsec:ec-bmatching}

Suppose offline node $j$ has capacity $B_j\ge 1$ rather than unit capacity, and each unit of that capacity is worth $w_j$. The offline LP benchmark is
\begin{equation}
\begin{aligned}
    \max \quad & \sum_{(t,j)\in E} w_j x_{tj} \\
    \text{s.t.}\quad
    & \sum_{j:(t,j)\in E} x_{tj} \le 1 && \forall t, \\
    & \sum_{t:(t,j)\in E} x_{tj} \le B_j && \forall j, \\
    & x_{tj}\ge 0.
\end{aligned}
\label{eq:ec-bmatching-offline}
\end{equation}
The right state variable is the normalized load
\[
  \bar y_j := \frac{1}{B_j}\sum_{s:(s,j)\in E} x_{sj} \in [0,1].
\]
After this normalization, the fully online regularized program becomes
\begin{equation}
\begin{aligned}
    \max_{x_{tj}\ge 0}\quad
    & \sum_{j\in N(t)} w_j x_{tj}
      - \sum_{j\in V} B_j w_j
        F\!\left(\bar y_j + \frac{x_{tj}}{B_j}\right) \\
    \text{s.t.}\quad
    & \sum_{j\in N(t)} x_{tj} \le 1, \\
    & x_{tj} \le B_j(1-\bar y_j) \qquad \forall j\in N(t).
\end{aligned}
\label{eq:ec-bmatching-program}
\end{equation}
The factor $B_j$ in front of the regularizer is exactly what preserves the geometry of the unit-capacity model: the penalty is applied to the fraction of capacity consumed rather than to the raw amount of consumed capacity.

The same normalization works in the batch-arrival model after replacing $F$ with the stage-dependent family $F_k$ and replacing $\bar y_j$ by the normalized pre-stage load. Every statement from the matching sections then carries over with $y_j$ replaced by $y_j/B_j$.

\begin{proposition}
\label{prop:ec-bmatching}
In fractional vertex-weighted $b$-matching, the convex-regularization recipe yields the same optimal competitive ratios as in the unit-capacity case after normalizing loads by $B_j$. In particular, the exponential regularizer gives the classical $(1-1/\Expo)$ guarantee in the fully online model, and the polynomial family from \Secref{sec:batch} gives the factor $\Gamma(K)=1-(1-1/K)^K$ in the $K$-stage model.
\end{proposition}

The integral problem can then be handled by rounding. Under standard large-capacity assumptions, independent rounding of the fractional solution after a slight capacity reduction gives the usual loss of order $O(\sqrt{\log B_{\min}/B_{\min}})$, where $B_{\min}:=\min_j B_j$.

\subsection{Choice-based assortment and configuration views}
\label{subsec:ec-assortment-special}

The configuration-allocation model also captures assortment-style decisions. One may think of an arriving user $t$ as a customer and a feasible configuration $c\in\mathcal C_t$ as an offered assortment, ranked menu, or page layout. The coefficients $\xi_{t,c,j}$ specify how much resource $j$ is consumed when configuration $c$ is chosen, while $w_{t,c,j}$ specifies the associated revenue contribution. Under this interpretation, the stage-wise program from \Secref{sec:configuration} is already an online assortment problem with resource constraints.

The modeling advantage is that one does not need a separate theorem for every such variant. Once an online assortment or whole-page decision can be written as ``choose one configuration from a feasible family,'' the same price-level state, the same regularized stage program, and the same KKT-based dual construction apply. This is one reason the configuration model is a natural endpoint for Part~I: many seemingly different online-allocation problems become different faces of the same convex-programming template.

For representative examples of this broader viewpoint, see \citet{devanur2016wholepage} for whole-page optimization and \citet{aouad2023assortment} for online assortment optimization on two-sided platforms.

%% file: tex/ec.tex
\section{Supplementary proofs for batch arrival}
\label{sec:ec-batch}

For this appendix only, it is convenient to index the stage loads one step earlier than in the main text. For a
fixed run of the algorithm in \Secref{sec:batch}, write
\[
    y_j^{(1)} := 0,
    \qquad
    y_j^{(k+1)} := y_j^{(k)} + \sum_{t:(t,j)\in E_k} x_{tj}^{(k)}
    \qquad (k=1,\ldots,K),
\]
where \(x^{(k)}=\{x_{tj}^{(k)}\}\) is the optimizer of the stage-\(k\) convex program. Thus \(y_j^{(k)}\) is the load
on offline node \(j\) \emph{before} stage \(k\), while \(y_j^{(k+1)}\) is the load after stage \(k\).

\subsection{Proof of \Propref{prop:batch-polynomials}}

\begin{proof}{Proof of \Propref{prop:batch-polynomials}.}
Fix a stage \(k\le K-1\), and write \(n:=K-k\ge 1\). Then
\[
    f_k(y)=\left(1-\frac{1-y}{n}\right)^n.
\]
Differentiating gives
\[
    f_k'(y)=\left(1-\frac{1-y}{n}\right)^{n-1}\ge 0,
\]
so \(f_k\) is increasing on \([0,1]\). Since \(F_k'(y)=f_k(y)\) and \(f_k\) is nonconstant and increasing,
\(F_k\) is differentiable and strictly convex.

Next, set \(z=(1-y)/n\in[0,1]\). The elementary bound \(1-z\le \Expo^{-z}\) yields
\[
    f_k(y)=\bigl(1-z\bigr)^n \le \Expo^{-nz}=\Expo^{y-1},
\]
which proves the domination by the exponential rule.

For the pointwise convergence, let \(k=1\) and write \(n=K-1\). Then
\[
    f_1(y)=\left(1-\frac{1-y}{n}\right)^n \xrightarrow[n\to\infty]{} \Expo^{y-1}
\]
by the standard limit \((1-a/n)^n\to \Expo^{-a}\).

Finally, define
\[
    g(y):=(1-y)f_1(y)
    = (1-y)\left(1-\frac{1-y}{K-1}\right)^{K-1}.
\]
With the change of variable \(u=1-y\), this becomes
\[
    g(u)=u\left(1-\frac{u}{K-1}\right)^{K-1},
    \qquad u\in[0,1].
\]
Differentiating,
\[
    g'(u)=\left(1-\frac{u}{K-1}\right)^{K-2}\left(1-\frac{Ku}{K-1}\right).
\]
Hence \(g\) is maximized at \(u=(K-1)/K\), i.e., at \(y=1/K\). Substituting gives
\[
    \max_{y\in[0,1]} (1-y)f_1(y)
    = \frac{K-1}{K}\left(1-\frac{1}{K}\right)^{K-1}
    = \left(1-\frac{1}{K}\right)^K.
\]
The last display is equivalent to
\[
    \GammaFunc{K} = 1-\left(1-\frac{1}{K}\right)^K
    = \min_{y\in[0,1]} \bigl[1-(1-y)f_1(y)\bigr],
\]
which is the required formula.\qed
\end{proof}

The competitive-ratio proof relies on the following one-dimensional recursive inequality.

\begin{lemma}
\label{lem:batch-recursion}
Fix \(K\in\NN\), \(k\in\{1,\ldots,K-1\}\), and \(x\in[0,1]\). For every nondecreasing sequence
\[
    0=u_1 \le u_2 \le \cdots \le u_k \le 1-x,
\]
we have
\[
    (1-x)f_k(u_k+x)
    - \sum_{s=1}^{k-1} (u_{s+1}-u_s) f_s(u_{s+1})
    \le 1-\GammaFunc{K}.
\]
Equivalently,
\[
    1-(1-x)f_k(u_k+x)
    + \sum_{s=1}^{k-1} (u_{s+1}-u_s) f_s(u_{s+1})
    \ge \GammaFunc{K}.
\]
\end{lemma}

\begin{proof}{Proof.}
We argue by induction on \(k\).

For \(k=1\), the claim reduces to
\[
    (1-x)f_1(x)\le 1-\GammaFunc{K},
\]
which is exactly the maximizing identity proved above with \(y=x\).

Assume now that the statement holds for \(k-1\), and consider stage \(k\). By the recursive definition of the
polynomials,
\[
    f_{k-1}(u_k) = \max_{y\in[0,1-u_k]} (1-y)f_k(u_k+y).
\]
Choosing \(y=x\) is feasible because \(u_k\le 1-x\), hence
\[
    (1-x)f_k(u_k+x) \le f_{k-1}(u_k).
\]
Therefore
\begin{align*}
    (1-x)f_k(u_k+x) - \sum_{s=1}^{k-1} (u_{s+1}-u_s)f_s(u_{s+1})
    &\le f_{k-1}(u_k) - \sum_{s=1}^{k-1} (u_{s+1}-u_s)f_s(u_{s+1}) \\
    &= \bigl(1-u_k+u_{k-1}\bigr)f_{k-1}(u_k)
       - \sum_{s=1}^{k-2} (u_{s+1}-u_s)f_s(u_{s+1}).
\end{align*}
Now apply the induction hypothesis to the shorter sequence
\[
    0=u_1\le u_2\le \cdots \le u_{k-1}
\]
with the choice \(x' := u_k-u_{k-1}\). Since \(u_k\le 1\), we have \(u_{k-1}\le 1-x'\), so the induction
hypothesis yields the desired bound.
\qed
\end{proof}

\subsection{Proof of \Propref{prop:batch-structure}}

\begin{proof}{Proof of \Propref{prop:batch-structure}.}
Fix a stage \(k\le K-1\), and abbreviate the optimizer of \Eqref{eq:batch-stage-program} by \(x=\{x_{tj}\}\).
Introduce Lagrange multipliers \(\lambda_t\ge 0\) for the online row constraints,
\(\theta_j\ge 0\) for the offline capacity constraints, and \(\gamma_{tj}\ge 0\) for nonnegativity. The KKT system is
\begin{equation}
    w_j\Bigl(1-f_k\bigl(y_j^{(k)}+\sum_{t:(t,j)\in E_k}x_{tj}\bigr)\Bigr)
    - \lambda_t - \theta_j + \gamma_{tj} = 0
    \qquad \forall (t,j)\in E_k,
    \label{eq:ec-batch-kkt}
\end{equation}
together with complementary slackness:
\begin{align}
    x_{tj}>0 &\implies \gamma_{tj}=0,
    \label{eq:ec-batch-cs-gamma}\\
    y_j^{(k)}+\sum_{t:(t,j)\in E_k}x_{tj}<1 &\implies \theta_j=0,
    \label{eq:ec-batch-cs-theta}\\
    \sum_{j:(t,j)\in E_k}x_{tj}<1 &\implies \lambda_t=0.
    \label{eq:ec-batch-cs-lambda}
\end{align}

Let \(G_k'\) be the positive-support graph formed by the edges with \(x_{tj}>0\). Define
\[
    \mathcal V_{k,0}:=\left\{j\in V: y_j^{(k)}+\sum_{t:(t,j)\in E_k}x_{tj}=1\right\}
\]
and let \(\mathcal U_{k,0}\) be the set of neighbors of \(\mathcal V_{k,0}\) inside \(G_k'\). Remove these vertices and let
\((\mathcal U_{k,\ell},\mathcal V_{k,\ell})\), \(\ell=1,\ldots,L_k\), be the connected components of the remaining support
subgraph.

\paragraph{Uniformity.}
If \(j\in\mathcal V_{k,0}\), then its post-stage load is one, so
\[
    w_j\bigl(1-f_k(y_j^{(k+1)})\bigr)=w_j(1-f_k(1))=0.
\]
Set \(c_{k,0}:=0\).

Now fix \(\ell\ge 1\) and two offline vertices \(j,j'\in\mathcal V_{k,\ell}\) that share an online neighbor
\(t\in\mathcal U_{k,\ell}\) in the positive-support graph. Since \(x_{tj}>0\) and \(x_{tj'}>0\), we have
\(\gamma_{tj}=\gamma_{tj'}=0\) by \Eqref{eq:ec-batch-cs-gamma}. Also, because
\(j,j'\notin\mathcal V_{k,0}\), both are unsaturated and \Eqref{eq:ec-batch-cs-theta} gives
\(\theta_j=\theta_{j'}=0\). Applying \Eqref{eq:ec-batch-kkt} to \((t,j)\) and \((t,j')\) yields
\[
    w_j\bigl(1-f_k(y_j^{(k+1)})\bigr)
    = \lambda_t
    = w_{j'}\bigl(1-f_k(y_{j'}^{(k+1)})\bigr).
\]
Since any two vertices in the same connected component can be joined by a path, the quantity above is constant on
\(\mathcal V_{k,\ell}\); call the common value \(c_{k,\ell}\).

\paragraph{Saturation.}
Fix \(\ell\ge 1\) and \(t\in\mathcal U_{k,\ell}\). Choose a neighbor \(j\in\mathcal V_{k,\ell}\) with \(x_{tj}>0\).
Because \(j\notin\mathcal V_{k,0}\), we have \(y_j^{(k+1)}<1\), and thus
\(f_k(y_j^{(k+1)})<1\). Hence
\[
    \lambda_t = w_j\bigl(1-f_k(y_j^{(k+1)})\bigr)=c_{k,\ell}>0.
\]
Complementary slackness \Eqref{eq:ec-batch-cs-lambda} therefore implies
\(\sum_{j:(t,j)\in E_k}x_{tj}=1\).

\paragraph{Monotonicity.}
Suppose there were an edge \((t,j')\in E_k\) with \(t\in\mathcal U_{k,\ell}\) and
\(j'\in\mathcal V_{k,\ell'}\). Choose \(j\in\mathcal V_{k,\ell}\) such that \(x_{tj}>0\); then
\(\gamma_{tj}=0\). Applying \Eqref{eq:ec-batch-kkt} to \((t,j)\) and \((t,j')\),
\[
    c_{k,\ell} - \lambda_t - \theta_j = 0,
    \qquad
    c_{k,\ell'} - \lambda_t - \theta_{j'} + \gamma_{tj'} = 0.
\]
Thus
\[
    c_{k,\ell'} = \lambda_t + \theta_{j'} - \gamma_{tj'} \le \lambda_t \le c_{k,\ell}
\]
whenever \(\ell'\ge 1\), because then \(\theta_{j'}=0\). If \(\ell'=0\), then \(c_{k,0}=0\) anyway. Consequently,
whenever \(c_{k,\ell}<c_{k,\ell'}\), no edge can join \(\mathcal U_{k,\ell}\) to \(\mathcal V_{k,\ell'}\).\qed
\end{proof}

\subsection{Proof of \Thmref{thm:batch-main}}

\begin{proof}{Proof of \Thmref{thm:batch-main}.}
Let \(x^{(k)}\) be the optimizer of the stage-\(k\) convex program for each \(k=1,\ldots,K\), and let
\((\mathcal U_{k,\ell},\mathcal V_{k,\ell})\) and \(c_{k,\ell}\) be the decomposition from
\Propref{prop:batch-structure} for every \(k\le K-1\). Consider the weighted offline dual
\begin{equation}
\begin{aligned}
    \min \quad & \sum_{t\in U} \alpha_t + \sum_{j\in V} \beta_j \\
    \text{s.t.}\quad
    & \alpha_t + \beta_j \ge w_j && \forall (t,j)\in E, \\
    & \alpha_t,\beta_j\ge 0.
\end{aligned}
\label{eq:ec-batch-offline-dual}
\end{equation}
We construct a feasible scaled dual solution.

For stages \(k\le K-1\), set
\[
    \widehat\alpha_t := c_{k,\ell}
    \qquad \text{for every } t\in \mathcal U_{k,\ell},
\]
and define the stage-wise increment of \(\widehat\beta_j\) by
\begin{equation}
    \Delta \widehat\beta_j^{(k)}
    := \bigl(y_j^{(k+1)}-y_j^{(k)}\bigr)\bigl(w_j-c_{k,\ell}\bigr)
    = w_j\bigl(y_j^{(k+1)}-y_j^{(k)}\bigr)f_k(y_j^{(k+1)})
    \qquad \text{for } j\in \mathcal V_{k,\ell}.
    \label{eq:ec-batch-beta-increment}
\end{equation}
At the last stage \(K\), the regularizer vanishes, so the stage problem is just the residual maximum-weight
fractional matching problem with capacities \(1-y_j^{(K)}\). Let
\((\bar\alpha_t)_{t\in U_K},(\bar\beta_j)_{j\in V}\) be an optimal dual solution of that residual LP, i.e.,
\[
    \bar\alpha_t + \bar\beta_j \ge w_j
    \qquad \forall (t,j)\in E_K,
\]
and
\[
    \sum_{t\in U_K}\bar\alpha_t + \sum_{j\in V}(1-y_j^{(K)})\bar\beta_j
    = \sum_{(t,j)\in E_K} w_j x_{tj}^{(K)}.
\]
Set
\[
    \widehat\alpha_t := \bar\alpha_t \quad (t\in U_K),
    \qquad
    \Delta \widehat\beta_j^{(K)} := (1-y_j^{(K)})\bar\beta_j,
\]
and define \(\widehat\beta_j:=\sum_{k=1}^K \Delta\widehat\beta_j^{(k)}\).

\paragraph{Primal and dual objective increments agree.}
Fix a stage \(k\le K-1\). Since only edges inside the support graph contribute, it is enough to compare the primal
and dual contributions componentwise.

If \(\ell=0\), then \(c_{k,0}=0\), and hence
\[
    \sum_{t\in\mathcal U_{k,0}} \widehat\alpha_t
    + \sum_{j\in\mathcal V_{k,0}} \Delta\widehat\beta_j^{(k)}
    = \sum_{j\in\mathcal V_{k,0}} \sum_{t:(t,j)\in E_k} w_j x_{tj}^{(k)},
\]
which is exactly the primal reward collected on that block.

If \(\ell\ge 1\), then by saturation from \Propref{prop:batch-structure},
\(\sum_{j:(t,j)\in E_k}x_{tj}^{(k)}=1\) for every \(t\in\mathcal U_{k,\ell}\). Therefore
\begin{align*}
    \sum_{t\in\mathcal U_{k,\ell}}\widehat\alpha_t
    + \sum_{j\in\mathcal V_{k,\ell}} \Delta\widehat\beta_j^{(k)}
    &= |\mathcal U_{k,\ell}|\, c_{k,\ell}
       + \sum_{j\in\mathcal V_{k,\ell}}
         \bigl(y_j^{(k+1)}-y_j^{(k)}\bigr)\bigl(w_j-c_{k,\ell}\bigr) \\
    &= \sum_{j\in\mathcal V_{k,\ell}} w_j\bigl(y_j^{(k+1)}-y_j^{(k)}\bigr)
       + c_{k,\ell}
         \left(|\mathcal U_{k,\ell}| - \sum_{j\in\mathcal V_{k,\ell}}(y_j^{(k+1)}-y_j^{(k)})\right) \\
    &= \sum_{j\in\mathcal V_{k,\ell}} w_j\bigl(y_j^{(k+1)}-y_j^{(k)}\bigr).
\end{align*}
The last equality uses conservation of mass inside the component: every edge carrying positive mass from a node in $\mathcal U_{k,\ell}$ ends in $\mathcal V_{k,\ell}$, and every node in $\mathcal U_{k,\ell}$ is fully matched. Therefore
\[
    \sum_{j\in\mathcal V_{k,\ell}}\bigl(y_j^{(k+1)}-y_j^{(k)}\bigr)
    = \sum_{t\in\mathcal U_{k,\ell}}\sum_{j\in\mathcal V_{k,\ell}} x_{tj}^{(k)}
    = |\mathcal U_{k,\ell}|.
\]
So the displayed quantity is exactly the primal reward collected on that block. At stage \(K\), equality of the primal and dual
objectives follows from the optimality of the residual dual solution \((\bar\alpha,\bar\beta)\). Summing over all
stages gives
\[
    \sum_{t\in U} \widehat\alpha_t + \sum_{j\in V}\widehat\beta_j = \ALG.
\]

\paragraph{Approximate dual feasibility.}
We show that every edge satisfies
\[
    \widehat\alpha_t + \widehat\beta_j \ge \GammaFunc{K} \, w_j.
\]

Fix an edge \((t,j)\in E_k\) with \(k\le K-1\). Suppose
\(t\in\mathcal U_{k,\ell}\) and \(j\in\mathcal V_{k,\ell'}\). By monotonicity,
\(c_{k,\ell'}\le c_{k,\ell}=\widehat\alpha_t\). Hence, using \Eqref{eq:ec-batch-beta-increment},
\begin{align*}
    \widehat\alpha_t + \sum_{s=1}^k \Delta\widehat\beta_j^{(s)}
    &\ge c_{k,\ell'} + \Delta\widehat\beta_j^{(k)} + \sum_{s=1}^{k-1} \Delta\widehat\beta_j^{(s)} \\
    &= w_j\bigl(1-f_k(y_j^{(k+1)})\bigr)
       + w_j\bigl(y_j^{(k+1)}-y_j^{(k)}\bigr)f_k(y_j^{(k+1)})
       + \sum_{s=1}^{k-1} \Delta\widehat\beta_j^{(s)} \\
    &= w_j\left[
        1 - \bigl(1-y_j^{(k+1)}+y_j^{(k)}\bigr)f_k(y_j^{(k+1)})
        + \sum_{s=1}^{k-1} \bigl(y_j^{(s+1)}-y_j^{(s)}\bigr)f_s(y_j^{(s+1)})
       \right].
\end{align*}
Apply \Lemref{lem:batch-recursion} with
\[
    x := y_j^{(k+1)}-y_j^{(k)},
    \qquad
    u_s := y_j^{(s)} \quad (s=1,\ldots,k),
\]
and note that \(u_k=y_j^{(k)}\) and \(u_k+x=y_j^{(k+1)}\). This yields
\[
    \widehat\alpha_t + \sum_{s=1}^k \Delta\widehat\beta_j^{(s)} \ge \GammaFunc{K} w_j.
\]
Since the remaining increments of \(\widehat\beta_j\) are nonnegative, the full dual constraint is satisfied.

Now consider an edge \((t,j)\in E_K\) from the last stage. By construction,
\[
    \widehat\alpha_t + \Delta\widehat\beta_j^{(K)}
    = \bar\alpha_t + (1-y_j^{(K)})\bar\beta_j
    \ge (1-y_j^{(K)}) w_j.
\]
If \(y_j^{(K)}=0\), then this is already at least \(\GammaFunc{K}w_j\). Otherwise, let \(k'<K\) be the last stage with
positive increment on node \(j\), so \(y_j^{(K)}=y_j^{(k'+1)}\). Since
\(f_{k'}(y)=\left(1-\frac{1-y}{K-k'}\right)^{K-k'}\ge y\) for every \(y\in[0,1]\) by Bernoulli's inequality,
\[
    (1-y_j^{(K)})w_j \ge \bigl(1-f_{k'}(y_j^{(k'+1)})\bigr)w_j.
\]
Therefore
\begin{align*}
    \widehat\alpha_t + \widehat\beta_j
    &\ge (1-y_j^{(K)})w_j + \sum_{s=1}^{k'} \Delta\widehat\beta_j^{(s)} \\
    &\ge \bigl(1-f_{k'}(y_j^{(k'+1)})\bigr)w_j + \sum_{s=1}^{k'} \Delta\widehat\beta_j^{(s)} \\
    &= w_j\left[
       1 - \bigl(1-y_j^{(k'+1)}+y_j^{(k')}\bigr)f_{k'}(y_j^{(k'+1)})
       + \sum_{s=1}^{k'-1} \bigl(y_j^{(s+1)}-y_j^{(s)}\bigr)f_s(y_j^{(s+1)})
      \right] \\
    &\ge \GammaFunc{K}w_j
\end{align*}
by one more application of \Lemref{lem:batch-recursion}. This proves approximate dual feasibility. Weak duality
now gives \(\ALG\ge \GammaFunc{K}\OPT\).\qed
\end{proof}

\section{Supplementary proofs for configuration allocation}
\label{sec:ec-config}

Throughout this section, write
\[
    \Delta w_\tau := \PriceLevel{\tau} - \PriceLevel{\tau-1}
\]
and, for each stage \(k\), let
\[
    H_{j,\tau}^{(k)} := H_{j,\tau}(x^{(k)},y^{(k)})
\]
for the cumulative price-level state of advertiser \(j\) at threshold \(\PriceLevel{\tau}\). We also set
\(H_{j,\tau}^{(0)}:=0\).

\subsection{Proofs of the structural lemmas}

We begin with the offline benchmark. The K-stage fractional configuration-allocation LP is
\begin{equation}
\begin{aligned}
    \max \quad & \sum_{t\in U}\sum_{c\in\mathcal C}\sum_{j\in V} w_{t,c,j} x_{t,c,j} \\
    \text{s.t.}\quad
    & \sum_{t\in U}\sum_{c\in\mathcal C} x_{t,c,j} \le 1 && \forall j\in V, \\
    & \sum_{c\in\mathcal C} z_{t,c} \le 1 && \forall t\in U, \\
    & x_{t,c,j} \le \xi_{t,c,j} z_{t,c} && \forall t,c,j, \\
    & x_{t,c,j}, z_{t,c} \ge 0.
\end{aligned}
\label{eq:offline-config-lp}
\end{equation}
Its dual is
\begin{equation}
\begin{aligned}
    \min \quad & \sum_{t\in U} \alpha_t + \sum_{j\in V} \beta_j \\
    \text{s.t.}\quad
    & \gamma_{t,c,j} + \beta_j \ge w_{t,c,j} && \forall t,c,j, \\
    & \alpha_t \ge \sum_{j\in V} \xi_{t,c,j}\gamma_{t,c,j} && \forall t,c, \\
    & \alpha_t,\beta_j,\gamma_{t,c,j}\ge 0.
\end{aligned}
\label{eq:offline-config-dual}
\end{equation}

For the stage-wise analysis, fix a stage \(k\le K-1\), and let \((x^*,y^*,z^*)\) be the optimizer of
\Eqref{eq:config-stage-program}--\Eqref{eq:config-stage-carry}. Introduce Lagrange multipliers
\(\lambda_t\ge 0\) for the configuration-choice constraints,
\(\theta_j\ge 0\) for the advertiser-capacity constraints,
\(\mu_{t,c,j}\ge 0\) for the link constraints,
\(\pi_{j,\tau}\ge 0\) for the carry constraints,
\(\psi_{t,c,j}\ge 0\) for nonnegativity of \(x_{t,c,j}\),
\(\phi_{t,c}\ge 0\) for nonnegativity of \(z_{t,c}\), and
\(\iota_{j,\tau}\ge 0\) for nonnegativity of \(y_{j,\tau}\).

\begin{proof}{Proof of \Lemref{lem:config-kkt-main}.}
The KKT system consists of primal feasibility, dual feasibility, complementary slackness, and stationarity. The complementary-slackness relations needed later are
\begin{align}
    \lambda_t\Bigl(\sum_{c\in\mathcal C} z_{t,c}^*-1\Bigr) &= 0,
    \label{eq:ec-config-cs-lambda}\\
    \mu_{t,c,j}\bigl(x_{t,c,j}^*-\xi_{t,c,j}z_{t,c}^*\bigr) &= 0,
    \label{eq:ec-config-cs-mu}\\
    \pi_{j,\tau}\bigl(y_{j,\tau}^*-\eta_{j,\tau}^{(k-1)}\bigr) &= 0,
    \qquad
    \iota_{j,\tau}y_{j,\tau}^* = 0,
    \label{eq:ec-config-cs-pi-iota}\\
    \psi_{t,c,j}x_{t,c,j}^* &= 0,
    \qquad
    \phi_{t,c}z_{t,c}^*=0.
    \label{eq:ec-config-cs-psi-phi}
\end{align}
The stationarity relations with respect to \(z_{t,c}\) and \(x_{t,c,j}\) are
\begin{align}
    \sum_{j\in V}\xi_{t,c,j}\mu_{t,c,j} - \lambda_t + \phi_{t,c} &= 0,
    \label{eq:ec-config-stationarity-z}\\
    w_{t,c,j}
    - \sum_{\tau:\,\PriceLevel{\tau}\le w_{t,c,j}} \Delta w_\tau f_k(H_{j,\tau}^{(k)})
    - \theta_j - \mu_{t,c,j} + \psi_{t,c,j} &= 0.
    \label{eq:ec-config-stationarity-x}
\end{align}
Similarly, stationarity with respect to \(y_{j,\tau}\) gives
\begin{equation}
    \PriceLevel{\tau}
    - \sum_{\tau'\le \tau} \Delta w_{\tau'} f_k(H_{j,\tau'}^{(k)})
    - \theta_j - \pi_{j,\tau} + \iota_{j,\tau} = 0.
    \label{eq:ec-config-stationarity-y}
\end{equation}
It remains to show that \(\theta_j=0\) for every advertiser \(j\) and stage \(k\le K-1\). Suppose instead that \(\theta_j>0\). Then complementary slackness for the capacity constraint forces advertiser \(j\) to be exactly full after stage \(k\). Choose the smallest price index \(\tau^\star\) that appears with positive retained mass after the stage, in the following sense: either \(y_{j,\tau^\star}^*>0\), or there exists some \((t,c)\) with \(x_{t,c,j}^*>0\) and \(w_{t,c,j}=\PriceLevel{\tau^\star}\). By minimality of \(\tau^\star\), every retained unit of mass for advertiser \(j\) has price at least \(\PriceLevel{\tau^\star}\). Since the total retained mass equals one, we must therefore have \(H_{j,\tau'}^{(k)}=1\) for every \(\tau'\le \tau^\star\). Because \(f_k(1)=1\),
\[
    \sum_{\tau'\le \tau^\star} \Delta w_{\tau'} f_k(H_{j,\tau'}^{(k)})
    = \sum_{\tau'\le \tau^\star} \Delta w_{\tau'}
    = \PriceLevel{\tau^\star}.
\]
We now distinguish two cases.

If \(y_{j,\tau^\star}^*>0\), then \(\iota_{j,\tau^\star}=0\) by \eqref{eq:ec-config-cs-pi-iota}. Plugging the previous display into \Eqref{eq:ec-config-stationarity-y} with \(\tau=\tau^\star\) gives
\[
    0 = \PriceLevel{\tau^\star}-\sum_{\tau'\le \tau^\star}\Delta w_{\tau'}f_k(H_{j,\tau'}^{(k)})-\theta_j-\pi_{j,\tau^\star}
      = -\theta_j-\pi_{j,\tau^\star},
\]
a contradiction.

If instead \(y_{j,\tau^\star}^*=0\), then by definition of \(\tau^\star\) there exists some pair \((t,c)\) with \(x_{t,c,j}^*>0\) and \(w_{t,c,j}=\PriceLevel{\tau^\star}\). In that case \(\psi_{t,c,j}=0\) by \Eqref{eq:ec-config-cs-psi-phi}, and \Eqref{eq:ec-config-stationarity-x} yields
\[
    0 = \PriceLevel{\tau^\star}-\sum_{\tau'\le \tau^\star}\Delta w_{\tau'}f_k(H_{j,\tau'}^{(k)})-\theta_j-\mu_{t,c,j}
      = -\theta_j-\mu_{t,c,j},
\]
again a contradiction.

Thus \(\theta_j=0\) for every advertiser \(j\) and every stage \(k\le K-1\). Substituting this fact into \Eqref{eq:ec-config-stationarity-z} and \Eqref{eq:ec-config-stationarity-x} yields exactly the identities stated in \Lemref{lem:config-kkt-main}.\qed
\end{proof}

\begin{proof}{Proof of \Lemref{lem:config-dispose}.}
Assume that \(\ConfigState{j}{\tau}{k-1}-y_{j,\tau}^*>0\). Then the carry constraint at \((j,\tau)\) is slack, so
\(\pi_{j,\tau}=0\) by complementary slackness. Since we already know that \(\theta_j=0\),
\Eqref{eq:ec-config-stationarity-y} gives
\[
    \PriceLevel{\tau}
    - \sum_{\tau'\le \tau} \Delta w_{\tau'} f_k(H_{j,\tau'}^{(k)})
    + \iota_{j,\tau} = 0.
\]
Hence
\[
    \sum_{\tau'\le \tau} \Delta w_{\tau'} f_k(H_{j,\tau'}^{(k)}) \ge \PriceLevel{\tau}
    = \sum_{\tau'\le \tau} \Delta w_{\tau'}.
\]
Because each term \(f_k(H_{j,\tau'}^{(k)})\) lies in \([0,1]\), equality can hold only if
\(f_k(H_{j,\tau'}^{(k)})=1\) for every \(\tau'\le \tau\). The polynomial \(f_k\) is increasing and attains the value
one only at the point one. Thus \(H_{j,\tau}^{(k)}=1\), as claimed.\qed
\end{proof}

\begin{proof}{Proof of \Corref{cor:config-monotone}.}
Fix \((j,\tau)\). If
\[
    \sum_{\tau'\ge \tau} \bigl(\eta_{j,\tau'}^{(k-1)}-y_{j,\tau'}^*\bigr) = 0,
\]
then no previously retained mass at levels \(\tau,\tau+1,\ldots,T\) is discarded, so
\[
    H_{j,\tau}^{(k)}
    = \sum_{\substack{t\in U_k,\,c\in\mathcal C:\ w_{t,c,j}\ge \PriceLevel{\tau}}} x_{t,c,j}^*
      + \sum_{\tau'\ge \tau} y_{j,\tau'}^*
    \ge \sum_{\tau'\ge \tau} \eta_{j,\tau'}^{(k-1)}.
\]
This proves part (i), and part (ii) is immediate because both sides are zero.

Otherwise,
\[
    \sum_{\tau'\ge \tau} \bigl(\eta_{j,\tau'}^{(k-1)}-y_{j,\tau'}^*\bigr) > 0.
\]
Then there exists some \(\tau'\ge \tau\) with \(\eta_{j,\tau'}^{(k-1)}-y_{j,\tau'}^*>0\). By
\Lemref{lem:config-dispose}, \(H_{j,\tau'}^{(k)}=1\). Since cumulative states are nonincreasing in the threshold,
\(H_{j,\tau}^{(k)}\ge H_{j,\tau'}^{(k)}=1\), and hence
\[
    H_{j,\tau}^{(k)}=1\ge \sum_{\tau'\ge \tau} \eta_{j,\tau'}^{(k-1)},
\]
which again proves part (i). In this case \(f_k(H_{j,\tau}^{(k)})=f_k(1)=1\), so part (ii) follows immediately as
well.\qed
\end{proof}

For the last stage we need one additional observation about the optimal LP dual.

\begin{lemma}
\label{lem:config-last-stage}
Fix an optimal extreme-point dual solution of the stage-\(K\) linear program. Without loss of generality,
\(\bar\theta_j\in\{\PriceLevel{0},\ldots,\PriceLevel{T}\}\) for every advertiser \(j\). Let
\(\bar\theta_j = \PriceLevel{\tau_j^\star}\). Then, for every \(\tau\in\PriceUniverse\),
\begin{align}
    \tau \le \tau_j^\star &\implies \eta_{j,\tau}^{(K-1)}\bar\pi_{j,\tau} = 0,
    \label{eq:ec-config-laststage-low}\\
    \tau > \tau_j^\star &\implies
    \eta_{j,\tau}^{(K-1)}\bigl(\PriceLevel{\tau}-\bar\pi_{j,\tau}\bigr)
    = \eta_{j,\tau}^{(K-1)}\bar\theta_j.
    \label{eq:ec-config-laststage-high}
\end{align}
\end{lemma}

\begin{proof}{Proof.}
At stage \(K\), the problem is linear, and the dual constraints involving \((j,\tau)\) are
\[
    \bar\theta_j + \bar\pi_{j,\tau} \ge \PriceLevel{\tau},
    \qquad \bar\pi_{j,\tau}\ge 0.
\]
Because the dual objective contains the term \(\eta_{j,\tau}^{(K-1)}\bar\pi_{j,\tau}\), optimality sets
\(\bar\pi_{j,\tau}\) to its smallest feasible value whenever \(\eta_{j,\tau}^{(K-1)}>0\). If
\(\tau\le \tau_j^\star\), that minimum is zero, giving \Eqref{eq:ec-config-laststage-low}. If
\(\tau>\tau_j^\star\), that minimum is \(\PriceLevel{\tau}-\bar\theta_j\), giving
\Eqref{eq:ec-config-laststage-high}. If \(\eta_{j,\tau}^{(K-1)}=0\), both displays are trivial.
\qed
\end{proof}

\subsection{Proof of \Thmref{thm:config-main}}

\begin{proof}{Proof of \Thmref{thm:config-main}.}
For each stage \(k\le K-1\), let \((x^{(k)},y^{(k)},z^{(k)})\) be the optimizer of the stage program and let
\(\lambda^{(k)},\mu^{(k)},\psi^{(k)},\phi^{(k)}\) be the multipliers supplied by \Lemref{lem:config-kkt-main}. For the
last stage \(K\), let \((\bar\lambda,\bar\theta,\bar\mu,\bar\pi)\) be an optimal extreme-point dual solution of the
residual LP, chosen as in \Lemref{lem:config-last-stage}.

We define a candidate dual solution for the offline LP \Eqref{eq:offline-config-dual} as follows.
\begin{itemize}
    \item For every user \(t\in U_k\) with \(k\le K-1\), set
    \(\widehat\alpha_t := \lambda_t^{(k)}\).
    \item For every triple \((t,c,j)\) with \(t\in U_k\) and \(k\le K-1\), set
    \(\widehat\gamma_{t,c,j} := \mu_{t,c,j}^{(k)}\).
    \item For every advertiser \(j\), define the stage-\(k\) advertiser-dual increment
    \begin{equation}
        \Delta\widehat\beta_j^{(k)}
        := \sum_{t\in U_k}\sum_{c\in\mathcal C} x_{t,c,j}^{(k)}
           \sum_{\tau:\,\PriceLevel{\tau}\le w_{t,c,j}}
           \Delta w_\tau f_k(H_{j,\tau}^{(k)})
           - \sum_{\tau\in\PriceUniverse}
             \PriceLevel{\tau}\bigl(\eta_{j,\tau}^{(k-1)}-y_{j,\tau}^{(k)}\bigr)
        \qquad (k\le K-1).
        \label{eq:ec-config-beta-early}
    \end{equation}
\end{itemize}
At stage \(K\), set
\[
    \widehat\alpha_t := \bar\lambda_t,
    \qquad
    \widehat\gamma_{t,c,j} := \bar\mu_{t,c,j}
    \qquad (t\in U_K),
\]
and
\begin{equation}
    \Delta\widehat\beta_j^{(K)}
    := \bar\theta_j
       - \sum_{\tau\in\PriceUniverse}
         \eta_{j,\tau}^{(K-1)}\bigl(\PriceLevel{\tau}-\bar\pi_{j,\tau}\bigr).
    \label{eq:ec-config-beta-last}
\end{equation}
Finally define \(\widehat\beta_j := \sum_{k=1}^K \Delta\widehat\beta_j^{(k)}\).

\paragraph{Objective equality for stages \(k\le K-1\).}
Fix such a stage. By \Eqref{eq:ec-config-cs-lambda} and \Eqref{eq:ec-config-stationarity-z},
\begin{align*}
    \sum_{t\in U_k} \widehat\alpha_t
    &= \sum_{t\in U_k} \lambda_t^{(k)} \sum_{c\in\mathcal C} z_{t,c}^{(k)} \\
    &= \sum_{t\in U_k}\sum_{c\in\mathcal C}\sum_{j\in V}
       \xi_{t,c,j} z_{t,c}^{(k)} \mu_{t,c,j}^{(k)}
       - \sum_{t\in U_k}\sum_{c\in\mathcal C} \phi_{t,c}^{(k)} z_{t,c}^{(k)} \\
    &= \sum_{t\in U_k}\sum_{c\in\mathcal C}\sum_{j\in V}
       x_{t,c,j}^{(k)} \mu_{t,c,j}^{(k)},
\end{align*}
where the last equality uses \Eqref{eq:ec-config-cs-mu} and \Eqref{eq:ec-config-cs-psi-phi}. Now apply
\Eqref{eq:ec-config-stationarity-x} with \(\theta_j=0\) and \(\psi_{t,c,j}^{(k)}x_{t,c,j}^{(k)}=0\):
\[
    x_{t,c,j}^{(k)}\mu_{t,c,j}^{(k)}
    = x_{t,c,j}^{(k)}w_{t,c,j}
      - x_{t,c,j}^{(k)}
        \sum_{\tau:\,\PriceLevel{\tau}\le w_{t,c,j}}
        \Delta w_\tau f_k(H_{j,\tau}^{(k)}).
\]
Summing over all \((t,c,j)\) and then adding \Eqref{eq:ec-config-beta-early} gives
\[
    \sum_{t\in U_k}\widehat\alpha_t + \sum_{j\in V}\Delta\widehat\beta_j^{(k)}
    = \sum_{t\in U_k}\sum_{c\in\mathcal C}\sum_{j\in V} w_{t,c,j}x_{t,c,j}^{(k)}
      - \sum_{j\in V}\sum_{\tau\in\PriceUniverse}
        \PriceLevel{\tau}\bigl(\eta_{j,\tau}^{(k-1)}-y_{j,\tau}^{(k)}\bigr),
\]
which is exactly the stage-\(k\) contribution to the primal objective.

\paragraph{Objective equality at stage \(K\).}
The stage-\(K\) problem is linear. Its dual objective is
\[
    \sum_{t\in U_K}\bar\lambda_t + \sum_{j\in V}\bar\theta_j
    + \sum_{j\in V}\sum_{\tau\in\PriceUniverse} \eta_{j,\tau}^{(K-1)}\bar\pi_{j,\tau}.
\]
By strong duality, this equals the stage-\(K\) primal objective, and by \Eqref{eq:ec-config-beta-last}
\[
    \sum_{t\in U_K}\widehat\alpha_t + \sum_{j\in V}\Delta\widehat\beta_j^{(K)}
    = \sum_{t\in U_K}\bar\lambda_t + \sum_{j\in V}\bar\theta_j
      + \sum_{j\in V}\sum_{\tau\in\PriceUniverse}\eta_{j,\tau}^{(K-1)}\bar\pi_{j,\tau}
      - \sum_{j\in V}\sum_{\tau\in\PriceUniverse}\eta_{j,\tau}^{(K-1)}\PriceLevel{\tau},
\]
which is again exactly the stage-\(K\) primal contribution. Therefore
\[
    \sum_{t\in U}\widehat\alpha_t + \sum_{j\in V}\widehat\beta_j = \ALG.
\]

\paragraph{Nonnegativity of the advertiser dual increments.}
For \(k\le K-1\), reorder the summation in \Eqref{eq:ec-config-beta-early}:
\begin{align*}
    \Delta\widehat\beta_j^{(k)}
    &= \sum_{\tau\in\PriceUniverse}
       \Delta w_\tau
       \left(
           \sum_{\substack{t\in U_k,\,c\in\mathcal C:\ w_{t,c,j}\ge \PriceLevel{\tau}}} x_{t,c,j}^{(k)}
       \right)
       f_k(H_{j,\tau}^{(k)})
       - \sum_{\tau\in\PriceUniverse}
         \PriceLevel{\tau}\bigl(\eta_{j,\tau}^{(k-1)}-y_{j,\tau}^{(k)}\bigr) \\
    &= \sum_{\tau\in\PriceUniverse}
       \Delta w_\tau
       \Bigl(H_{j,\tau}^{(k)} - \sum_{\tau'\ge \tau} y_{j,\tau'}^{(k)}\Bigr)
       f_k(H_{j,\tau}^{(k)})
       - \sum_{\tau\in\PriceUniverse}
         \Delta w_\tau
         \sum_{\tau'\ge \tau}\bigl(\eta_{j,\tau'}^{(k-1)}-y_{j,\tau'}^{(k)}\bigr).
\end{align*}
By \Corref{cor:config-monotone}(ii),
\[
    \sum_{\tau'\ge \tau}\bigl(\eta_{j,\tau'}^{(k-1)}-y_{j,\tau'}^{(k)}\bigr)
    =
    \sum_{\tau'\ge \tau}\bigl(\eta_{j,\tau'}^{(k-1)}-y_{j,\tau'}^{(k)}\bigr)
    f_k(H_{j,\tau}^{(k)}).
\]
Substituting and simplifying,
\[
    \Delta\widehat\beta_j^{(k)}
    = \sum_{\tau\in\PriceUniverse}
      \Delta w_\tau
      \bigl(H_{j,\tau}^{(k)}-H_{j,\tau}^{(k-1)}\bigr)
      f_k(H_{j,\tau}^{(k)}) \ge 0,
\]
where the last inequality uses \Corref{cor:config-monotone}(i). For stage \(K\), dual feasibility implies
\(\PriceLevel{\tau}-\bar\pi_{j,\tau}\le \bar\theta_j\), and hence
\[
    \Delta\widehat\beta_j^{(K)}
    \ge \bar\theta_j - \bar\theta_j \sum_{\tau\in\PriceUniverse}\eta_{j,\tau}^{(K-1)} \ge 0.
\]
Thus every advertiser-dual increment is nonnegative.

\paragraph{Feasibility of the \(\alpha\)-constraints.}
For \(k\le K-1\), \Eqref{eq:ec-config-stationarity-z} gives
\[
    \widehat\alpha_t = \lambda_t^{(k)}
    = \sum_{j\in V}\xi_{t,c,j}\mu_{t,c,j}^{(k)} + \phi_{t,c}^{(k)}
    \ge \sum_{j\in V}\xi_{t,c,j}\widehat\gamma_{t,c,j}
    \qquad \forall c\in\mathcal C.
\]
For stage \(K\), the same inequality is one of the LP dual constraints. So the \(\alpha\)-constraints of
\Eqref{eq:offline-config-dual} are satisfied.

\paragraph{Approximate feasibility for stages \(k\le K-1\).}
Fix a triple \((t,c,j)\) with \(t\in U_k\), and let \(\tau(t,c,j)\) denote the price index of \(w_{t,c,j}\). Since
\(\psi_{t,c,j}^{(k)}\ge 0\), \Eqref{eq:ec-config-stationarity-x} gives
\[
    \widehat\gamma_{t,c,j}
    = \mu_{t,c,j}^{(k)}
    \ge w_{t,c,j}
       - \sum_{\tau:\,\PriceLevel{\tau}\le w_{t,c,j}}
         \Delta w_\tau f_k(H_{j,\tau}^{(k)}).
\]
Combining this with the expression for \(\Delta\widehat\beta_j^{(s)}\) derived above,
\begin{align*}
    \widehat\gamma_{t,c,j} + \sum_{s=1}^k \Delta\widehat\beta_j^{(s)}
    &\ge \sum_{\tau:\,\PriceLevel{\tau}\le w_{t,c,j}}
       \Delta w_\tau
       \Biggl[
           1 - f_k(H_{j,\tau}^{(k)})
           + \sum_{s=1}^k
             \bigl(H_{j,\tau}^{(s)}-H_{j,\tau}^{(s-1)}\bigr)
             f_s(H_{j,\tau}^{(s)})
       \Biggr].
\end{align*}
For each threshold \(\tau\) in the sum, apply \Lemref{lem:batch-recursion} to the nondecreasing sequence
\[
    0=H_{j,\tau}^{(0)} \le H_{j,\tau}^{(1)} \le \cdots \le H_{j,\tau}^{(k)}
\]
with
\(x = H_{j,\tau}^{(k)}-H_{j,\tau}^{(k-1)}\). This yields
\[
    1 - f_k(H_{j,\tau}^{(k)})
    + \sum_{s=1}^k \bigl(H_{j,\tau}^{(s)}-H_{j,\tau}^{(s-1)}\bigr)f_s(H_{j,\tau}^{(s)})
    \ge \GammaFunc{K}.
\]
Hence
\[
    \widehat\gamma_{t,c,j} + \sum_{s=1}^k \Delta\widehat\beta_j^{(s)}
    \ge \GammaFunc{K}
       \sum_{\tau:\,\PriceLevel{\tau}\le w_{t,c,j}} \Delta w_\tau
    = \GammaFunc{K} w_{t,c,j}.
\]
Since all later increments of \(\widehat\beta_j\) are nonnegative, the full approximate constraint holds.

\paragraph{Approximate feasibility at stage \(K\).}
Fix a triple \((t,c,j)\) with \(t\in U_K\), and let \(\tau^\dagger:=\tau(t,c,j)\). Also write
\(\bar\theta_j = \PriceLevel{\tau_j^\star}\) as in \Lemref{lem:config-last-stage}. We distinguish two cases.

\emph{Case 1: \(\tau^\dagger < \tau_j^\star\), i.e., \(w_{t,c,j}<\bar\theta_j\).}
By \Lemref{lem:config-last-stage},
\[
    \Delta\widehat\beta_j^{(K)}
    = \bar\theta_j
      - \sum_{\tau>\tau_j^\star} \eta_{j,\tau}^{(K-1)}\bar\theta_j
      - \sum_{\tau\le \tau_j^\star} \eta_{j,\tau}^{(K-1)}\PriceLevel{\tau}.
\]
Split the second sum at \(\tau^\dagger\) and use \(\PriceLevel{\tau}\le \bar\theta_j\) for
\(\tau^\dagger\le \tau\le \tau_j^\star\):
\[
    \Delta\widehat\beta_j^{(K)}
    \ge \bar\theta_j\Bigl(1-\sum_{\tau\ge \tau^\dagger}\eta_{j,\tau}^{(K-1)}\Bigr)
       - \sum_{\tau<\tau^\dagger}\eta_{j,\tau}^{(K-1)}\PriceLevel{\tau}.
\]
Since \(\bar\mu_{t,c,j}+\bar\theta_j\ge w_{t,c,j}\) and \(\bar\mu_{t,c,j}\ge 0\),
\begin{align*}
    \bar\mu_{t,c,j} + \Delta\widehat\beta_j^{(K)}
    &\ge \Bigl(1-\sum_{\tau\ge \tau^\dagger}\eta_{j,\tau}^{(K-1)}\Bigr)
       (\bar\mu_{t,c,j}+\bar\theta_j)
       - \sum_{\tau<\tau^\dagger}\eta_{j,\tau}^{(K-1)}\PriceLevel{\tau} \\
    &\ge \Bigl(1-\sum_{\tau\ge \tau^\dagger}\eta_{j,\tau}^{(K-1)}\Bigr)w_{t,c,j}
       - \sum_{\tau<\tau^\dagger}\eta_{j,\tau}^{(K-1)}\PriceLevel{\tau}.
\end{align*}
The right-hand side can be rewritten as
\[
    \sum_{\tau\le \tau^\dagger}
    \Delta w_\tau\bigl(1-H_{j,\tau}^{(K-1)}\bigr),
\]
which equals
\[
    \sum_{\tau\le \tau^\dagger}
    \Delta w_\tau\bigl(1-f_{K-1}(H_{j,\tau}^{(K-1)})\bigr)
\]
because \(f_{K-1}(x)=x\).

\emph{Case 2: \(\tau^\dagger \ge \tau_j^\star\), i.e., \(w_{t,c,j}\ge \bar\theta_j\).}
Again by \Lemref{lem:config-last-stage},
\[
    \Delta\widehat\beta_j^{(K)}
    = \bar\theta_j
      - \sum_{\tau>\tau_j^\star} \eta_{j,\tau}^{(K-1)}\bar\theta_j
      - \sum_{\tau\le \tau_j^\star} \eta_{j,\tau}^{(K-1)}\PriceLevel{\tau}.
\]
Now use \(\bar\theta_j\le \PriceLevel{\tau}\) for
\(\tau_j^\star<\tau\le \tau^\dagger\) to obtain
\[
    \Delta\widehat\beta_j^{(K)}
    \ge \bar\theta_j
      - \sum_{\tau>\tau^\dagger} \eta_{j,\tau}^{(K-1)}\bar\theta_j
      - \sum_{\tau<\tau^\dagger} \eta_{j,\tau}^{(K-1)}\PriceLevel{\tau}.
\]
Using again \(\bar\mu_{t,c,j}+\bar\theta_j\ge w_{t,c,j}\),
\begin{align*}
    \bar\mu_{t,c,j} + \Delta\widehat\beta_j^{(K)}
    &\ge \Bigl(1-\sum_{\tau\ge \tau^\dagger}\eta_{j,\tau}^{(K-1)}\Bigr)
       (\bar\mu_{t,c,j}+\bar\theta_j)
       - \sum_{\tau<\tau^\dagger}\eta_{j,\tau}^{(K-1)}\PriceLevel{\tau} \\
    &\ge \Bigl(1-\sum_{\tau\ge \tau^\dagger}\eta_{j,\tau}^{(K-1)}\Bigr)w_{t,c,j}
       - \sum_{\tau<\tau^\dagger}\eta_{j,\tau}^{(K-1)}\PriceLevel{\tau}.
\end{align*}
The same telescoping rewrite as in Case 1 shows that the right-hand side equals
\[
    \sum_{\tau\le \tau^\dagger}
    \Delta w_\tau\bigl(1-f_{K-1}(H_{j,\tau}^{(K-1)})\bigr).
\]
Thus, in both cases,
\[
    \widehat\gamma_{t,c,j} + \Delta\widehat\beta_j^{(K)}
    \ge \sum_{\tau\le \tau(t,c,j)}
       \Delta w_\tau\bigl(1-f_{K-1}(H_{j,\tau}^{(K-1)})\bigr).
\]
Adding the earlier advertiser-dual increments and applying the same threshold-by-threshold recursion as above,
we obtain
\[
    \widehat\gamma_{t,c,j} + \widehat\beta_j \ge \GammaFunc{K} w_{t,c,j}.
\]
This proves approximate dual feasibility for every triple \((t,c,j)\). By weak duality for
\Eqref{eq:offline-config-dual}, the constructed dual satisfies
\(\ALG\ge \GammaFunc{K}\OPT\).\qed
\end{proof}